\documentclass{amsart}
\usepackage{amsthm}
\usepackage{amsmath}
\usepackage{amssymb}
\usepackage{mathrsfs}
\usepackage[active]{srcltx}
\usepackage{verbatim}

\usepackage[fleqn,tbtags]{mathtools}
\mathtoolsset{showonlyrefs,showmanualtags}

\allowdisplaybreaks

\DeclareRobustCommand{\SkipTocEntry}[5]{} 

\let\mathcal \undefined
\def\mathcal{\mathscr}
\renewcommand{\le}{\leq}
\renewcommand{\ge}{\geq}
\newcommand{\eps}{\varepsilon}

\theoremstyle{plain}
\newtheorem{theorem}{Theorem} [section]
\theoremstyle{remark}
\newtheorem{remark}[theorem]{Remark}
\newtheorem{example}[theorem]{Example}

\theoremstyle{plain}
\newtheorem{corollary}[theorem]{Corollary}
\newtheorem{lemma}[theorem]{Lemma}
\newtheorem{proposition}[theorem]{Proposition}
\newtheorem{definition}[theorem]{Definition}

\numberwithin{equation}{section}

\font\Bbbten=msbm10
\font\Bbbseven=msbm7
\font\Bbbfive=msbm5
\newfam\Bbbfam

  \textfont\Bbbfam=\Bbbten
  \scriptfont\Bbbfam=\Bbbseven
  \scriptscriptfont\Bbbfam=\Bbbfive
\def\R{{\mathbb R}}
\def\N{{\mathbb N}}

\def\E{{\mathbb E}}
\def\P{{\mathbb P}}

\newcommand{\maxsym}{\vee}
\newcommand{\minsym}{\wedge}

\newcommand{\F}{\mathscr{F}}

\newcommand{\e}{\varepsilon}
\newcommand{\g}{\gamma}

\renewcommand{\O}{\Omega}
\renewcommand{\l}{\lambda}
\newcommand{\n}{\Vert}
\newcommand{\embed}{\hookrightarrow}

\newcommand{\lb}{\langle}
\newcommand{\rb}{\rangle}
\renewcommand{\a}{\alpha}
\renewcommand{\b}{\beta}
\renewcommand{\d}{\delta}

\newcommand{\calF}{\mathcal{F}}
\newcommand{\calL}{\mathcal{L}}

\newcommand{\calR}{\mathcal{R}}

\renewcommand{\H}{\mathscr{H}}

\newcommand{\inv}[1]{\frac{1}{#1}}
\newcommand{\tinv}[1]{\tfrac{1}{#1}}
\newcommand{\un}[1]{\underline{#1}}
\newcommand{\ov}[1]{\overline{#1}}

\newcommand{\beq}{\begin{equation}}
\newcommand{\eeq}{\end{equation}}
\newcommand{\bal}{\begin{aligned}}
\newcommand{\eal}{\end{aligned}}
\newcommand{\ben}{\begin{enumerate}}
\newcommand{\een}{\end{enumerate}}
\newcommand{\bit}{\begin{itemize}}
\newcommand{\eit}{\end{itemize}}

\newcommand{\Umod}{U^{(n)}}
\newcommand{\Ucla}{\tilde U^{(n)}}

\newcommand{\Vapc}{V^{\alpha,p}_{\rm c}}

\newcommand{\Vpinf}[1]{V^{#1,p}_{\infty}}
\newcommand{\Winfp}[1]{\mathcal{V}^{#1,p}_{\infty}}

\newcommand{\munj}{\mu_n^{*j}}

\newcommand{\En}[1]{E(#1)}
\newcommand{\Enj}{E(t_{j}^{(n)})}

\newcommand{\MI}{{\bf (M2)}}
\newcommand{\MII}{{\bf (M1)}}
\newcommand{\MIII}{{\bf (M3)}}
\newcommand{\MA}{{\bf (A)}}
\newcommand{\MF}{{\bf (F)}}
\newcommand{\MG}{{\bf (G)}}
\newcommand{\MFII}{{\bf (F$'$)}}
\newcommand{\MGII}{{\bf (G$'$)}}
\newcommand{\MFloc}{{\bf (F$_\textrm{loc}$)}}
\newcommand{\MGloc}{{\bf (G$_\textrm{loc}$)}}
\DeclareMathOperator*{\esssup}{ess\,sup}

\begin{document}

\title[H\"older convergence of the Euler scheme for SPDEs]
{Pathwise H\"older convergence of the implicit Euler scheme for semi-linear
{\large SPDE}s with multiplicative noise}
\date{\today}

\author{Sonja Cox}
\author{Jan van Neerven}
\address{Delft Institute of Applied Mathematics\\
Delft University of Technology \\ P.O. Box 5031\\ 2600 GA Delft\\The
Netherlands}
\email{S.G.Cox@tudelft.nl, J.M.A.M.vanNeerven@tudelft.nl}

\thanks{Part of this research was conducted while the
first-named author was visiting the University of New South Wales in Sydney, Australia. She would
like to
thank the university, and Ben Goldys in particular, for their hospitality. 
The second-named author gratefully acknowledges support by VICI subsidy
639.033.604
of the Netherlands Organization for Scientific Research (NWO).
The authors thank Markus Haase and Mark Veraar for their helpful comments.}

\begin{abstract} In this article we prove pathwise H\"older convergence with
optimal rates of the implicit Euler
scheme for the abstract stochastic Cauchy problem
\begin{equation}
\left\{ \begin{aligned} dU(t) & = AU(t)\,dt + F(t,U(t))\,dt +
G(t,U(t))\,dW_H(t);\quad t\in [0,T],\\
U(0)&=x_0. \end{aligned}\right.
\end{equation}
Here $A$ is the generator of an analytic $C_0$-semigroup 
on a \textsc{umd} Banach space $X$,
$W_H$ is a cylindrical Brownian motion in a Hilbert space $H$, and the
functions 
$F:[0,T]\times X\rightarrow X_{\theta_F}$ and 
$G:[0,T]\times X\rightarrow \calL(H,X_{\theta_G})$ 
satisfy appropriate (local) Lipschitz conditions.
The results are applied to a class of 
second order parabolic SPDEs driven by multiplicative
space-time white noise. 
\end{abstract}

\maketitle


\section{Introduction}\label{sec:intro}
In this article we prove pathwise H\"older convergence with optimal rates for
various numerical 
schemes, including the implicit Euler
scheme and the so-called splitting scheme,
associated with parabolic stochastic partial differential equations (SPDE)
driven by multiplicative Gaussian noise. Such 
equations can be written as abstract Cauchy problems of the form
\begin{equation}\label{SDE}
\left\{ 
\begin{aligned} dU(t) & = AU(t)\,dt +F(t,U(t))\,dt +
G(t,U(t))\,dW_H(t), \quad t\in [0,T];\\ U(0)&=x_0,
\end{aligned}\right.
\end{equation}
where $A$ is the generator of an analytic $C_0$-semigroup on a \textsc{umd}
Banach 
space $X$ and $W_H$ is a cylindrical Brownian motion in a Hilbert space 
$H$ with respect to some given filtration $(\F_t)_{t\in [0,T]}$. 

The functions $F: [0,T] \times X\rightarrow X_{\theta_F}$ and 
$G: [0,T]\times X\rightarrow \calL(H,X_{\theta_G})$ are assumed to satisfy
appropriate (local) 
Lipschitz and linear growth 
conditions that will be specified in Section \ref{ss:setting}. Here,
$X_\theta$ denotes the fractional domain space (if $\theta\ge 0$) 
or extrapolation space (if $\theta
\le 0$) of $A$ of order $\theta$.
The initial value $x_0$ is an $\F_0$-measurable random variable 
taking values in $X_\eta$ for some $\eta\ge 0$.

Under the conditions on the exponents $\theta_F$ and $\theta_G$
stated in Theorem \ref{t:euler_concrete_intro}, the problem 
\eqref{SDE} admits a unique mild solution in a suitable Banach space of `weighted' 
stochastically integrable $X$-valued processes (see Subsection 
\ref{subsec:stoch_evol_eq} for the details).
Our strategy is to first prove convergence of the various approximation schemes
in this space for globally Lipschitz coefficients $F$ and $G$. These results imply uniform
convergence of all $p^{\textrm{th}}$ moments. A Kolmogorov argument then allows us to 
deduce convergence in certain H\"older norms (Theorems
\ref{t:euler_concrete_intro} and \ref{t:split_intro}). 
A standard Borel-Cantelli argument then gives almost sure H\"older convergence
of the paths, and a localization 
argument allows us to extend the results to locally Lipschitz coefficients $F$ and $G$.

To the best of our knowledge, this is the first article proving convergence 
with respect to H\"older norms. A detailed comparison with known results in the
literature is given 
below (see page \pageref{subsec:related_work}).
 
At the end of the paper we apply our abstract results 
to a class of second order parabolic SPDEs driven by multiplicative space-time
white noise. 
Further examples of SPDEs that fit into this framework can be found in
\cite[Section 10]{NVW08}. 
In these applications, typically one takes
$X = L^p(D)$ (or a fractional domain space derived from it) with $D\subseteq \R^d$
an open domain; the choice 
$H = L^2(D)$ then corresponds to the case of space-time white noise on $D$.

\addtocontents{toc}{\SkipTocEntry}\subsection*{The implicit-linear Euler scheme}
The main result of this article gives optimal pathwise H\"older convergence
rates for the 
implicit-linear Euler scheme associated with \eqref{SDE}. The scheme is defined as follows.
Fixing a finite time horizon $0<T<\infty$ and an integer $n\in \N =
\{1,2,3,\dots\}$, we set 
$V^{(n)}_0 := x_0$
 and, for $j=1,\dots,n$,  define the random variables $V_j^{(n)}$ implicitly by
the identity
\begin{align*}
V_j^{(n)} & =  V_{j-1}^{(n)} + \tfrac{T}{n} \big[A V_j^{(n)} +
F\big(t_{j-1}^{(n)},V_{j-1}^{(n)}\big)\big] +
G\big(t_{j-1}^{(n)},V_{j-1}^{(n)}\big)\Delta W_j^{(n)}.
\end{align*}
Here $$t_j^{(n)} =\tfrac{jT}{n} \  \hbox{ and, formally, } \ 
\Delta W_j^{(n)} = W_H(t_j^{(n)}) - W_H(t_{j-1}^{(n)}).$$
The rigorous interpretation of the term
$G\big(t_{j-1}^{(n)},V_{j-1}^{(n)}\big)\Delta W_j^{(n)}$ is explained in Section
\ref{sec:absEuler}.
Note that this scheme is implicit only in its linear part. As a consequence of
this, 
and noting that for large enough $n\in\N$ we have $\frac{n}{T}\in\varrho(A)$,
the resolvent set of $A$, 
this identity may be rewritten in
the explicit format
\begin{equation}\label{EulerDefIntro}
\begin{aligned}
V_j^{(n)} = (1-\tfrac{T}{n}A)^{-1}
\big[V_{j-1}^{(n)}+ \tfrac{T}{n}F\big(t_{j-1}^{(n)},V_{j-1}^{(n)}\big) +
G\big(t_{j-1}^{(n)},V_{j-1}^{(n)}\big)\Delta W_j^{(n)}\big].
\end{aligned}
\end{equation}\par

For a sequence $x=(x_j)_{j=0}^n$ in $X$ (we think of
the $x_j$ as
the values $f(t_j^{(n)})$ of some function $f:[0,T]\to X$) and $0\le \gamma\le
1$ we define the discrete H\"older norm $\n \cdot \n_{c^{(n)}_{\gamma}}$ as follows:
$$\n x\n_{c_\gamma^{(n)}([0,T];X)} := 
\sup_{0\le j\le n} \n x_j\n_X + 
\sup_{0\le i < j\le n} \frac{\n x_j - x_i\n_X}{|t_j^{(n)} -
t_i^{(n)}|^{\gamma}}.
$$\par
Set $u = (U(t_j^{(n)}))_{j=0}^n$, where $(U(t))_{t\in [0,T]}$
is the mild solution of \eqref{SDE}, and $v^{(n)} = (V_j^{(n)})_{j=0}^n$. 
Under appropriate Lipschitz and linear growth
conditions on the  functions $F: [0,T] \times X\rightarrow X_{\theta_F}$ and 
$G: [0,T]\times X\rightarrow \calL(H,X_{\theta_G})$ (conditions \MF{}, \MG{} in Section
\ref{ss:setting} and conditions \MFII{}, \MGII{} in Section \ref{sec:absEuler}), 
the following result is obtained in Section \ref{sec:pathwise}.

\begin{theorem}[H\"older convergence of the Euler
scheme]\label{t:euler_concrete_intro} 
Let $X$ be a \textsc{umd} Banach space with Pisier's property $(\alpha)$, and 
let $\tau\in (1,2]$ be
the type of $X$. 
Let $\theta_F>  -1 + (\inv{\tau}-\frac12)$, $\theta_G>-\inv{2}$. Suppose
$p\in (2,\infty)$
and $\gamma,\delta \in [0,\inv{2})$ and $\eta>0$ 
satisfy 
$$\gamma+ \delta+\tfrac1p<  \min\{1 - (\tinv{\tau}-\tfrac12)
+(\theta_F\wedge 0) ,\,\tfrac12+(\theta_G\wedge 0), \,\eta \},$$
and suppose that $x_0\in L^p(\Omega,\calF_0;X_{\eta})$. 
There is a constant $C$, independent of $x_0$, such that for all 
large enough $n\in \N$,
\begin{equation}\label{eq:moments}
\begin{aligned}
\big(\E \n u-v^{(n)} \n_{c_\gamma^{(n)}([0,T];X)}^p\big)^{\inv{p}} &\le C
n^{-\delta}(1+\n
x_0\n_{L^p(\Omega;X_{\eta})}).
\end{aligned}
\end{equation}
\end{theorem} 
Examples of \textsc{umd} 
spaces with property $(\alpha)$ are the Hilbert spaces and the Lebesgue spaces
$L^p(\mu)$ with and $1<p<\infty$ and $\mu$ a $\sigma$-finite measure
(see the introductions of Sections \ref{sec:prelim} and \ref{sec:absEuler} for more details).
\par

By a Borel-Cantelli argument, \eqref{eq:moments} implies that for 
almost all $\omega\in\Omega$ there exists a constant $C$ depending on $\omega$
but independent of $n$, 
such that:
\begin{equation}\label{eq:local} \sup_{j\in \{0,\ldots,n\}} \n
u(\omega)-v^{(n)}(\omega) \n_{c_\gamma^{(n)}([0,T];X)} 
\le C n^{-\d}.
\end{equation}
By a localization argument, we can now obtain almost sure pathwise convergence,
with the same rates,
for the case that $x_0:\Omega\to X_\eta$ is merely $\F_0$-measurable and the
coefficients $F$ and $G$ are
locally Lipschitz continuous with linear growth
(see Section \ref{sec:local}).

\addtocontents{toc}{\SkipTocEntry}\subsection*{The splitting scheme}
Theorem \ref{t:euler_concrete_intro} will be deduced from the corresponding 
convergence result for a splitting scheme, essentially by using the
classical Trotter-Kato formula
\begin{equation}\label{eq:TK} S(t)x = \lim_{n\to\infty} (I - \tfrac{t}{n}A)^{-n}x
\end{equation}
to pass from the resolvent of $A$ to the semigroup $S(t)$ generated by $A$. 

We shall use a scheme which  
is an appropriate modification of the `classical' splitting scheme, which allows us to cover the case of negative fractional indices $\theta_F$ and
$\theta_G$.
It is defined by setting
$U^{(n)}_0(0) := x_0 $ 
and then successively solving, for $j=1,\dots,n$, the problem
\begin{equation}\label{ModSplitIntro}
\left\{
\begin{aligned} 
dU_j^{(n)}(t) & = S(\tfrac{T}{n})\big[F(t,U_j^{(n)}(t))\, dt
+ G(t,U_j^{(n)}(t))\, dW_H(t)\big], \quad t\in [t_{j-1}^{(n)},t_j^{(n)}], \\
 U_j^{(n)}(t_{j-1}^{(n)}) & = S(\tfrac{T}{n})U^{(n)}_{j-1}(t_{j-1}^{(n)}) .
\end{aligned} \right.
\end{equation}
Set $u = (U(t_j^{(n)}))_{j=0}^n$ and $u^{(n)} =
(U^{(n)}_j(t_j^{(n)}))_{j=0}^n$. In Section \ref{sec:pathwise} the following 
theorem is obtained from
Theorem \ref{thm:convV}, assuming that $F$ and $G$ satisfy the Lipschitz and linear growth
conditions \MF{} and \MG{} as stated in Section \ref{ss:setting}.

\begin{theorem}[H\"older convergence of the splitting
scheme]\label{t:split_intro}
Let $X$ be a \textsc{umd} Banach space and let $\tau\in (1,2]$ be the type
of $X$. 
Let $\theta_F>  -1 + (\inv{\tau}-\frac12)$, 
$\theta_G>-\inv{2}$. Suppose $p\in [2,\infty)$
and $\gamma,\delta\in [0,1)$ and $\eta>0$ satisfy 
$$\gamma+ \delta+\tfrac1p < \min\{1 -
(\tinv{\tau}-\tfrac12) + \theta_F ,\, \tfrac12+\theta_G,\,\eta,\,1\},$$
and suppose that $x_0\in L^p(\Omega,\calF_0;X_{\eta})$. 
There is a constant $C$, independent of $x_0$, such that 
for all  $n\in \N$,
\begin{align*}
\begin{aligned}
\big(\E \n u - u^{(n)}\n_{c_\g^{(n)}([0,T];X)}^p\big)^{\inv{p}} &\le C
n^{-\d}(1+\n
x_0\n_{L^p(\Omega;X_{\eta})}).
\end{aligned}
\end{align*}
\end{theorem}

No property $(\alpha)$ assumption is needed for this result. In contrast to our 
result for the Euler scheme, 
the convergence rate does improve as $\theta_F$ and $\theta_G$ increase above $0$.
As a consequence of Theorem \ref{t:split_intro}
it is possible to obtain almost sure H\"older convergence of the paths for the case 
that $F$ and $G$ are locally Lipschitz. However, this requires some tedious extra 
arguments that will be presented in \cite{cox:thesis} 
(see also Remark \ref{r:split_local}).
   
In the special case of fractional indices $\theta_F\ge 0$ and $\theta_G\ge 0$ 
we can compare the above scheme with the `classical'
splitting scheme defined by solving the problems
\begin{align}\label{ClaSplitIntro}
\left\{ \begin{aligned} dU_j^{(n)}(t) & = F(t,U_j^{(n)}(t))\,dt +
G(t,U_j^{(n)}(t))\,dW_H(t), \quad t\in [t_{j-1}^{(n)},t_j^{(n)}],\\
U_j^{(n)}(t_{j-1}^{(n)})&=S(\tfrac{T}{n})U_{j-1}^{(n)}(t_{j-1}^{(n)}).
\end{aligned}\right.
\end{align}
The conditions $\theta_F\ge 0$ and $\theta_G\ge 0$ guarantee 
the existence of unique $X$-valued mild solutions to these problems;
in the modified splitting scheme the additional term $S(\tfrac{T}{n})$ provides
the smoothing needed in the case of negative fractional indices.
We shall prove that for  $\theta_F\ge 0$ and $\theta_G\ge 0$ one obtains the same 
convergence rates as in Theorem \ref{t:split_intro}.

\addtocontents{toc}{\SkipTocEntry}\subsection*{An example: the 1D stochastic
heat equation}

In Section \ref{sec:example} we present a detailed application of our results to
the stochastic heat equation
with multiplicative noise in space dimension one.
We consider the problem 
\begin{equation*}
\left\{ 
\begin{aligned}
\frac{\partial u}{\partial t}(\xi,t) &= a_2(\xi)\frac{\partial^2 u}{\partial
\xi^2}(\xi,t)+ a_1(\xi)\frac{\partial u}{\partial \xi}(\xi,t)
\\
&\quad\quad  +f(t,\xi,u(\xi,t))+g(t,\xi,u(\xi,t))\frac{\partial w}{\partial
t}(\xi,t); &&
 \xi\in (0,1), \ t\in (0,T],\\
u(0,\xi)  & = u_0(\xi); && \xi\in [0,1],
\end{aligned}
\right.
\end{equation*}
with Dirichlet or Neumann boundary conditions. 
Assuming $a_2\in C[0,1]$ to be bounded away from $0$ and $a_1\in C[0,1]$, 
and assuming standard Lipschitz and linear growth assumptions on the nonlinearities
$f$ and $g$, it is shown that if $p>4$, $q>2$, $\a>0$,
$\b\in (\inv{2q},\inv{4})$, and $\g,\,\d\ge 0$ satisfy $$\b+\g+\d+\tfrac1p< \min\{\tfrac14,\a\},$$ then for initial
values $u_0\in L^{p}(\Omega, \F_0;H^{2\a,q}(0,1))$ the modified
splitting scheme $u^{(n)}=(U^{(n)}_j(t_j^{(n)}))_{j=0}^{n}$, defined by \eqref{ModSplitIntro}, 
and the implicit Euler scheme $v^{(n)}=(V_j^{(n)})_{j=0}^{n}$, defined by
\eqref{EulerDefIntro}, satisfy:
\begin{align*}
\begin{aligned}
\big(\E \n u-u^{(n)}
\n_{c_{\gamma}^{(n)}([0,T];H^{2\b,q}(0,1))}^p\big)^{\inv{p}} &\lesssim n^{-\delta}(1+\n
u_0\n_{L^p(\Omega;H^{2\a,q}(0,1))}),\\
\big(\E \n u-v^{(n)}
\n_{c_{\gamma}^{(n)}([0,T];H^{2\b,q}(0,1))}^p\big)^{\inv{p}} &\lesssim n^{-\delta}(1+\n
u_0\n_{L^p(\Omega;H^{2\a,q}(0,1))}).
\end{aligned}
\end{align*}\par
Let us take $\g=0$. By a Borel-Cantelli argument and the Sobolev
embedding theorem, among other things we shall deduce that for large enough $p$ and $q$ and
initial values in $L^{p}(\Omega;H^{\inv{2},q}(0,1))$ we obtain
\begin{align*}
 \sup_{0\le j\le n}\n u(t_j^{(n)})-u_j^{(n)}\n_{C^\lambda[0,1]}
&\lesssim n^{-\delta}(1+\n u_0\n_{L^{\infty}(\Omega;H^{\inv{2},q}(0,1))}), \\
 \sup_{0\le j\le n}\n
u(t_j^{(n)})-v_j^{(n)}\n_{C^\lambda[0,1]}
& \lesssim n^{-\delta}(1+\n u_0\n_{L^{p}(\Omega;H^{\inv{2},q}(0,1))}).
\end{align*} 
whenever $\lambda\ge 0$ and $\d\ge 0$ satisfy $\lambda + 2\d < \tfrac12.$

\addtocontents{toc}{\SkipTocEntry}\subsection*{Concerning the optimality of the
rate}
For end-point estimates (i.e., a weaker type of estimates than the type we
consider, which are pathwise) 
it is proven in \cite{DaGai:00} that the optimal convergence rate of a time
discretization for 
the heat equation in one dimension with additive space-time white noise based on
$n$ equidistant 
time steps is $n^{-\inv{4}}$ (see Remark \ref{r:optimalDG}). In that sense our
results on the 
convergence for the heat equation are optimal.

For `trace class noise' (which corresponds to taking $\theta_G=0$ in our
framework) it is known 
that the critical convergence rate $n^{-\inv{2}}$ is in a sense 
the best possible even in the simpler setting of 
ordinary stochastic differential equations with globally Lipschitz
coefficients. 
To be precise, it is shown in \cite{ClarkCam:80} that there exist examples of 
equations of the form 
$$
\left\{\begin{aligned}
dX(t) & = f(X(t))\,dt + g(X(t))\,dW(t), &&  t\in [0,T],\\
X(0) & = x_0,
\end{aligned}\right.
$$
with $x_0\in \R^d$ and $W$ a Brownian motion in $\R^d$, whose solution $X$
satisfies the endpoint estimate
$$\big(\E |X(T)-\E(X(T)|\mathcal{P}_n)|^2\big)^\frac12 = n^{-\frac12}\sqrt{\tfrac12T}.
$$
Here $\mathcal{P}_n=\sigma\{ W(t_j^{(n)})\, : \, j=1,\ldots,n\}$.  
Thus, if $X$ is approximated by a sequence of processes $X^{(n)}$ whose
definition only depends on 
knowing $\{ W(t_j^{(n)})\, 
: \, j=1,\ldots,n\}$ (such is the case for the implicit Euler scheme), the
convergence rate cannot 
be better than $n^{-\frac12}$. 

In the special case $\theta_G = 0$, Theorems \ref{t:euler_concrete_intro}
and \ref{t:split_intro} can be applied with any $\g,\d\ge 0$ such that $\g+\d<\frac12$,
provided $\theta_F \ge  - \frac1\tau$, 
$x_0$ takes values in $X_\eta$ with $\eta\ge \frac12$, and $p$ is taken large enough. For bounded $\F_0$-measurable initial values $x_0$ in $X_{\frac12}$ and taking $\gamma = 0$, this 
leads to pathwise uniform convergence of order $n^{-\d}$ for arbitrary
$\d\in [0,\frac12)$.

Our methods do not produce the critical convergence rate $n^{-\frac12}$. In the
two examples
in the literature that we know of where this rate is obtained (\cite{Gyo:99,
Kruse:11}), 
the operator $A$ has the property of 
`stochastic maximal regularity' (see \cite{NVW10})
and the underlying space has type $2$, and in neither of
these results the convergence is pathwise.   
In the present work, we do obtain pathwise convergence under 
the weaker assumption that $A$ generates an analytic semigroup,
but in this more general framework we do not expect to attain the critical rate. 
We plan to study the maximally regular case in the type $2$ setting in a 
forthcoming publication.

\addtocontents{toc}{\SkipTocEntry}\subsection*{Related work on pathwise
convergence}\label{subsec:related_work}
The literature on convergence rates for numerical schemes for stochastic
evolution equations and 
SPDEs is extensive. For an overview we refer the reader to the excellent review
paper \cite{JenKloe:09b}. 

To the best of our knowledge, this is the first article proving pathwise
convergence 
with respect to H\"older norms for numerical schemes for stochastic evolution
equations
with (locally) Lipschitz coefficients.
The works
\cite{GreKloe:96,  Gyo:99, GyoKry:03, 
GyoNu:95, Hau:02, Hau:03, Jen:09, Kruse:11,
MillMor:05, Pri:01, Yan:05} consider frameworks which are amenable to a
comparison
with our Theorems \ref{t:euler_concrete_intro} and \ref{t:split_intro}; 
quite likely this list is far from complete. All these papers exclusively 
consider Hilbert spaces $X$, the only exception being \cite{Hau:03} where
$X$ is taken to be of martingale type $2$. In particular, in all these papers
$X$ has type $2$.
Let us also mention the paper \cite{GyoMil:09}, where convergence is proved, for
$p=2$ and $X$ Hilbertian, 
for the implicit Euler scheme under monotonicity assumptions on the operator
$A$. 

Most of these papers cited above give endpoint convergence rates only. The first 
pathwise uniform convergence result of the implicit Euler scheme seems to be due to
Gy\"ongy \cite{Gyo:99}, who
obtains convergence rate $n^{-\inv{8}}$ for the 1D stochastic heat equation with
multiplicative space-time white noise. Pathwise uniform convergence of the
implicit Euler schemes 
(with convergence in probability) has been obtained by Printems \cite{Pri:01}
for the Burgers equation with rate $n^{-\gamma}$ for any $\gamma<\inv{4}$.\par

Pathwise uniform convergence of the splitting scheme, with rate $n^{-1}$ (which
is the rate corresponding to exponents $\theta_F\ge \frac1\tau -\frac12$ and
$\theta_G \ge \frac12$ in our framework, provided we take $\eta\ge 1$), 
has been obtained by Gy\"ongy and Krylov \cite{GyoKry:03} for again a different 
splitting scheme, namely one for which the solution of \eqref{ModSplitIntro} is 
guaranteed by adding (roughly speaking) a term $\eps A U(t)\,dt$ to the right-hand 
side of \eqref{ModSplitIntro}. 

\addtocontents{toc}{\SkipTocEntry}\subsection*{Outline of the paper}

Section \ref{sec:prelim}
contains the preliminaries on stochastic analysis in  \textsc{umd} Banach spaces 
and its applications to stochastic evolution equations in such spaces. 
The precise assumptions on the operator $A$ and the nonlinearities $F$ and $G$ 
in \eqref{SDE}, which are assumed to hold
throughout the article, are stated in Subsection \ref{subsec:stoch_evol_eq}.
\par

The first main result of this article, Theorem \ref{thm:convV}, concerns the
convergence of the so-called modified splitting scheme. Section \ref{sec:split} 
deals entirely with the proof of this result and a comparison with the classical
splitting scheme.  
Section \ref{sec:gammaHP} is dedicated to obtaining quantitative bounds for the 
rate of convergence in the Trotter-Kato formula \eqref{eq:TK}. Our second main result, 
Theorem \ref{t:euler}, is presented in  Section \ref{sec:absEuler}.
It provides optimal convergence rates for certain abstract 
time discretization schemes including the implicit-linear Euler
scheme.\par

The theorems presented in the Sections \ref{sec:split} and \ref{sec:absEuler}
concerning convergence rates do not provide rates in the H\"older norm, but in a
class of spaces $\Winfp{\alpha}([0,T]\times\Omega;X)$, $p\in (2,\infty)$,
introduced in Section \ref{sec:prelim}. A Kolmogorov type argument then allows us to 
obtain pathwise H\"older convergence rates. This is demonstrated in Section
\ref{sec:pathwise}, thereby completing the proofs of Theorems \ref{t:euler_concrete_intro} 
and \ref{t:split_intro}. By a Borel-Cantelli argument, we also give an almost sure
version of the main theorems.
In Section \ref{sec:local} we show how to extend the almost sure
pathwise convergence of Section \ref{sec:pathwise} to the case that $F$ and $G$ are locally
Lipschitz continuous.

The final section \ref{sec:example} contains the application of our results to a class of 
SPDEs driven by multiplication space-time white noise.
\par

In the set-up of this paper, the convergence
of the Euler scheme is deduced from the convergence of the splitting scheme.
A more streamlined proof for the Euler scheme would be possible, but we have chosen 
the present indirect route for the following reason.
In the splitting scheme, the semigroup is discretized, but not the noise.
In the Euler scheme, both the semigroup and the noise are discretized.
Because of this, it is not possible to derive the correct rates for the
 splitting scheme from those of the Euler scheme. The present arrangement
gives the optimal rates for both schemes.
 
In Sections \ref{app:1} and \ref{app:1b} 
of the Appendix, 
we prove some technical lemmas whose proofs would 
interrupt the flow of the main text. In Section \ref{app:2}
we state and prove an existence and uniqueness result for mild solutions to the 
problem \eqref{SDE}.

\section{Preliminaries}\label{sec:prelim}
Throughout this paper, we use $H$ denotes a real Hilbert space 
with inner product $[\cdot,\cdot]$ and $X$ a real Banach space with duality
 $\lb\cdot,\cdot\rb$ between $X$ and its dual $X^*$.
 
Our work relies on the theory of stochastic integration in \textsc{umd} Banach
spaces developed in \cite{NVW07a}. 
For an overview of the theory of \textsc{umd} spaces we refer to 
\cite{Bur01} and the references given therein. Examples of \textsc{umd} spaces
are Hilbert spaces
and the spaces $L^p(\mu)$ with $1<p<\infty$ and $\mu$ a $\sigma$-finite measure.
We shall frequently
use the following well-known facts:
\begin{itemize}
 \item Banach spaces isomorphic to a closed subspace of a \textsc{umd} space are
\textsc{umd};
 \item If $X$ is \textsc{umd}, $1<p<\infty$, and $\mu$ is a $\sigma$-finite
measure,
       then $L^p(\mu;X)$ is \textsc{umd};
 \item Every \textsc{umd} space is $K$-convex. Hence, by a theorem of Pisier
\cite{Pis:75}, 
       every \textsc{umd} space has non-trivial type $\tau\in (1,2]$. 
\end{itemize}
For a thorough treatment of the notion of {\em type} and the dual notion of {\em cotype}
we refer to the monograph of Albiac and Kalton \cite{AlbKal}. 

The results of  \cite{NVW07a} make use of the concept of $\gamma$-radonifying operators
in an essential
way. We refer to \cite{Nee-survey} for a survey on this topic and 
shall use the notations and the results of this paper freely.

Let $\H$ be a Hilbert space. The Banach space $\gamma(\H,X)$ is defined as the completion of $\H\otimes X$ with
respect to the norm
$$ \Big\n \sum_{n=1}^N h_n\otimes x_n \Big\n_{\gamma(\H,X)}^2 := \E \Big\n
\sum_{n=1}^N \gamma_n\otimes x_n \Big\n^2.$$
Here we assume that $(h_n)_{n=1}^N$ is an orthonormal sequence in $\H$, 
$(x_n)_{n=1}^N$ is a sequence in $X$, and  $(\gamma_n)_{n=1}^N$ is a standard
Gaussian sequence on some
probability space. The space $\gamma(H,X)$ embeds continuously into $\calL(\H,X)$
and it elements
are referred to as the {\em $\g$-radonifying} operators from $\H$ to $X$. In the
special cases where
$\H = L^2(S)$ and $\H = L^2(S;H)$ we write 
$$ \g(L^2(S);X) = \g(S;X), \qquad \g(L^2(S;H),X) = \g(S;H,X).$$
If $X$ is a space with type $2$, then we have a continuous embedding
\begin{equation}\label{L2Embed}
L^2(S,\g(H,X)) \hookrightarrow \gamma(S;H,X)
\end{equation}
with norm depending only on the type $2$ constant of $X$.\par
We shall frequently need the so-called {\em ideal property}. It states that if 
$\H_1$ and $\H_2$ are Hilbert spaces and $X_1$ and $X_2$
are Banach spaces, then for all $V\in\calL(\H_2,\H_1)$, $T\in\g(\H_1,X_1)$,  
and $U\in\calL(X_1,X_2)$ we have $UTV\in\g(\H_2,X_2)$ and
\begin{equation}\label{eq:ideal} 
\n U TV \n_{\g(\H_2,X_2)} \le \n U\n \, \n V\n_{\g(\H_1,X_1)} \n R\n.
\end{equation}

By \cite[Proposition 2.6]{NVW07a}, the mapping
$J:L^p(R;\gamma(\mathcal{H},X))\rightarrow
\calL(\mathcal{H},L^p(R;X))$, where $p\in [1,\infty)$, defined by
\begin{align*}
((Jf)h)(r):=f(r)h,\quad r \in R, \ h\in \mathcal{H},
\end{align*}
defines an isomorphism of Banach spaces
\begin{align}\label{eq:gFub}
L^p(R;\gamma(\mathcal{H},X))\simeq \gamma(\mathcal{H},L^p(R,X)).
\end{align}

\addtocontents{toc}{\SkipTocEntry}\subsection{Stochastic
integration}\label{ss:stochint}
An {\em $H$-cylindrical Brownian motion} with respect to
a filtration $(\mathcal{F}_t)_{t\in
[0,T]}$ is a linear mapping
$W_{H}:L^2(0,T;H)\rightarrow L^2(\Omega)$ with the following properties: 
\begin{enumerate}
\item for all $f\in L^2(0,T;H)$,  $W_{H}(f)$ is Gaussian;
\item for all $f_1,f_2\in L^2(0,T;H)$ we have $\E (W_{H}(f_1)W_{H}(f_2))=
[f_1,f_2]$;
\item for all $h\in H$ and $t\in [0,T]$, $W_H(1_{(0,t]}\otimes
h)$ is $\mathcal{F}_t$-measurable;
\item for all $h\in H$ and $0\le s\leq t<\infty$, 
$W_H(1_{(s,t]}\otimes h)$ is independent of $\mathcal{F}_s$.
\end{enumerate}

For all $f_1,\ldots,f_n \in L^2(0,T;H)$ the
random variables $W_{H}(f_1),\ldots,W_{H}(f_n)$ are jointly Gaussian. As a
consequence,
these random variables are independent if and only if $f_1,\ldots,f_n$ are
orthogonal
in $L^2(0,T;H)$. For further details on cylindrical Brownian motions see
\cite[Section
3]{Nee-survey}.\par

The stochastic integral with respect to an $H$-cylindrical Brownian motion
$W_H$ 
of a {\em finite rank adapted step process}
$\Phi: (0,T)\times \Omega \to H\otimes X$ 
of the form
\begin{align*}
\Phi(t,\omega)= \sum_{n=1}^{N}1_{(t_{n-1},t_n]}(t)
\sum_{m=1}^{M} 1_{A_{nm}}(\omega) \sum_{k=1}^{K}h_k \otimes x_{nmk},
\end{align*}
where $0\le t_0<t_1<...<t_N< T$, $A_{nm}\in \F_{t_{n-1}}$, $x_{nmk}\in X$, and
the vectors $h_k$ are orthonormal in $H$,
is defined by
\begin{align*}
\int_0^{t_N} \Phi \,dW_H := \sum_{n=1}^{N} \sum_{m=1}^{M} 1_{A_{nm}}
\sum_{k=1}^{K}
W_H(1_{(t_{n-1},t_n]}\otimes h_k)\otimes x_{nmk}.
\end{align*}
In the above, for $h\in H$, $\phi\in L^2(\Omega)$, and $x\in X$, we write 
$h\otimes x$ for the rank one operator from $H$ to $X$ given by $h'\mapsto
[h,h']x$ and
$\phi\otimes x$ for the random variable
$\omega\mapsto \phi(\omega)x$.

\begin{theorem}[Burkholder-Davis-Gundy estimates (\hbox{\cite{NVW07a}} for $p\in (1,\infty)$,  
\hbox{\cite{CoxVer:10}} for $p\in (0,\infty)$)]\label{t:stochint}
Let $X$ be a \textsc{umd} Banach space and let $p\in (1,\infty)$ be fixed. 
For every finite rank adapted step process $\Phi: (0,T)\times \Omega \to H\otimes
X$
we have 
\begin{equation}\label{BDG}
\E\, \sup_{0\leq t\leq T} \Big\n \int_{0}^t \Phi \,dW_H \Big\n_X^p \eqsim_p
\E\, \n \Phi \n_{\g(0,T;H,X)}^{p},
\end{equation}
the implied constants being independent of $\Phi$.
\end{theorem}

In what follows we denote by $$L_\F^p(\O;\g(0,T;H,X))$$ 
the closure in $L^p(\O;\g(0,T;H,X))$ of the adapted finite rank step processes. 
Due to estimate \eqref{BDG}, an adapted measurable process $\Phi:[0,T]\times \Omega\rightarrow X$ 
is \emph{$L^p$-stochastically integrable with respect to $W_H$} if and only if 
it defines an element of $L_\F^p(\O;\g(0,T;H,X))$.

We shall frequently use the following H\"older norm
estimate for adapted processes $\Phi\in L^p_{\calF}(\Omega;\gamma(0,T;H,X))$, 
which follows from Theorem \ref{t:stochint} and a domination argument.
For all $\a\in [0,\frac12)$ and $T_0\in (0,T]$: 
\begin{equation}
\begin{aligned}\label{stochIntCont}
& \Big\n s\mapsto \int_{0}^{s} \Phi(u)\,dW_H(u)
\Big\n_{C^{\alpha}([0,T_0];L^p(\Omega;X))} 
\\ & \qquad \eqsim 
\sup_{0\leq t\leq T_0}\n \Phi \n_{L^p(\Omega;\gamma(0,t;H,X))} +
\sup_{0\leq s<t\leq T_0} {(t-s)^{-\alpha}}{\n
\Phi \n_{L^p(\Omega;\gamma(s,t;H,X))}}\\
& \qquad \leq 
\n \Phi \n_{L^p(\Omega;\gamma(0,T_0;H,X))} +
\sup_{0\leq t\leq T_0}\n s\mapsto (t-s)^{-\alpha}\Phi(s)
\n_{L^p(\Omega;\gamma(0,t;H,X))}\\
& \qquad \le (T^{\a}+1)
\sup_{0\leq t\leq T_0}\n s\mapsto (t-s)^{-\alpha}\Phi(s)
\n_{L^p(\Omega;\gamma(0,t;H,X))}
\end{aligned}
\end{equation}
with implied constant independent of $\Phi$ and $T_0$.

\addtocontents{toc}{\SkipTocEntry}\subsection{Besov spaces}\label{sec:besov}

An important tool for estimating the $\gamma(0,T;X)$-norm is the Besov embedding \eqref{BesovEmbed}
below. \par

Fix an interval $I=(a,b)$ with
$-\infty\leq a < b\leq
\infty$ and let $X$ be a Banach space. For $q,r\in [1,\infty]$ and $s\in (0,1)$
the Besov space
$B_{q,r}^{s}(I;X)$ is defined by:
\begin{align*}
B_{q,r}^{s}(I;X) & = \{ f\in L^q(I;X)\,:\,\n f\n_{B_{q,r}^{s}(I;X)}<\infty\},
\end{align*}
where
\begin{align*}
\n f\n_{B_{q,r}^s(I;X)} &:= \n f\n_{L^q(I;X)} + \Big( \int_{0}^{1} \rho^{-sr}
\sup_{|h|<\rho}\n T_h^I f
-f \n_{L^q(I;X)}^r
\frac{d\rho}{\rho} \Big)^{\inv{r}},
\end{align*}
with, for $h\in \R$,
\begin{align*}
T_h^I f(s)=\left\{
\begin{array}{ll} 
f(s+h); &s+h\in I,\\
0;& s+h\notin I.          
\end{array}\right.
\end{align*}

Observe that if $I'\subseteq I$ are nested intervals, then we have a natural
contractive 
restriction mapping from $B_{q,r}^s(I;X)$ into $B_{q,r}^s(I';X)$

If (and only if) a Banach space $X$ has type $\tau\in [1,2)$,
by \cite{NVW07a} we have a continuous embedding
\begin{equation}\label{BesovEmbed}
B_{\tau,\tau}^{\inv{\tau}-\inv{2}}(I;\gamma(H,X)) \hookrightarrow \gamma(I;H,X),
\end{equation}
where the constant of the embedding depends on $|I|$ and the type $\tau$
constant of $X$.\par

\addtocontents{toc}{\SkipTocEntry}\subsection{Randomized boundedness for
analytic semigroups}\label{sec:analytic}

Let $(\gamma_k)_{k\ge 1}$ denote a sequence of real-valued
independent standard Gaussian random variables.
A fam\-ily of operators $\mathscr{R}\subseteq \calL(X_1,X_2)$ is called {\em
$\g$-bounded} if there exists a finite constant $C\ge 0$ such that for all
finite choices $R_1,\dots,R_n \in  \mathscr{R}$ and vectors $x_1,\dots,x_n\in
X_1$
we have
$$ \E \Big\n \sum_{k=1}^n \g_k R_k x_k\Big\n_{X_2}^2 \le C^2
\E \Big\n \sum_{n=1}^n \g_k x_k\Big\n_{X_1}^2.$$
The least admissible constant $C$ is called the {\em $\g$-bound} of
$\mathscr{R}$, notation $\g(\mathscr{R})$. 
When we want to emphasize the domain and range spaces we shall write
$\g_{[X_1,X_2]}(\mathscr{R})$.
Replacing the role of the Gaussian sequence by a Rademacher sequence we arrive 
at the related notion of {\em $R$-boundedness}. Every $R$-bounded set is
$\gamma$-bounded,
and the converse holds if $X_1$ has finite cotype.
We refer to \cite{CPSW, DHP, KunWei, Wei} for examples and more information
$\g$-boundedness and $R$-boundedness.\par

The following $\g$-multiplier result, due to Kalton
and Weis \cite{KalWei07} (see also \cite{Nee-survey}),
establishes a relation between stochastic integrability
and $\g$-bounded\-ness. 

\begin{theorem}[$\gamma$-Multiplier theorem]\label{t:KW}
Suppose $X_1$ does not contain a closed subspace
isomorphic to $c_0$.
Suppose $M:(0,T)\to \calL(X_1,X_2)$ is a
strongly measurable function
with $\gamma$-bounded range $\mathcal{M}=\{M(t): \ t\in (0,T)\}$.
If $\Phi\in  \g(0,T;H,X_1)$ then $M\Phi\in \g(0,T;H,X_2)$ and:
$$ \n M\Phi\n_{\g(0,T;H,X_2)} \leq \g_{[X_1,X_2]}(\mathcal{M})\,
\n \Phi\n_{\g(0,T;H,X_1)}.$$
\end{theorem}
Due to Theorem \ref{t:stochint}, the theorem above implies that if
$\Phi\in  L^p_{\calF}(\Omega;\g(0,T;H,X_1))$ for some  $p\in (1,\infty)$, 
then the function $M\Phi:(0,T)\times\Omega\to \g(H,X_2)$ is
$L^p$-stochastically integrable and 
$$
\Big \n \int_0^T M\Phi\,dW_H \Big\n_{L^p(\O;X_2)} \lesssim
\g_{[X_1,X_2]}(\mathcal{M})\,\Big\n
\int_0^T \Phi\,dW_H\Big\n_{L^p(\O;X_1)}.
$$

The $\g$-multiplier theorem will frequently by applied in conjunction with the following
basic result due to Kaiser and Weis \cite[Corollary 3.6]{KaiWei08}:

\begin{theorem}\label{t.KaiWei} Let $X$ be a Banach space with finite cotype.
Define, for every $h\in H$, the operator $U_h: X\to \g(H,X)$
by
$$ U_h x := h\otimes x, \quad x\in X.$$
Then the family $\{U_h: \ \n h\n\le 1\}$ is $\g$-bounded.
\end{theorem}

We proceed with a useful $\g$-boundedness results which will
be used frequently below. It is a minor variation of \cite[Lemma 4.1]{NVW08},
which is a key ingredient in
the proof of the existence and uniqueness result for stochastic evolution
equations
proved there. The proof
is entirely analogous and is left to the reader; it is based on the fact that if
$A$ generates an analytic
$C_0$-semigroup on $X$, then for all $\alpha>0$ and any $T>0$ there exists a
constant $C$ such that
\begin{equation}\label{analyticDiff}
\n S(t) \n_{\calL(X, X_{\alpha})} \leq Ct^{-\alpha}; \quad \textrm{for all }
t\in (0,T];
\end{equation}
and
\begin{equation}\label{analyticSGIdiff}
\n S(t) - I \n_{\calL(X_{\alpha}, X)} \leq Ct^{\alpha\minsym 1}; \quad \textrm{for all }
t\in (0,T].
\end{equation}
In the typical application of the lemma, it is combined with the Kalton-Weis
$\gamma$-multiplier theorem to estimate $\g$-radonifying norms
of certain vector-valued or operator-valued functions.

\begin{lemma}\label{lem:analyticRbound} 
Let $A$ generate an analytic $C_0$-semigroup $S$ on a Banach space $X$.
\begin{enumerate}\item[\rm(1)] For all $0\leq \alpha<\beta$ and $t\in (0,T]$
the set
$\mathscr{S}_{\beta,t} = \{s^\beta S(s): \ s\in [0,t]\}$ is $\g$-bounded
in $\calL(X,X_\a)$ and we have
$$\gamma_{[X,X_{\alpha}]}(\mathscr{S}_{\beta,t})\lesssim t^{\beta-\a},$$
with implied constant independent of $t\in (0,T]$.
\item[\rm(2)] For all $0<\alpha\le 1$ and $t\in (0,T]$ the set
$\mathscr{S}_{t} = \mathscr{S}_{0,t}  = \{S(s): \ s\in [0,t]\}$ is $\g$-bounded
in $\calL(X_\a,X)$ and we have $$\gamma_{[X_{\a},X]}(\mathscr{S}_{t})\lesssim
t^{\a},$$
with implied constant independent of $t\in (0,T]$.
\item[\rm(3)] For all $0<\alpha\le 1$ and $t\in (0,T]$ 
the set $\mathscr{T}_{t} = \{S(s)-I: \ s\in [0,t]\}$ is $\g$-bounded
in $\calL(X_\a,X)$ and we have $$\gamma_{[X_\a,X]}(\mathscr{T}_{t})\lesssim
t^{\a},$$
with implied constant independent of $t\in (0,T]$.
\end{enumerate}\end{lemma}
Let us emphasize that the constants in Lemma \ref{lem:analyticRbound} may depend
on the final time $T$: 
all we are asserting is that, given $T$, the constants are independent of $t\in
[0,T]$.

\addtocontents{toc}{\SkipTocEntry}\subsection{Stochastic evolution
equations}\label{subsec:stoch_evol_eq}

\subsubsection{The equation}\label{ss:setting} We are interested in the 
convergence rate of various numerical schemes
associated with stochastic evolution equations of the form:
\begin{equation}\label{SEE}\tag{SEE}
\left\{ \begin{aligned} dU(t) & = AU(t)\,dt + F(t,U(t))\,dt +
G(t,U(t))\,dW_H(t);\quad t\in
[0,T],\\
U(0)&=x_0. \end{aligned}\right.
\end{equation}
Here, $0<T<\infty$ is fixed and $W_H$ is an $H$-cylindrical $(\F_t)_{t\in
[0,T]}$-Brownian motion on a probability space 
$(\Omega,\F,\P)$. We make the following standing assumptions on the Banach space
$X$, the operator $A$,
 and the functions $F$ and $G$:

\begin{itemize}
\item[\MA{}] $A$ generates an analytic $C_0$-semigroup on a \textsc{umd} Banach
space $X$.
\item[\MF{}] For some $\theta_F>  -1 + (\inv{\tau}-\frac12)$, where $\tau$ is
the type of $X$, the function 
$F:[0,T]\times X\rightarrow X_{\theta_F}$ is measurable in the sense that for
all $x\in X$ the mapping $F(\cdot,x):[0,T]\rightarrow X_{\theta_F}$ is strongly
measurable. Moreover, $F$ is uniformly Lipschitz continuous and uniformly of
linear growth in its second variable.
That is to say, 
there exist constants $C_0$ and
$C_1$ such that for all $t\in [0,T]$ and all $x,x_1,x_2\in X$:
\begin{align*}
\n F(t,x_1) - F(t,x_2) \n_{X_{\theta_F}} & \leq C_0 \n x_1-x_2\n_{X}, \\ 
\n F(t,x)\n_{X_{\theta_F}} &\leq 
C_1(1+\n x\n_{X}).
\end{align*}
The least constant $C_0$ such that the above holds is denoted by Lip$(F)$, and
the least constant $C_1$ such that the
above holds is denoted by $M(F)$.\par

\item[\MG{}] For some $\theta_G>-\inv{2}$, the function $G : [0,T]\times
X\rightarrow \calL(H,X_{\theta_G})$ 
is measurable in the sense that for all $h\in H$ and $x\in X$ the mapping
$G(\cdot,x)h:[0,T]\rightarrow X_{\theta_G}$ is strongly measurable. Moreover,
$G$ is uniformly $L^2_{\gamma}$-Lipschitz continuous and uniformly of linear
growth in its second variable. That is to say, 
there exist constants $C_0$ and 
$C_{1}$ such that for 
all $\alpha\in [0,\inv{2})$, all $t\in [0,T]$, and all simple functions
$\phi_1$, $\phi_2$, $\phi: [0,T]\to X$ one has:
\begin{align*}
& \qquad \n s\mapsto (t-s)^{-\alpha}[G(s,\phi_1(s)) -G(s,\phi_2(s))]
\n_{\gamma(0,t;H,X_{\theta_G})} \\
& \qquad \qquad \qquad \le  C_0 \n s\mapsto (t-s)^{-\alpha}[\phi_1 -\phi_2] 
\n_{L^2(0,t;X)\,\cap \,\gamma(0,t;X)}; \\
& \qquad \n s\mapsto (t-s)^{-\alpha}
G(s,\phi(s))\n_{\gamma(0,t;H,X_{\theta_G})}\\
& \qquad \qquad \qquad  \leq C_1\big(1+ \n s\mapsto
(t-s)^{-\alpha}\phi(s)\n_{L^2(0,t;X)\,\cap\, \gamma(0,t;X)}\big).
\end{align*}
The least constant $C_0$ such that the above holds is denoted by
$\textrm{Lip}_{\gamma}(G)$, and the 
least constant $C_1$ such that the
above holds is denoted by $M_\gamma(G)$.
\end{itemize}\par
Our definition of $L^2_{\gamma}$-Lipschitz continuity is a slight adaptation of
the definition 
given in \cite{NVW08}. Examples of $L^2_{\gamma}$-Lipschitz continuous operators
can be found in
that article. In particular:
\begin{itemize}
 \item If $G$ is defined by an Nemytskii map on $[0,T]\times L^p(R)$, where
$p\in [1,\infty)$ and $(R,\calR,\mu)$ a
$\sigma$-finite
measure space, then $G$ is $L^2_{\gamma}$-Lipschitz continuous (see
\cite[Example 5.5]{NVW08}).
 \item if $G:[0,T]\times X_1\rightarrow \gamma(H,X_2)$ is Lipschitz continuous,
uniformly in $[0,T]$, and $X_2$ is a
type 2 space, then $G$ is $L^2_{\gamma}$-Lipschitz continuous (see \cite[Lemma
5.2]{NVW08}).
\end{itemize}

\subsubsection{Existence and uniqueness}

In \cite[Theorem 6.2]{NVW08} an existence and uniqueness result is presented for solutions to \eqref{SEE}
under the conditions \MA{}, \MF{}, \MG. In this section we give a variation to these results (see
Theorem \ref{thm:NVW08} and Remark \ref{r:diffNVW} below). We shall prove existence and uniqueness
of a solution to \eqref{SEE} in a class of spaces which turns out to be the most suitable for
proving convergence of the various numerical schemes under consideration.\par

\begin{definition}\label{d:Wspace}
For $\alpha \geq 0$ and $0\leq a<b<\infty$ we define $\Winfp{\alpha}([a,b]\times
\Omega;X)$ to be the space containing all adapted processes $\Phi\in 
L^p_{\calF}(\Omega;\gamma(a,b;X))$ for which the 
following norm is finite:
\begin{align*}
\n \Phi\n_{\Winfp{\alpha}([a,b]\times \Omega;X)}&= \n \Phi \n_{L^{\infty}(a,b;L^p(\Omega;X))} 
+ \sup_{a\leq t\leq b}\n
s\mapsto
(t-s)^{-\alpha}\Phi(s)\n_{L^p(\Omega;\gamma(a,t;X))}.
\end{align*}\par
\end{definition}

For $0\le\beta\le\alpha<\frac12$ the $\g$-multiplier theorem (Theorem
\ref{t:KW}) implies:
\begin{equation}\label{Vchange-of-alpha}
\n\Phi\n_{\Winfp{\beta} ([a,b]\times \Omega;X)} \leq (b-a)^{\beta-\alpha}\n
\Phi\n_{\Winfp{\alpha}([a,b]\times \Omega;X)}.
\end{equation}
One also checks that for $0\leq a \leq b \leq T$,
\begin{equation}\label{Vtransinv}
\n \Phi 1_{[a,b]} \n_{ \Winfp{\alpha}([0,T]\times\Omega;X)} = \n \Phi|_{[a,b]} \n_{
\Winfp{\a} ([a,b]\times \Omega;X)}.
\end{equation}
Finally, if $G:[0,T]\times X\rightarrow \calL(H,X_{\theta_G})$ satisfies \MG{} and
$\Phi_1,\Phi_2 \in \Winfp{\alpha}([0,T]\times \Omega;X)$ 
for some $p\geq 2$, then:
\begin{equation}
\begin{aligned}\label{GLipschitzV2}
& \sup_{0\leq t\leq T}\n s\mapsto
(t-s)^{-\alpha}[G(s,\Phi_1(s))-G(s,\Phi_2(s))]\n_{L^p(\Omega;\gamma(0,t;X_{
\theta_G}))}\\
& \quad  \leq \textrm{Lip}_{\gamma}(G) \sup_{0\leq t\leq T}\n s\mapsto
(t-s)^{-\alpha}[\Phi_1(s)-\Phi_2(s)]\n_{L^p(\Omega;L^2(0,t;X))\cap
L^p(\Omega;\gamma(0,t;X))} \\
& \quad  \leq \textrm{Lip}_{\gamma}(G) \sup_{0\leq t\leq T}\n s\mapsto
(t-s)^{-\alpha}[\Phi_1(s)-\Phi_2(s)]\n_{L^2(0,t;L^p(\Omega;X))\cap
L^p(\Omega;\gamma(0,t;X))} \\
&\quad  \leq (1+T^{\inv{2}-\alpha})\textrm{Lip}_{\gamma}(G)\n
\Phi_1-\Phi_2\n_{\Winfp{\alpha}([0,T]\times \Omega;X)},
\end{aligned}
\end{equation}
and, by a similar argument one has, for $\Phi \in \Winfp{\alpha}([0,T]\times \Omega;X)$:
\begin{equation}
\begin{aligned}\label{GLipschitzV1}
&\sup_{0\leq t\leq T}\n s\mapsto
(t-s)^{-\alpha}G(s,\Phi(s))\n_{L^p(\Omega;\gamma(0,t;X_{\theta_G}))} \\
&\qquad \qquad  \qquad  \leq (1+T^{\inv{2}-\alpha})M_\gamma(G)\big( 1+\n \Phi
\n_{\Winfp{\alpha} ([0,T]\times \Omega;X)}\big).
\end{aligned}
\end{equation}
\par

\begin{definition} An adapted, strongly measurable process $U: [0,T]\times\Omega \to X$ 
is called a {\em mild solution}
of \eqref{SEE} if, for all $t\in[0,T]$,
\begin{enumerate}
\item \label{detConvInt} $s\mapsto S(t-s)F(S,U(s))\in L^{0}(\Omega,L^1(0,T;X))$,
\item \label{stochConvInt} $s\mapsto S(t-s)G(s,U(s))$ is $H$-strongly measurable, 
adapted and almost
surely in $\gamma(0,t;H,X)$,
\end{enumerate}
and moreover $U$ satisfies:
\begin{equation}\label{SDE_sol}
U(t)=S(t)x_0 + \int_{0}^{t} S(t-s)F(s,U(s))\,ds + \int_{0}^{t}
S(t-s)G(s,U(s))\,dW_H(s)
\end{equation}
almost surely for all $t\in [0,T]$. 
\end{definition}

A rigorous definition of the stochastic
integral in \eqref{SDE_sol} can be given if condition \eqref{stochConvInt} in the definition above
is satisfied, but this is beyond the theory discussed in \ref{ss:stochint}. We refer to \cite{NVW08}
for the details. With exception of Section \ref{sec:local}, throughout this article the process
$s\mapsto S(t-s)G(s,U(S))1_{[0,t]}$ will always be $L^p$-stochastically integrable for some $p\in
(1,\infty)$.

The proof of the following theorem, which is entirely 
analogous to the proof given for \cite[Theorem 6.2]{NVW08}, is presented in
Appendix \ref{app:2}.

We set \begin{align}\label{eq:eta-max}\eta_{\rm max} := 
\min\{1 -(\tinv{\tau} - \tfrac{1}{2})+\theta_F,\,\tinv{2}+\theta_G\}.
\end{align}
\begin{theorem}\label{thm:NVW08} 
Let  $0\le \eta < \eta_{\max}$
and $p\in [2,\infty)$. 
For all initial values $ x_0\in
L^p(\Omega;\mathcal{F}_0;X_{\eta})$ and all $\alpha\in
[0,\inv{2})$, the problem \eqref{SEE} has 
a unique mild solution $U$ in $\Winfp{\alpha}([0,t]\times\Omega;X_\eta)$. It satisfies 
\begin{align}\label{UinVestimate}
\n U \n_{ \Winfp{\alpha}([0,T]\times \Omega;X_\eta)}  & \lesssim 1+\n
x_0\n_{L^p(\Omega;X_{\eta})}.
\end{align}
\end{theorem}
\begin{remark}\label{r:diffNVW}
In \cite{NVW08} the authors prove existence in the space $\Vapc ([0,T]\times \Omega;X_\eta)$ of all 
continuous adapted 
processes $\Phi\in L_{\calF}^{p}(\Omega;\gamma(0,T;X))$ 
for which the norm
\begin{equation}\label{eq:Vapc}
\n \Phi\n_{\Vapc ([0,T]\times \Omega;X_\eta)} \!:= \!\n \Phi \n_{L^p(\Omega;C([0,T];X_\eta))}
+ \sup_{0\leq t\leq T}\n
s\mapsto (t-s)^{-\alpha}\Phi(s)\n_{L^p(\Omega;\gamma(0,t;X_\eta))}
\end{equation}
is finite,
under the assumption that $\inv{p}<\inv{2}+(\theta_G\minsym 0)$ and $0\leq \eta <
\min\{1- (\frac1\tau-\frac12)+(\theta_F\minsym 0),\,\inv{2}+(\theta_G\minsym 0)-\inv{p} 
\}$. 
The extra factor $\frac1p$ arises from a Kolmogorov-type estimate on the
stochastic integral.\par
Note that the approximations obtained by the splitting scheme and the 
Euler scheme are not continuous. Accordingly we first prove convergence of 
the various schemes in $\Winfp{\alpha}([0,T]\times\Omega;X)$ for arbitrarily 
large $p\in (2,\infty)$. Pathwise convergence results (in the grid points) 
can then be obtained by a Kolmogorov argument. However, if we were to use the 
existence and uniqueness results of \cite{NVW08}, we would lose a factor $\inv{p}$ 
twice. Instead, we shall use Theorem \ref{thm:NVW08}.
\end{remark}

\section{Convergence of the splitting-up method}\label{sec:split}
We consider the stochastic differential equation \eqref{SEE} under the
assumptions \MA{}, \MF{}, \MG{} and with initial value $x_0\in
L^p(\Omega,\F_0;X_\eta)$ with $0\le \eta<\eta_{\rm max}$.

In order to define a scheme, which we shall call the {\em modified
splitting scheme} (for reasons to be explained shortly), we fix an initial value 
$y_0 \in L^p(\Omega,\mathcal{F}_0;X)$,
possibly different from $x_0$,
and fix an integer $n\in \N$. For $j=1,\ldots,n$ we define the process
$U^{(n)}_{j}:[t_{j-1}^{(n)},t_{j}^{(n)}]\times \Omega\rightarrow X$ as the 
mild solution to the problem 
\begin{equation}\label{SDE_noAmod}
\left\{ \begin{aligned} dU^{(n)}_{j}(t) & =
S(\tfrac{T}{n})\big[F(t,U^{(n)}_{j}(t))\,dt +
G(t,U^{(n)}_{j}(t))\,dW_H(t)\big], \quad   t\in
[t_{j-1}^{(n)},t_{j}^{(n)}];\\
U^{(n)}_{j}(t_{j-1}^{(n)})&=S(\tfrac{T}{n})U^{(n)}_{j-1}(t_{j-1}^{(n)}),
\end{aligned}\right.
\end{equation}
where we set $U^{(n)}_{0}(0):=y_0$; 
recall that $t_j^{(n)}:=\tfrac{jT}{n}.$

The existence of a unique mild solution to \eqref{SDE_noA} in
$\Winfp{\alpha}([t_{j-1}^{(n)},t_{j}^{(n)}]\times
\Omega;X)$, for $\alpha\in [0,\inv{2})$ and $p\in [2,\infty)$, 
is guaranteed by Theorem
\ref{thm:NVW08}. Here we use that $S(\tfrac{T}{n})F: [0,T] \times
X\rightarrow X$
satisfies \MF{} and that $S(\tfrac{T}{n})G: [0,T] \times X\rightarrow
\gamma(H,X)$ satisfies \MG{}.\par
For $j=1,\ldots,n$ we define $I_j^{(n)} :=  [t_{j-1}^{(n)},t_{j}^{(n)})$. 
Observe that the adapted process
$U^{(n)}: [0,T)\times \Omega\to X$ defined by 
\begin{align}\label{splitdef} 
U^{(n)} := \sum_{j=1}^n 1_{I_j^{(n)}}(t)\, U_j^{(n)}(t),  \quad t\in [0,T), 
\end{align}
defines an element of $\Winfp{\alpha}([0,T]\times \Omega;X)$. 
In the next subsection we prove convergence of $U^{(n)}$ against $U$ in this space.

There is a subtle difference between the modified splitting scheme and the 
 \emph{classical splitting scheme}, which is defined by
\begin{equation}\label{SDE_noA}
\left\{ \begin{aligned} d\tilde U^{(n)}_{j}(t) & = F(t,\tilde
U^{(n)}_{j}(t))\,dt +
G(t,\tilde U^{(n)}_{j}(t))\,dW_H(t), \quad
t\in
[t_{j-1}^{(n)},t_{j}^{(n)}];\\
\tilde U^{(n)}_{j}(t_{j-1}^{(n)})&=S(\tfrac{T}{n})\tilde
U^{(n)}_{j-1}(t_{j-1}^{(n)})
\end{aligned}\right.
\end{equation}
with  $\tilde U^{(n)}_{0}(0):=y_0$.
The existence of a unique mild solution $\tilde U^{(n)}_{j}$ to \eqref{SDE_noA}
in
$\Winfp{\alpha}([t_{j-1}^{(n)},t_{j}^{(n)}]\times \Omega;X)$, for every
$\alpha\in [0,\inv{2})$ and $p\in [2,\infty)$, is again guaranteed by Theorem \ref{thm:NVW08} 
{\em provided that} $\theta_F\geq 0$ and $\theta_G\geq 0$. However, if
$\theta_F<0$ or $\theta_G
<0$, then we have
no means to define a solution to \eqref{SDE_noA} in $X$, since we cannot guarantee that
the integrals corresponding to $F$ and $G$ in the definition of a mild solution
take values in $X$. 
 In the modified splitting scheme,
this problem is overcome by extra operator $S(\frac{T}{n})$ with provides the
required additional smoothing. \par

Once convergence of the modified splitting scheme has been established, 
convergence of the classical splitting 
scheme is derived from it under the additional assumptions $\theta_F\geq 0$ 
and $\theta_G\geq 0$ (Theorem \ref{thm:convVclassical}).\par

\subsection{Convergence of the modified splitting scheme}

For $t\in I_j^{(n)}$ we define
\begin{align*}
\un{t}:=t_{j-1}^{(n)}, \quad  
\ov{t}:=t_{j}^{(n)}.  
\end{align*}
In particular,
$\ov{t_{j-1}^{(n)}}=t_{j}^{(n)}$. It should be kept in mind that 
in the notation $\un{x}$ and $\ov{x}$ 
we suppress the dependence on $n$ and
$T$. 

The key idea of our approach is the following observation.

\addtocontents{toc}{\SkipTocEntry}\subsection*{Claim.} 
{\em Let $U^{(n)}$ be defined by \eqref{splitdef}. Almost surely, for all $ t\in
[0,T)$ we have:}
\begin{equation}\label{splitConv}
\begin{aligned}
U^{(n)}(t)  = S(\ov{t})y_0 & + \int_{0}^{t}
S(\ov{t}-\un{s})F\big(s,U^{(n)}(s)\big)\,ds\\
& + \int_{0}^{t}S(\ov{t}-\un{s})G\big(s,U^{(n)}(s)\big)\,dW_H(s).
\end{aligned}
\end{equation}
\begin{proof}[Proof of Claim.]
It suffices to prove that for any $j\in\{1,\ldots,n\}$, almost surely the following identity
holds for all $t\in I_j^{(n)}$:
\begin{equation}\begin{aligned}\label{splitConvHelp}
U_{j}^{(n)}(t) = S(t_{j}^{(n)})y_0 & +
\int_{t_{j-1}^{(n)}}^{t}S(\tfrac{T}{n})F(s,U_{j}^{(n)}(s))\,ds \\
& + \int_{t_{j-1}^{(n)}}^{t}S(\tfrac{T}{n})G(s,U_{j}^{(n)}(s))\,dW_H(s)\\
& + \sum_{k=1}^{j-1} \int_{I_{k}^{(n)}}
S(t_{j-k+1}^{(n)})F(s,U^{(n)}_{k}(s))\,ds\\
& + \sum_{k=1}^{j-1} \int_{I_{k}^{(n)}} 
S(t_{j-k+1}^{(n)})G(s,U^{(n)}_{k}(s))\,dW_H(s).
\end{aligned}\end{equation}
By definition, the process $U^{(n)}_{j}$, being a mild solution to \eqref{SDE_noAmod}, satisfies:
\begin{equation}\begin{aligned}\label{varconsUn}
U^{(n)}_{j}(t) = S(\tfrac{T}{n})U_{j-1}^{(n)}(t_{j-1}^{(n)}) & +
\int_{t_{j-1}^{(n)}}^{t}S(\tfrac{T}{n})F(s,U^{(n)}_{j}(s))\,ds\\
&   + \int_{t_{j-1}^{(n)}}^{t}
S(\tfrac{T}{n})G\big(s,U^{(n)}_{j}(s)\big)\,dW_H(s)
\end{aligned}\end{equation}
almost surely for all $t\in I_j^{(n)}$. For $j=1$ \eqref{splitConvHelp} follows 
directly from \eqref{varconsUn}, and for $j\in \{2,\ldots,n\}$
it follows by induction.
\end{proof}

As always we assume that
\MA{}, \MF{}, \MG{} hold, and we denote by
$U$ the mild solution of the problem \eqref{SEE} with initial value $x_0$. 
We define $U^{(n)}$ as above with initial value $y_0$.
The proof of the following theorem uses the strong resemblance between identity
\eqref{splitConv} for $U^{(n)}$ and the identity  \eqref{SDE_sol} satisfied by $U$.

\begin{theorem}\label{thm:convV}
Let $ 0\le \eta \le 1$ satisfy $\eta< \eta_{\rm max}$, let $p \in [2,\infty)$,
and assume that $x_0\in L^p(\calF_0,X_{\eta})$ and $ y_0\in L^p(\calF_0,X)$. 
Then for all $\alpha\in [0,\inv{2})$ one has:
\begin{equation}\label{Wconvrate}
\begin{aligned}
\n U - U^{(n)} \n_{\Winfp{\alpha}([0,T]\times \Omega;X)} &\lesssim \n x_0-y_0
\n_{L^p(\Omega;X)}
+  n^{-\eta} \big(1+\n x_0\n_{L^p(\Omega;X_{\eta})}\big), 
\end{aligned}
\end{equation}
with implied constants independent of $n$, $x_0$ and $y_0$.
\end{theorem}

\begin{proof}
Let $\e>0$ be such that $$\eps< \min\{\tfrac1{2}, 1-2\a, \eta_{\rm max}-\eta\}.$$ 
In particular we have $\eps<\inv{2}+\theta_G$ and thus, by replacing $\a\in [0,\frac12)$ 
by some larger value if necessary, we may assume that
$$ \max\{1-\tfrac43\eps, \e - 2\theta_G\} < 2\alpha < 1-\eps.$$ 

We split the proof of \eqref{Wconvrate} into several parts.
In each part, constants will be allowed to 
depend on the final time $T$. Thus, the statement
`$A(t)\lesssim B$ with a constant independent of $t\in [0,T]$' is to 
be interpreted as short-hand 
for `there is a constant $C$, possibly depending on $T$, such that 
$\sup_{t\in [0,T]}A(t)\le CB$'.\par

\addtocontents{toc}{\SkipTocEntry}\subsection*{Part 1.} Fix $n\in \N$. Let $x,y\in
L^p(\Omega;X_{\eta})$.
By the identities \eqref{SDE_sol} and
\eqref{splitConv}, 
for all $s\in [0,T]$ we have:
\begin{equation}
\begin{aligned}\label{aph1}
 U(s) -U^{(n)}(s)  
& =  (S(s)-S(\ov{s}))x_0 \\ 
& \qquad + S(\ov{s})(x_0-y_0) \\
& \qquad + \int_{0}^{s}[S(s-u)-S(\ov{s}-\un{u})]F(u,U(u))\,du \\
& \qquad + \int_{0}^{s}S(\ov{s}-\un{u})[F(u,U(u))-F(u,U^{(n)}(u))]\,du \\ 
& \qquad + \int_{0}^{s}[S(s-u)-S(\ov{s}-\un{u})]G(u,U(u))\,dW_H(u) \\
& \qquad + \int_{0}^{s}S(\ov{s}-\un{u})[G(u,U(u))-G(u,U^{(n)}(u))]\,dW_H(u).
\end{aligned}
\end{equation}
We shall estimate the $\Winfp{\alpha}([0,T_0]\times \Omega;X)$-norm of each 
of the six terms separately for arbitrary $T_0\in [0,T]$. 
In the fourth and sixth term (Part 1d and 1f below) it will be necessary to keep track of the
dependence on $T_0$.
\addtocontents{toc}{\SkipTocEntry}\subsection*{Part 1a.} 
We start with the first term in \eqref{aph1}. 
Fix an arbitrary $\b\in (0,\frac12)$. Writing $S(s)-S(\ov{s}) = (I-S(\ov{s}-s))S(s)$ and 
$S(s) = s^{-\beta} s^\beta S(s)$,
from Lemma \ref{lem:analyticRbound} (1) and (3) and Theorem \ref{t:KW}
we obtain, almost surely for all $t\in [0,T]$:
\begin{align*}
\begin{aligned}
& \n s\mapsto (t-s)^{-\alpha}(S(s)-S(\ov{s}))x_0\n_{\gamma(0,t;X)} \\
& \qquad \lesssim n^{-\eta} \n s\mapsto (t-s)^{-\alpha} S(s) x_0\n_{\gamma(0,t;X)} \\
& \qquad \lesssim n^{-\eta} \n s\mapsto (t-s)^{-\alpha} s^{-\beta} x_0\n_{\gamma(0,t;X)} \\
& \qquad = n^{-\eta} \n s\mapsto (t-s)^{-\alpha} s^{-\beta}\n_{L^2(0,t)}\n x_0\n_{X} \\
& \qquad \lesssim n^{-\eta} \n x_0\n_{X},
\end{aligned}
\end{align*}
with implied constants independent of $n$, $t$, and $x_0$.
Also, by \eqref{analyticSGIdiff},
\begin{align*}
& \n s\mapsto (S(s)-S(\ov{s}))x_0\n_{L^{\infty}(0,T_0;X)} \\
& \qquad \leq \sup_{s\in [0,T_0]} \n S(s)\n_{\calL(X)} \n
I - S(\ov{s}-s)\n_{\calL(X_{\eta},X)} \n x_0\n_{X_{\eta}}  \lesssim n^{-\eta} \n x_0\n_{X_{\eta}}.
\end{align*}
By taking $p^{\textrm{th}}$ moments it follows that
that for every $T_0\in [0,T]$ we have:
\begin{align}\label{Ves2}
\n s\mapsto (S(s)-S(\un{s}))x_0\n_{\Winfp{\alpha}([0,T_0]\times \Omega;X)} &
\lesssim  n^{-\eta} \n x_0\n_{L^p(\Omega;X_{\eta})}
\end{align}
with implied constant independent of $n$, $T_0$ and $x_0$.
\addtocontents{toc}{\SkipTocEntry}\subsection*{Part 1b.} 
Concerning the second term on the right-hand side in \eqref{aph1} we note that, almost surely:
\begin{align*}
\n s\mapsto S(\ov{s})(x_0-y_0) \n_{L^{\infty}(0,T;X)} & \lesssim \n x_0-y_0\n_{X}, 
\end{align*} 
with implied constant independent of $n$, $x_0$ and $y_0$.
Also, by Lemma
\ref{lem:analyticRbound} (1) and Theorem \ref{t:KW},
 almost surely we have, for all $t\in [0,T]$:
\begin{align*}
& \n s\mapsto (t-s)^{-\alpha}S(\ov{s})(x_0-y_0)\n_{\gamma(0,t;X)}  \\ 
& \qquad \lesssim  \n s\mapsto
(t-s)^{-\alpha}(\ov{s})^{-\eps}(x_0-y_0)\n_{\gamma(0,t;X)}\\
& \qquad = \n s\mapsto (t-s)^{-\alpha}(\ov{s})^{-\eps}
\n_{L^2(0,t)}\n x_0-y_0\n_{X} \\
& \qquad \lesssim \n x_0-y_0\n_{X}
\end{align*}
with implied constants are independent of $n$, $t$, $x_0$ and $y_0$.\par
Combining these estimates we obtain, for all $T_0 \in [0,T]$:
\begin{align}\label{Ves1}
\n s\mapsto S(\un{s})(x_0-y_0) \n_{\Winfp{\alpha}([0,T_0]\times \Omega;X)} &
\lesssim \n x_0-y_0\n_{L^p(\O;X)}
\end{align}
with implied constants independent of $n$, $T_0$, $x_0$ and $y_0$.
\addtocontents{toc}{\SkipTocEntry}\subsection*{Part 1c.}
Concerning the third term on the right-hand side in \eqref{aph1} we observe that:
\begin{equation}\label{Sdiffdecomp}
S(s-u)-S(\ov{s}-\un{u})= (I-S(\ov{s}-s))S(s-u) +
S(\ov{s}-s)S(s-u)(I-S(u-\un{u})),
\end{equation}
and hence
\begin{align*}
& \int_{0}^{s}[S(s-u)-S(\ov{s}-\un{u})]F(u,U(u))\,du \\
& \qquad \qquad = (I-S(\ov{s}-s))\int_{0}^{s}S(s-u)F(u,U(u))\,du \\
& \qquad \qquad \qquad +
S(\ov{s}-s)\int_{0}^{s}S(s-u)(I-S(u-\un{u}))F(u,U(u))\,du.
\end{align*}
Let $T_0\in [0,T]$. It follows from Lemma
\ref{lem:analyticRbound}, part (2) (with exponent $\frac12\eps$) and part  
(3) (with exponent $\frac{3}{2}+\theta_F-\inv{\tau}-\eps)$ and
and the $\g$-multiplier theorem (Theorem \ref{t:KW})
that:
\begin{equation}\label{h1ca}
\begin{aligned}
&\Big\n  s\mapsto \int_{0}^{s}[S(s-u)-S(\ov{s}-\un{u})]F(u,U(u))\,du
\Big\n_{\Winfp{\alpha}([0,T_0]\times \Omega;X)}\\
& \ \lesssim  n^{-\min\{\frac{3}{2}+\theta_F-\inv{\tau}-\eps,1\}}\Big\n s\mapsto
\int_{0}^{s}S(s-u)F(u,U(u))\,du\Big\n_{\Winfp{\alpha}([0,T_0]\times
\Omega;X_{\frac{3}{2}+\theta_F-\inv{\tau}-\eps})}\\
& \  \quad + \Big\n s\mapsto
\int_{0}^{s}S(s-u)(I-S(u-\un{u}))F(u,U(u))\,du\Big\n_{
\Winfp{\alpha}([0,T_0]\times
\Omega;X_{\frac12\eps })},
\end{aligned}
\end{equation}
with implied constants independent of $n$ and $T_0$.
We shall estimate the two terms on the right-hand side of 
\eqref{h1ca} separately. 

We begin with the first term. 
Recall that $U\in \Winfp{\alpha}([0,T_0]\times
\Omega;X)$ and therefore, by  \MF{}, we have $F(\cdot,U(\cdot))
\in L^{\infty}(0,T_0;L^p(\Omega;X_{\theta_F}))$.
By Lemma \ref{lem:detConv} (applied
with $Y=X_{\frac{3}{2}+\theta_F-\inv{\tau}-\eps}$, $\Phi(u)=F(u,U(u))$, and 
$\delta=
-\frac{3}{2}+\inv{\tau}+\eps$) we obtain, for all $t\in [0,T_0]$:
\begin{align*}
&\Big\n s\mapsto
\int_{0}^{s}S(s-u)F(u,U(u))\,du\Big\n_{\Winfp{\alpha}([0,T_0]\times
\Omega;X_{\frac{3}{2}+\theta_F-\inv{\tau}-\eps})} \\
& \qquad \lesssim \n u\mapsto
F(u,U(u))\n_{L^{\infty}(0,T_0;L^p(\Omega;X_{\theta_F}))}\\
& \qquad \lesssim (1+\n
U\n_{L^{\infty}(0,T_0;L^p(\Omega;X))}),
\end{align*}
with implied constants independent of $n$, $x_0$ and $T_0$. \par

For the second term in the right-hand side of \eqref{h1ca} we apply Lemma \ref{lem:detConv} 
(with $Y=X_{\frac12\eps }$,
$\delta=-\frac{3}{2}+\inv{\tau}+\frac12\eps $ and
$\Phi(u)=(I-S(u-\un{u}))F(u,U(u))$). Note that $\Phi \in
L^{\infty}(0,T;L^p(\Omega;X_{-\frac{3}{2}+\inv{\tau}+\eps}))$
by
the boundedness of $u\mapsto (I-S(u-\un{u}))$ in
$\calL(X_{\theta_{F}},X_{-\frac{3}{2}+\inv{\tau}+\eps})$, the linear growth
condition in \MF{} and the fact that $U\in \Winfp{\alpha}([0,T_0]\times
\Omega;X)$. We obtain:
\begin{align*}
&\Big\n s\mapsto
\int_{0}^{s}S(s-u)(I-S(u-\un{u}))F(u,U(u))\,du\Big\n_{
\Winfp{\alpha}([0,T_0]\times
\Omega;X_{\frac12\eps })} \\
& \quad \lesssim \n u\mapsto (I-S(u-\un{u}))F(u,U(u))
\n_{L^{\infty}(0,T_0;L^p(\Omega;X_{-\frac{3}{2}+\inv{\tau}+\eps}))}\\
& \quad \lesssim n^{-\min\{\frac{3}{2}+\theta_F-\inv{\tau}-\eps,1\}}\n
u\mapsto F(u,U(u))\n_{L^{\infty}(0,T_0;L^p(\Omega;X_{\theta_F}))}\\
& \quad \lesssim n^{-\min\{\frac{3}{2}+\theta_F-\inv{\tau}-\eps,1\}}
(1+\n U\n_{L^{\infty}(0,T_0;L^p(\Omega;X))}),
\end{align*}
with implied constants independent of $n$, $x_0$ and $T_0$. For the penultimate
estimate we used \eqref{analyticSGIdiff}.\par
Combining these estimates, applying \eqref{UinVestimate}, and recalling the assumptions  
$\eta\le 1$ and $\eta < \frac32-\frac1\tau -\e + \theta_F$,  we obtain:
\begin{equation}\label{Ves3}
\begin{aligned}
&\Big\n  s\mapsto \int_{0}^{s}[S(s-u)-S(\ov{s}-\un{u})]F(u,U(u))\,du
\Big\n_{\Winfp{\alpha}([0,T_0]\times \Omega;X)}\\ 
& \qquad \qquad \lesssim n^{-\min\{\frac{3}{2}+
\theta_F-\inv{\tau}-\eps,1\}}(1+\n U\n_{L^{\infty}(0,T_0;L^p(\Omega;X))})\\
& \qquad \qquad \lesssim  n^{-\eta}(1+\n x_0\n_{L^p(\Omega;X)}),
\end{aligned}
\end{equation}
with implied constants independent of $n$, $x_0$ and $T_0$.
\addtocontents{toc}{\SkipTocEntry}\subsection*{Part 1d.}
Concerning the fourth term on the right-hand side in \eqref{aph1} we first apply Theorem
\ref{t:KW} and Lemma \ref{lem:analyticRbound}
(2) (with exponent $\frac12\eps$)
and then apply Lemma \ref{lem:detConv} (with $Y= X_{\frac12\eps }$,
$\delta= \theta_F-\eps$ and
$\Phi(u)=S(u-\un{u})[F(u,U(u))-F(u,U^{(n)}(u))]$).
Observe that $\Phi \in
L^{\infty}(0,T;L^p(\Omega;X))$ by  the
fact that both $U$ and 
$U^{(n)}$ belong to $\Winfp{\alpha}([0,T]\times\Omega;X)$, \MF{}, 
and the uniform boundedness of $u\mapsto S(u-\un{u})$
in $\calL(X_{\theta_F}, X_{\theta_F-\frac12\e})$. 
We obtain: 
\begin{equation}\label{Ves4}
\begin{aligned}
&\Big\n  s\mapsto
S(\ov{s}-s)\int_{0}^{s}S(s-\un{u})[F(u,U(u))-F(u,U^{(n)}(u))]\,du
\Big\n_{\Winfp{\alpha}([0,T_0]\times
\Omega;X)}\\
&\quad  \lesssim \Big\n  s\mapsto
\int_{0}^{s}S(s-\un{u})[F(u,U(u))-F(u,U^{(n)}(u))]\,du
\Big\n_{\Winfp{\alpha}([0,T_0]\times
\Omega;X_{\frac12\eps })}\\
&\quad  \lesssim (T_0^{1-(\theta_F-\eps)^-}+T_0^{\inv{2}-\alpha})\\
&\quad \quad \qquad \times\n u\mapsto
S(u-\un{u})[F(u,U(u))-F(u,U^{(n)}(u))]\n_{L^{\infty}(0,T_0;L^p(\Omega;X_{
\theta_F-\frac12\eps }))}&\\
& \quad  \lesssim (T_0^{1-(\theta_F-\eps)^-}+T_0^{\inv{2}-\alpha})\n U -
U^{(n)}\n_{L^{\infty}(0,T_0;L^p(\Omega;X))},
\end{aligned}
\end{equation}
with implied constants independent of $n$ and $T_0$.
\addtocontents{toc}{\SkipTocEntry}\subsection*{Part 1e.}
For the fifth term on the right-hand side in \eqref{aph1} we proceed as in Part 1c.
Using \eqref{Sdiffdecomp}, Lemma
\ref{lem:analyticRbound}, part (2) (with exponent $\inv{3}\eps$) and part (3) (with
exponent $\inv{2}+\theta_G-\frac23\eps$), and
Theorem \ref{t:KW}, we obtain:
\begin{equation}\label{h1ea}
\begin{aligned}
&\Big\n  s\mapsto \int_{0}^{s}[S(s-u)-S(\ov{s}-\un{u})]G(u,U(u))\,dW_H(u)
\Big\n_{\Winfp{\alpha}([0,T_0]\times
\Omega;X)}\\
& \lesssim   
n^{-\min\{\inv{2}+\theta_G-\frac23\eps,1\}} \Big\n s\mapsto\! 
\int_{0}^{s}\!S(s-u)G(u,U(u))\,dW_H(u) 
\Big\n_{\Winfp{\alpha}([0,T_0]\times \Omega;X_{\inv{2}+\theta_G-\frac23\eps })}\\
& \qquad + \Big\n s\mapsto \int_{0}^{s}S(s-u)(I-S(u-\un{u}))G(u,U(u))\,dW_H(u)
\Big\n_{\Winfp{\alpha}([0,T_0]\times
\Omega;X_{\frac13\eps})}.
\end{aligned}
\end{equation}\par
Now we apply Lemma \ref{lem:stochConv} to the two terms on the 
right-hand side of \eqref{h1ea}. For the
first term we apply Lemma
\ref{lem:stochConv} with $Y= X_{\inv{2}+\theta_G-\frac23\eps }$,
$\delta=-\inv{2}+\frac23\eps $
and $\Phi(u)=G(u,U(u))$, noting that
$\alpha>\inv{2}-\frac23\eps = -\delta$. Assumption \eqref{lem_stochConv_cond} is 
satisfied due to \eqref{GLipschitzV1} and the fact that 
$U\in \Winfp{\alpha}([0,T]\times \Omega;X)$. 
By Lemma \ref{lem:stochConv} and \eqref{GLipschitzV1} we obtain:
\begin{align*}
&\Big\n s\mapsto \int_{0}^{s}S(s-u)G(u,U(u))\,dW_H(u)
\Big\n_{\Winfp{\a}([0,T_0]\times
\Omega;X_{\inv{2}+\theta_G-\frac23\eps })} \\
& \qquad \lesssim \sup_{s\in [0,T_0]} \n u\mapsto (s-u)^{-\a} 
G(u,U(u))\n_{L^p(\Omega;\gamma(0,s;X_{\theta_G}))}\\
& \qquad \lesssim 1+\n U\n_{\Winfp{\alpha}([0,T_0]\times \Omega;X)},
\end{align*}
with implied constants independent of $n$, $x_0$ and $T_0$.\par

For the second term we take $Y=X_{\frac13\eps}$, $\delta= -\inv{2}+\frac23\e$ and
$\Phi(u)=
(I-S(u-\un{u}))G(u,U(u))$ in Lemma \ref{lem:stochConv}, noting that
$\alpha>\inv{2}-\frac23\eps =-\d$; 
assumption  \eqref{lem_stochConv_cond} is satisfied
because of $U\in \Winfp{\alpha}([0,T_0]\times \Omega;X)$, \eqref{GLipschitzV1}, 
and the fact that the operators $I-S(u-\un{u})$ are $\g$-bounded from $X_{\theta_G}$
to $X_{-\frac12+\e}$ by Lemma
\ref{lem:analyticRbound} (3). 

By Lemma \ref{lem:stochConv}, Lemma \ref{lem:analyticRbound} (3) (applied
with exponent $\frac12+\theta_G -\eps$), and \eqref{GLipschitzV1} we obtain:
\begin{align*}
& \Big\n s\mapsto \int_{0}^{s}S(s-u)(I-S(u-\un{u}))G(u,U(u))\,dW_H(u)
\Big\n_{\Winfp{\alpha}([0,T_0]\times
\Omega;X_{\frac13\eps})} \\
& \qquad \lesssim \sup_{s\in [0,T_0]} \n u\mapsto(s-u)^{-\a} 
(I-S(u-\un{u}))G(u,U(u))\n_{L^p(\Omega;\gamma(0,s;X_{-\inv{2}+\eps}))} \\
& \qquad\lesssim n^{-\min\{\inv{2}+\theta_G-\eps,1\}} \sup_{s\in [0,T_0]}\n u\mapsto
(s-u)^{-\a}G(u,U(u))\n_{L^p(\Omega;\gamma(0,s;X_{\theta_G}))} \\
& \qquad \lesssim n^{-\eta} (1+\n U\n_{\Winfp{\alpha}([0,T_0]\times \Omega;X)}),
\end{align*}
with implied constants independent of $n$, $x_0$ and $T_0$.\par
Combining these estimates and applying \eqref{UinVestimate} we obtain:
\begin{equation}\label{Ves5}
\begin{aligned}
&\Big\n  s\mapsto \int_{0}^{s}[S(s-u)-S(\ov{s}-\un{u})]G(u,U(u))\,dW_H(u)
\Big\n_{\Winfp{\alpha}([0,T_0]\times
\Omega;X)}\\
& \qquad \lesssim n^{-\eta} (1+\n U\n_{\Winfp{\alpha}([0,T_0]\times \Omega;X)}) \\
& \qquad \lesssim n^{-\eta} (1+\n x_0\n_{L^p(\Omega;X)}),
\end{aligned}
\end{equation}
with implied constants independent of $n$, $x_0$ and $T_0$.\par
\addtocontents{toc}{\SkipTocEntry}\subsection*{Part 1f.} 
For the final term in 
\eqref{aph1} we proceed as in Part 1d. 
First we apply Theorem \ref{t:KW} in
combination with Lemma \ref{lem:analyticRbound} (2) (with exponent $\frac14\eps$) to get rid
of the term $S(\ov{s}-s)$. Then we
apply Lemma \ref{lem:stochConv} (with $Y=X_{\frac14\eps}$, $\delta=\theta_G-\inv{2}\eps$,
and $\Phi=
S(u-\un{u})G(u,U(u))-G(u,U^{(n)}(u))$). Note that $\a> \inv{2}\eps - \theta_G = -\delta$.
Assumption \eqref{lem_stochConv_cond} is satisfied because $U$ and $U^{(n)}$ are in 
$\Winfp{\alpha}([0,T_0]\times \Omega;X)$, condition \MG{} holds,
and the operators $S(u-\un{u})$ are $\g$-bounded from $X_{\theta_G}$ to
$X_{\theta_G-\frac14\e}$. Finally, we apply Theorem
\ref{t:KW} again in combination with
Lemma
\ref{lem:analyticRbound} (2) (with exponent $\frac14\eps$) to get rid of the
term $S(u-\un{u})$. We obtain that
there exists an $\epsilon>0$, independent of $T_0\in[0,T]$, such that:
\begin{equation}\label{Ves6}
\begin{aligned}
&\Big\n  s\mapsto
\int_{0}^{s}S(\ov{s}-\un{u})[G(u,U(u))-G(u,U^{(n)}(u))]\,dW_H(u)
\Big\n_{\Winfp{\alpha}([0,T_0]\times
\Omega;X)}\\
& \qquad \lesssim \Big\n  s\mapsto
\int_{0}^{s}S(s-\un{u})[G(u,U(u))-G(u,U^{(n)}(u))]\,dW_H(u)
\Big\n_{\Winfp{\alpha}([0,T_0]\times \Omega;X_{\frac14\eps})}\\
& \qquad \lesssim T_0^{\epsilon}\sup_{0\leq s\leq T_0}\n  u\mapsto
(s-u)^{-\alpha}S(u-\un{u})\\
& \qquad \qquad \qquad \qquad \times [G(u,U(u))-G(u,U^{(n)}(u))]
\n_{L^p(\Omega;\gamma(0,s;H,X_{\theta_G-\frac14\eps}))}\\
& \qquad \lesssim T_0^{\epsilon}\sup_{0\leq s\leq T_0}\n  s\mapsto
(s-u)^{-\alpha} [G(u,U(u))-G(u,U^{(n)}(u))]
\n_{L^p(\Omega;\gamma(0,s;H,X_{\theta_G}))}\\
& \qquad \lesssim T_0^{\epsilon} \n U-U^{(n)} \n_{\Winfp{\alpha}([0,T_0]\times
\Omega;X)},
\end{aligned}
\end{equation}
where the last step used \eqref{GLipschitzV2}; 
the implied constants are independent of $n$ and $T_0$.
\addtocontents{toc}{\SkipTocEntry}\subsection*{Part 2.} Substituting
\eqref{Ves2}, \eqref{Ves1}, \eqref{Ves3}, \eqref{Ves4}, \eqref{Ves5},
\eqref{Ves6}
into \eqref{aph1} we obtain that there exists an exponent
$\epsilon_0>0$ and a constant $C>0$, both of which are independent of $n$, $x_0$, and $y_0$, such
that for all $T_0\in [0,T]$ we have:
\begin{equation}\label{comb}\begin{aligned}
\n U -U^{(n)}\n_{\Winfp{\alpha}([0,T_0]\times \Omega;X)} & \leq C \big(\n
x_0-y_0\n_{L^p(\Omega;X)}
 +  n^{-\eta}(1+\n x_0\n_{L^p(\Omega;X_{\eta})})\big) \\
& + C T_0^{\epsilon_0} \n U - U^{(n)} \n_{\Winfp{\alpha}([0,T_0]\times
\Omega;X)}. \end{aligned}
\end{equation}\par
From now on we fix $T_0:=\min\{(2C)^{1/\epsilon_0},T\}$. Note that $T_0$ is
independent of $n$, $x_0$, $y_0$, and we have: 
\begin{equation}\label{Vest}\begin{aligned}
&\n U -U^{(n)}\n_{\Winfp{\alpha}([0,T_0]\times \Omega;X)}  \\
& \qquad \leq 2C\n x_0-y_0\n_{L^p(\Omega;X)} +   2Cn^{-\eta} (1+\n
x_0\n_{L^p(\Omega;X_{\eta})}).\end{aligned}
\end{equation}
\addtocontents{toc}{\SkipTocEntry}\subsection*{Part 3.}
Let us fix $n\in \N$ and pick $t_0\in \{t_{j}^{(n)}\,:\,j=0,1,\ldots,n \}$. 
For $x\in L^p(\Omega,\calF_{t_0};X)$ we denote by $U(x,t_0,\cdot)$ the process in
$\Winfp{\alpha}([t_0,t_0+T]\times\Omega;X)$ satisfying, almost surely for all $s\in
[t_0,t_0+T]$:
\begin{align*}
U(x,t_0,s) = S(t-t_0)x
& + \int_{t_0}^{t}S(t-t_0-s)F\big(s,U(x,t_0,s)\big)\,ds \\
& + \int_{t_0}^{t}S(t-t_0-s)G\big(s,U(x,t_0,s)\big)\,dW_H(s).
\end{align*}
By $U^{(n)}(x,t_0,\cdot)$ we denote the process obtained from the modified
splitting scheme initiated in $t_0$ with
initial value $x\in L^p(\Omega,\calF_{t_0};X)$. Thus, almost surely for $t\in [t_0,t_0+T]$: 
\begin{align*}
U^{(n)}(x,t_0,t) = S(\ov{t}-t_0)x & + \int_{t_0}^{t}S(\ov{t}-t_0-\un{s})F\big(s,U^{(n)}(x,
t_0,s)\big)\,ds \\
& + \int_{t_0}^{t}S(\ov{t}-t_0-\un{s})G\big(s,U^{(n)}(x,t_0,s)\big)\,dW_H(s).
\end{align*}
From the proof of \eqref{Vest} it follows that for any $x\in
L^p(\Omega,\calF_{t_0};X_{\eta})$ and $y\in
L^p(\Omega,\calF_{t_0};X)$ we have:
\begin{equation}\label{Vest.shift}
\begin{aligned}
&\n U(x,t_0,\cdot) - U^{(n)}(y,t_0,\cdot)\n_{\Winfp{\alpha}([t_0,t_0+T_0]\times
\Omega;X)}  \\
& \qquad \leq 2C\n x-y\n_{L^p(\Omega;X)} +   2Cn^{-\eta} (1+\n
x\n_{L^p(\Omega;X_{\eta})}),
\end{aligned}
\end{equation}
with $C$ as in \eqref{Vest}.
\addtocontents{toc}{\SkipTocEntry}\subsection*{Part 4.} 
Let $T_0$ be as in Part 2 and fix $N\in \N$ large enough such that $\frac{T}{N}\le T_0$.
Let $M = \lceil 2T/T_0 \rceil$. Then $M\ge 2$ and $2T\le MT_0\le 2T+T_0 \le 3T$.  

Let us now fix $n\ge N$. Then $\frac12 T_0\leq \max\{T_0-\frac{T}{n},\frac{T}{n}\}$ 
and therefore $ \un{T}_0 \ge \frac12 T_0$. Hence, 
$T\le  M\un{T}_0 \le 3T$.

From now on we fix an integer $n\geq N$. By the uniqueness of the mild solution to \eqref{SDE} and by 
the definition of $U^{(n)}$ we have, for any $s_0,t_0\in \{
t_{j}^{(n)}\,:\,j=0,1,\ldots,M\}$, any $x\in L^p(\Omega,\mathcal{F}_{s_0};X)$ and
any $t\in [t_0,t_0+T_0]$ that:
\begin{align*}
U(x,s_0,t)& = U\big(U(x,s_0,t_0),t_0,t \big);\\
U^{(n)}(x,s_0,t)& = U^{(n)}\big(U^{(n)}(x,s_0,t_0),t_0,t \big).
\end{align*}
For $j\in\{ 1,\dots,M\}$, from 
\eqref{Vest.shift} (with $x=U(x_0,(j-1)\un{T}_0)$ and $y=U^{(n)}(y_0,(j-1)\un{T}_0)$)
we obtain: 
\begin{equation}\label{recur2}\begin{aligned}
& \n U(x_0,0,j\un{T}_0)-U^{(n)}(y_0,0,j\un{T}_0)\n_{L^p(\Omega;X)} \\
& \qquad = \big\n U\big(U(x_0,0,(j-1)\un{T}_0),(j-1)\un{T}_0,\un{T}_0\big)\\
& \qquad \qquad \qquad
-U^{(n)}\big(U^{(n)}(y_0,0,(j-1)\un{T}_0),(j-1)\un{T}_0,\un{T}_0\big)
\big\n_{L^p(\Omega;X)}\\
& \qquad \lesssim \n
U(x_0,0,(j-1)\un{T}_0)-U^{(n)}(y_0,0,(j-1)\un{T}_0)\n_{L^p(\Omega;X)}\\
& \qquad \qquad + n^{-\eta} (1+\n
U(x_0,(j-1)\un{T}_0)\n_{L^p(\Omega;X_{\eta})}),\end{aligned}
\end{equation}
with implied constants independent of $j$, $n$, $x_0$, $y_0$.\par

By \eqref{UinVestimate} we have
\begin{equation}\label{Vdeltaest}
\begin{aligned}
\sup_{1\leq j\leq M} \n U(x_0,0,j\un{T}_0) \n_{L^p(\Omega;X_{\eta})} & \leq
\sup_{s\in [0,3T]}\n U(x_0,0,s)
\n_{L^p(\Omega;X_{\eta})}\\
 &  \lesssim 1+\n x_0 \n_{L^p(\Omega;X_{\eta})},
\end{aligned}
\end{equation}
and therefore, by \eqref{recur2}:
\begin{align*}
& \n U(x_0,0,j\un{T}_0)-U^{(n)}(y_0,0,j\un{T}_0)\n_{L^p(\Omega;X)} \\
& \qquad \qquad  \lesssim \n
U(x_0,0,(j-1)\un{T}_0)-U^{(n)}(y_0,0,(j-1)\un{T}_0)\n_{L^p(\Omega;X)}\\
& \qquad \qquad \qquad + n^{-\eta} (1+\n x_0 \n_{L^p(\Omega;X_{\eta})}),
\end{align*}
with implied constants independent of $j$ and $n$. By induction we obtain:
\begin{equation}\label{recur3}\begin{aligned}
&\sup_{1\leq j\leq M} \n
U(x_0,0,j\un{T}_0)-U^{(n)}(y_0,0,j\un{T}_0)\n_{L^p(\Omega;X)} \\
& \qquad\qquad  \lesssim \n x_0 - y_0 \n_{L^p(\Omega;X)} + n^{-\eta} (1+\n x_0
\n_{L^p(\Omega;X_{\eta})}),\end{aligned}
\end{equation}
with implied constants independent of $j$ and $n$ as $M$ is independent of
$n$.\par

The estimate \eqref{recur3} is precisely what we need to extend \eqref{Vest} to the
interval $[0,T]$. To do so, we once
again fix $j\in \{ 1,\ldots,M\}$. Set $$x=U(x_0,0,(j-1)\un{T}_0)\quad
\textrm{and} \quad
y=U^{(n)}(y_0,0,(j-1)\un{T}_0)$$ in \eqref{Vest.shift} to obtain, using
\eqref{Vdeltaest} and \eqref{recur3}: 
\begin{align*}
&\big\n U(x_0,0,\cdot)-U^{(n)}(y_0,0,\cdot)\big\n_{\Winfp{\alpha}([(j-1)\un{T}_0,j\un{T}_0]\times \Omega;X)} \\
& \qquad = \big\n U\big(U(x_0,0,(j-1)\un{T}_0),(j-1)\un{T}_0,\cdot\big) \\
& \qquad \qquad 
-U^{(n)}\big(U^{(n)}(y_0,0,(j-1)\un{T}_0),(j-1)\un{T}_0,\cdot
\big)\big\n_{\Winfp{\alpha}([(j-1)\un{T}_0,j\un{T}_0]\times \Omega;X)} \\
& \qquad \lesssim   \n
U(x_0,0,(j-1)\un{T}_0)-U^{(n)}(y_0,0,(j-1)\un{T}_0)\n_{L^p(\Omega;X)} \\
&\qquad \qquad + n^{-\eta} (1+\n U(x_0,0,(j-1)\un{T}_0)
\n_{L^p(\Omega;X_{\eta})})\\
& \qquad \lesssim   \n x_0 - y_0 \n_{L^p(\Omega;X)} +n^{-\eta} (1+\n x_0
\n_{L^p(\Omega;X_{\eta})}),
\end{align*}
with implied constants independent of $j$ and $n$.\par

Due to inequality \eqref{Vtransinv} we thus obtain: 
\begin{align*}
&\n U - U^{(n)} \n_{\Winfp{\alpha}([0,T]\times \Omega;X)} \\
& \qquad \leq  \sum_{j=1}^{M}  \big\n
U\big(U(x_0,0,(j-1)\un{T}_0),(j-1)\un{T}_0,\cdot\big) \\
& \qquad \qquad -U^{(n)}\big(U^{(n)}(y_0,0,(j-1)\un{T}_0),(j-1)\un{T}_0,\cdot
\big)
\big\n_{\Winfp{\alpha}([(j-1)\un{T}_0,j\un{T}_0]\times \Omega;X)}\\
& \qquad \lesssim \sum_{j=1}^{M}  \n x_0 - y_0 \n_{L^p(\Omega;X)} +n^{-\eta}
(1+\n x_0
\n_{L^p(\Omega;X_{\eta})})\\
& \qquad \lesssim  \n x_0 - y_0 \n_{L^p(\Omega;X)} +n^{-\eta} (1+\n x_0
\n_{L^p(\Omega;X_{\eta})})
\end{align*}
since $M$ is independent of $n$. This
proves estimate \eqref{Wconvrate}.
\end{proof}\par

In the next subsection we shall need the following corollary of Theorem
\ref{thm:convV}:
\begin{corollary}
\label{cor:Lpest}
Let the setting be as in Theorem \ref{thm:convV}. Let $0\le \delta <\eta_{\rm max}$  and
$p\in[2,\infty)$,
and assume that $y_0\in L^p(\Omega;X_{\delta})$. Then for all $\alpha\in [0,\inv{2})$ one
has:
\begin{align*}
& \sup_{n\in \N}\n U^{(n)} \n_{\Winfp{\alpha}([0,T]\times \Omega;X_{\delta})}
\lesssim 1 + \n y_0 \n_{L^p(\Omega;X_{\delta})}.
\end{align*}
\end{corollary}
\begin{proof}
By assumption one can pick $\eps>0$ such that $\delta+ \eps <
\eta_{\rm max}$. Because
$\theta_F>\d-1+(\frac1\tau -\frac{1}{2})+\eps$, the restriction of 
$F: [0,T] \times X\to  X_{\theta_F}$  to $[0,T]\times X_\d$ induces a mapping
 $F: [0,T] \times
X_{\delta}\rightarrow
X_{\d-1+(\frac1\tau -\frac{1}{2})+\eps}$ which satisfies \MF{} with
$\tilde\theta_F=-1+(\frac1\tau -\frac{1}{2})+\eps$.
Similarly, from $\theta_G>\d-\inv{2}+\eps$ we obtain a mapping
$G: [0,T] \times X_{\delta}\rightarrow
X_{\d-\inv{2}+\eps}$ which satisfies \MG{} with $\tilde \theta_G=-\inv{2}+\eps$. The
desired result is now obtained by combining
Theorem \ref{thm:convV} (with state space $X_{\delta}$, initial conditions $x_0=y_0$, 
 and exponent $\eta=0$) and Theorem \ref{thm:NVW08}.
\end{proof}
\subsection{Convergence of the classical splitting scheme}
We consider the stochastic differential equation \eqref{SEE} under the assumptions
\MA{}, \MF{}, \MG{}, with initial value $x_0$, under the
additional assumption that $\theta_F,\theta_G \geq 0$. 

For $n\in\N$ let $(\tilde U_j^{(n)})_{j=0}^n$ be defined by \eqref{SDE_noA}
and set 
\begin{align*} 
\tilde U^{(n)}(t) := \sum_{j=1}^n 1_{I_j^{(n)}}(t)\, \tilde U_j^{(n)}(t), \quad t\in [0,T). 
\end{align*}
As before, $U$ denotes the mild solution of \eqref{SEE} with initial value 
$x_0$.
\begin{theorem}\label{thm:convVclassical}
Let $0\leq \eta<\eta_{\max}$, and suppose $x_0\in L^p(\calF_0,X_{\eta})$ and 
$\tilde{y}_0\in L^p(\calF_0,X)$ for some $p\in [2,\infty)$. Then for all $\alpha\in [0,\inv{2})$ we
have:
\begin{equation}\label{Wconvrate_mod}\begin{aligned}
\n U - \tilde U^{(n)} \n_{\Winfp{\alpha}([0,T]\times \Omega;X)} &\lesssim \n x_0-\tilde{y}_0
\n_{L^p(\Omega;X)} +  n^{-\eta} \big(1+\n x_0\n_{L^p(\Omega;X_{\eta})}\big), 
\end{aligned}
\end{equation}
with implied constants independent of $n$, $x_0$ and $\tilde{y}_0$.
\end{theorem}
\begin{proof}
Fix $T>0$, $n\in \N$. For $\Ucla$ the following relation holds (see also \eqref{splitConv}):
\begin{align}\label{splitConvCla}
\Ucla(s) = S(\ov{s})\tilde{y}_0 &+ \int_{0}^{s} S(\un{s}-\un{u})F(u,\Ucla(u))\,du\\
& + \int_0^{s}S(\un{s}-\un{u})G(u,\Ucla(u))\,dW_H(u).
\end{align}
At first sight the processes $\Ucla$ and $\Umod$ are very similar, and one would expect 
the proof of Theorem \ref{thm:convVclassical} to be entirely analogous to the proof of 
Theorem \ref{thm:convV}. However, there is a subtle difficulty when considering $\Ucla$: 
for the proof of Theorem \ref{thm:convV} we make use of the fact that $\ov{s}-\un{u}\geq s-u$ 
for all $0\leq u\leq s$, $s\in [0,T]$. This allows us to write
\begin{align}\label{Sconv}
S(\ov{s}-\un{u}) & = S(\ov{s}-s)S(s-u)S(u-\un{u})
\end{align}
and (see \eqref{Sdiffdecomp}):
\begin{equation}\label{SdiffdecompRepeat}
S(s-u)-S(\ov{s}-\un{u})= (I-S(\ov{s}-s))S(s-u) +
S(\ov{s}-s)S(s-u)(I-S(u-\un{u})),
\end{equation}
As a result, we can interpret the (deterministic and stochastic) integral terms in 
\eqref{aph1} as (stochastic) convolutions and use Lemmas \ref{lem:detConv} and 
\ref{lem:stochConv} to obtain estimates for these terms.\par
For $\Ucla$ one of the difficulties lies in the fact that for 
$s\in [0,T]\setminus\{t_j^{(n)}\,:\,j=1,\ldots,n\}$ fixed we have $\un{s}-\un{u}>s-u$ 
for some values of $u\in [0,s]$, but $\un{s}-\un{u}< s-u$ for other values of $u\in[0,s]$. 
Instead of \eqref{Sconv} we have, for $s-u\geq \frac{T}{n}$:
\begin{align}\label{Sconv2}
S(\un{s}-\un{u}) & = S(\ov{s}-s)S(s-\tfrac{T}{n}-u)S(u-\un{u}).
\end{align}
Roughly speaking, this allows one to apply Lemmas \ref{lem:detConv} and \ref{lem:stochConv} 
on the interval $[0,(s-\frac{T}{n})^+]$. However, an extra argument is needed for the 
remainder of the interval.\par
Another difficulty in dealing with $\Ucla$ is that for 
$s\in [0,T]\setminus\{t_j^{(n)}\,:\,j=1,\ldots,n\}$ and $u\in [\un{s},s]$ we have 
$S(\un{s}-\un{u})=I$, and thus we cannot use the smoothing property of the semigroup there. 
Note that this occurs precisely in the aforementioned remainder.
\addtocontents{toc}{\SkipTocEntry}\subsection*{Part 1.}
It is easier to deal with the remainder if we compare $\Ucla$ with $\Umod$ instead 
of comparing $\Ucla$ with $U$: by Theorem \ref{thm:convV} suffices to prove that
\begin{equation}\label{UmodUcla}\begin{aligned}
\n \Umod - \Ucla \n_{\Winfp{\alpha}([0,T]\times \Omega;X)} 
&\lesssim \n y_0-\tilde{y}_0 \n_{L^p(\Omega;X)}+  n^{-\eta} 
\big(1+\n y_0\n_{L^p(\Omega;X_{\eta})}\big), \end{aligned}
\end{equation}
with implied constants independent of $n$, $y_0$ and $\tilde{y}_0$, where $\Umod$ denotes 
the process obtained by applying the modified splitting scheme to \eqref{SDE} with initial 
value $y_0\in L^{p}(\Omega,\F_0;X_{\eta})$.\par
By \eqref{splitConv} and \eqref{splitConvCla} we have:
\begin{equation}\begin{aligned}\label{splitclas}
\Umod(s)-\Ucla(s)  & = S(\ov{s})(y_0-\tilde{y}_0)\\
&\quad +  (S(\tfrac{T}{n})-I)\int_{0}^{s} S(\un{s}-\un{u})F(u,\Umod(u))\,du \\
& \quad + \int_{0}^{s} S(\un{s}-\un{u})\big[F(u,\Umod(u))-F(u,\Ucla(u))\big]\,du\\
&\quad + (S(\tfrac{T}{n})-I)\int_{0}^{s}S (\un{s}-\un{u})G(u,\Umod(u))\,dW_H(u)  \\
& \quad + \int_{0}^{s} S(\un{s}-\un{u})\big[G(u,\Umod(u))-G(u,\Ucla(u))\big]\,dW_H(u).
\end{aligned}\end{equation}\par
As mentioned above, we can rewrite each of the (deterministic and stochastic) 
integrals above as a (deterministic or stochastic) convolution and a remainder term. 
Below, we will demonstrate this for the first deterministic integral term in 
\eqref{splitclas}. 
The convolutions can be dealt with in the same manner as in the proof of Theorem 
\ref{thm:convV}, and in Part 2 of this proof we will demonstrate how to deal with 
the remainder.\par

For the first deterministic integral in \eqref{splitclas} we have, by \eqref{Sconv2}:
\begin{align*}
& (S(\tfrac{T}{n})-I)\int_{0}^{s} S(\un{s}-\un{u})F(u,\Umod(u))\,du \\
& \qquad = (S(\tfrac{T}{n})-I)S(\ov{s}-s)\int_0^{(s-\frac{T}{n})^+}S((s-\tfrac{T}{n})^+-u)
S(u-\un{u})F(u,\Umod(u))\,du\\
& \qquad \quad + (S(\tfrac{T}{n})-I)\int_{(s-\frac{T}{n})^+}^{s} S(\un{s}-\un{u}) F(u,\Umod(u))\,du
\end{align*}
Note that the first term on the right-hand side above involves the convolution of the 
process $$u\mapsto S(u-\un{u})F(u,\Umod(u))$$ with the semigroup $S$, evaluated in
 $(s-\frac{T}{n})^+$. By arguments analogous to Part 1c in the proof of Theorem 
\ref{thm:convV} we can estimate this term, using Corollary \ref{cor:Lpest} where 
in Part 1c the estimate of Theorem \ref{thm:NVW08} is applied:
\begin{align*}
& \Big\n s\mapsto (S(\tfrac{T}{n})-I)S(\ov{s}-s)\\
& \qquad \qquad \times \int_0^{(s-\frac{T}{n})^+}S((s-\tfrac{T}{n})^+-u)
S(u-\un{u})F(u,\Umod(u))\,du \Big\n_{\Winfp{\alpha}([0,T_0]\times\Omega;X)} \\
& \qquad \qquad \qquad \lesssim n^{-\eta}(1+\n y_0\n_{L^p(\Omega;X_{\eta})}.
\end{align*}
\par
For the remainder term 
we apply, for the time being, only the ideal property for $\gamma$-radonifying 
operators \eqref{eq:ideal} to get rid of the term $S(\frac{T}{n})-I$. We thus obtain,
 for all $T_0\in [0,T]$:
\begin{align*}
& \Big\n s\mapsto (S(\tfrac{T}{n})-I) \int_{0}^{s} S(\un{s}-\un{u})F(u,\Umod(u))\,du 
\Big\n_{\Winfp{\alpha}([0,T_0]\times
\Omega;X)}
\\ & \quad \lesssim n^{-\eta} (1 + \n y_0\n_{L^p(\Omega;X)} ) \\
& \quad \quad +  n^{-(\theta_F \wedge 1)}\Big\n s\mapsto \int_{(s-\frac{T}{n})^+}^{s} S(\un{s}-\un{u}) 
F(u,\Umod(u))\,du \Big\n_{\Winfp{\alpha}([0,T_0]\times \Omega;X_{\theta_F})}.
\end{align*}\par

By applying similar arguments to the other three integral terms in \eqref{splitclas} and by 
applying the argument of Part 1b in the proof of Theorem \ref{thm:convV} to the first 
term in \eqref{splitclas}, one obtains that there exists an $\eps_0\in (0,\inv{2})$ 
such that for $T_0\in [0,T]$ and $\alpha$ sufficiently large we have, setting $I_s:=[(s-\frac{T}{n})^+,s]$:
\begin{align}
& \n \Umod - \Ucla \n_{\Winfp{\alpha}([0,T_0]\times \Omega;X)} \notag\\
&\  \lesssim \n y_0-\tilde{y}_0 \n_{L^p(\Omega;X)}  +  n^{-\eta} \big(1+\n y_0\n_{L^p(\Omega;X_{\eta})}\big) + T_0^{\eps_0} \n \Umod - \Ucla \n_{\Winfp{\alpha}([0,T_0]\times \Omega;X)} \notag\\
& \  \quad+ n^{-(\theta_F\wedge 1)}\Big\n s\mapsto \int_{I_s} S(\un{s}-\un{u})F(u,\Umod(u))\,du \Big\n_{\Winfp{\alpha}([0,T_0]\times \Omega;X_{\theta_F})} \tag{i}\label{MC2}\\
& \  \quad + \Big\n s\mapsto\int_{I_s} S(\un{s}-\un{u})\big[F(u,\Umod(u))-F(u,\Ucla(u))\big]\,du\Big\n_{\Winfp{\alpha}([0,T_0]\times \Omega;X)}
 \tag{ii}\label{MC4}\\
&\  \quad + n^{-(\theta_G\wedge 1)}\Big\n s\mapsto\int_{I_s} S(\un{s}-\un{u})G(u,\Umod(u))\,dW_H(u) \Big\n_{\Winfp{\alpha}([0,T_0]\times \Omega;X_{\theta_G})} \tag{iii}\label{MC6}\\
& \  \quad + \Big\n s\mapsto\int_{I_s} S(\un{s}-\un{u})\big[G(u,\Umod(u))-G(u,\Ucla(u))\big]\,dW_H(u)\Big\n_{\Winfp{\alpha}([0,T_0]\times \Omega;X)}.
 \tag{iv}\label{MC8}
\end{align}
\addtocontents{toc}{\SkipTocEntry}\subsection*{Part 2.}
We now demonstrate how \eqref{MC2} - \eqref{MC8} can be estimated using the following 
two claims. The proofs of the claims are postponed to Parts 3 and 4.\par
\addtocontents{toc}{\SkipTocEntry}\subsection*{Claim 1.}
Let $\delta\in\R$, $\alpha\in [0,\inv{2})$ and 
$\Phi\in \Winfp{\alpha}([0,T]\times \Omega;X_{\delta})$. Then for all $\eps>0$ and all $T_0\in [0, T]$ we have:
\begin{align*}
\Big\n s\mapsto \int_{I_s} S(\un{s}-\un{u})\Phi(u)\, du 
\Big\n_{\Winfp{\alpha}([0,T_0]\times \Omega;X_{\delta})} 
\lesssim \big(\tfrac{T}{n}\minsym T_0\big)^{\frac32-\inv{\tau}-\eps}
\n \Phi\n_{L^{\infty}(0,T_0;L^p(\Omega,X_{\delta}))},
\end{align*}
with implied constant independent of $n$ and $T_0$.
\addtocontents{toc}{\SkipTocEntry}\subsection*{Claim 2.}
Let $\delta\in\R$, $\alpha\in [0,\inv{2})$, and 
$\Phi\in \Winfp{\alpha}([0,T]\times \Omega;X_{\delta})$. Then for all $T_0\in [0,T]$ we have:
\begin{align*}
&\Big\n s\mapsto \int_{I_s} S(\un{s}-\un{u})\Phi(u) \,dW_H(u) 
\Big\n_{\Winfp{\alpha}([0,T_0]\times \Omega;X_{\delta})}\\
&\qquad \lesssim \big(\tfrac{T}{n}\minsym T_0\big)^{\alpha}\sup_{0\leq t\leq T_0} 
\n s\mapsto (t-s)^{-\alpha}\Phi(s) \n_{L^{\infty}(0,t;L^p(\Omega,X_{\delta}))},
\end{align*}
with implied constant independent of $n$ and $T_0$.\par
Pick $\eps>0$ such that
$$ \eps< \tfrac32-\tinv{\tau}+\theta_F -\eta.$$ We shall apply Claim 1 with this choice of $\eps$. To be precise, for \eqref{MC2} we apply Claim 1 with $\Phi= F(\cdot,\Umod(\cdot))$ and $\delta=\theta_F$. For \eqref{MC4} we apply Claim 1 with $\Phi= F(\cdot,\Umod(\cdot))-F(\cdot,\Ucla(\cdot))$ 
and $\delta=0$.\par
Replacing $\alpha\in [0,\inv{2})$ by a larger value if necessary, we may assume 
$\eta-\theta_G<\alpha<\inv{2}$. For \eqref{MC6} we apply Claim 2 with $\Phi= G(\cdot,\Umod(\cdot))$ and 
$\delta=\theta_G$. Finally, for \eqref{MC8} we apply Claim 2 with 
$\Phi= G(\cdot,\Umod(\cdot))-G(\cdot,\Ucla(\cdot))$ and $\delta=0$. This gives:
\begin{align*}
& \n \Umod - \Ucla \n_{\Winfp{\alpha}([0,T_0]\times \Omega;X)} \\
&\  \lesssim \n y_0-\tilde{y}_0 \n_{L^p(\Omega;X)}  +  n^{-\eta} 
\big(1+\n y_0\n_{L^p(\Omega;X_{\eta})}\big)  + T_0^{\eps_0} \n \Umod - \Ucla 
\n_{\Winfp{\alpha}([0,T_0]\times \Omega;X)}\\
& \  \quad  + n^{-\eta}\n F(\cdot,\Umod(\cdot))
\n_{L^{\infty}(0,T_0;L^p(\Omega;X_{\theta_F}))}\\
& \ \quad  + T_0^{\frac32-\inv\tau-\eps} \n F(\cdot,\Umod(\cdot))- F(\cdot,\Ucla(\cdot))\n_{L^{\infty}(0,T_0;L^p(\Omega;X))}\\
& \ \quad  + n^{-\eta}\sup_{0\leq t\leq T_0} \n s\mapsto 
(t-s)^{-\alpha}G(u,\Umod(u))\n_{L^p(\Omega,\gamma(0,t;X_{\theta_G}))} \\
& \ \quad  + T_0^{\alpha-\inv{p}-\eps}\sup_{0\leq t\leq T_0} \n s\mapsto 
(t-s)^{-\alpha}[G(u,\Umod(u))-G(u,\Ucla(u))]\n_{L^p(\Omega,\gamma(0,t;X))}.
\end{align*}
Note that, as $\theta_F,\theta_G\geq 0$, we have continuous inclusions
$X_{\theta_F}\embed X$ and $X_{\theta_G}\embed X$, so that the 
norms in $L^p(\Omega;X)$ and $\gamma(0,t;X)$ may be estimated by the norms in 
$L^p(\Omega;X_{\theta_F})$ and $\gamma(0,t;X_{\theta_G})$
in the third and fifth line, respectively.\par
Applying assumption \MF{} and the estimates \eqref{GLipschitzV1} and \eqref{GLipschitzV2}, 
then Corollary \ref{cor:Lpest} (with $\delta=0$), 
we obtain that there exists an $\tilde{\eps}_0>0$ such that:
\begin{align*}
& \n \Umod - \Ucla \n_{\Winfp{\alpha}([0,T_0]\times \Omega;X)} \\
&\  \lesssim \n y_0-\tilde{y}_0 \n_{L^p(\Omega;X)}  +  n^{-\eta} 
\big(1+\n y_0\n_{L^p(\Omega;X_{\eta})}\big)  + T_0^{\tilde{\eps}_0} 
\n \Umod - \Ucla \n_{\Winfp{\alpha}([0,T_0]\times \Omega;X)}.
\end{align*}
The remainder of the proof is entirely analogous to Parts 3 and 4 in 
the proof of Theorem \ref{thm:convV}.
\addtocontents{toc}{\SkipTocEntry}\subsection*{Part 3: Proof of Claim 1}
Fix $\eps>0$. Recall that $I_s=[(s-\frac{T}{n})^+,s]$. Observe that for $s\in [0,T_0]$ we have:
\begin{align}\label{remainder_split}
\int_{I_s} S(\un{s}-\un{u})\Phi(u) \,du & = S(\tfrac{T}{n})\int_{(s-\frac{T}{n})^+}^{\un{s}} \Phi(u) \,du
 + \int_{\un{s}}^{s} \Phi(u) \,du.
\end{align}
Let $a,b:[0,T]\rightarrow \R$ be measurable and satisfy $a\leq b$. We shall prove that:
\begin{equation}\label{remainder_est_1}
\begin{aligned}
& \Big\n s\mapsto \int_{I_s} 1_{\{a(s)\leq u \leq b(s)\}}\Phi(u) \,du 
\Big\n_{\Winfp{\alpha}([0,T_0]\times \Omega;X_{\delta})} \\
& \qquad \qquad \lesssim 
\big(\tfrac{T}{n}\minsym T_0\big)^{\frac{3}{2}-\inv{\tau}-\eps}\n \Phi\n_{L^{\infty}(0,T_0;L^{p}(\Omega;X_{\delta}))}.
\end{aligned}\end{equation}\par
The claim follows by applying the above estimate with $a(s)= (s-\frac{T}{n})^+;\, b(s)=\un{s}$ 
to the first term in \eqref{remainder_split}, and with 
$a(s)=\un{s};\, b(s)=s$ to the second term. (The term $S(\frac{T}{n})$ can be estimated 
away by \eqref{eq:ideal}.)\par
For $s\in [0,T_0]$ we have $|I_s|=s-(s-\frac{T}{n})^+\leq \frac{T}{n}\minsym T_0$, and thus:
\begin{equation}\label{Cl1:Linfty}
\Big\n s\mapsto \int_{I_s} 1_{\{a(s)\leq u \leq b(s)\}}\Phi(u)\, du 
\Big\n_{L^{\infty}(0,T_0;L^{p}(\Omega;X_{\delta}))} \lesssim \big(\tfrac{T}{n}\minsym T_0\big)
\n \Phi\n_{L^{\infty}(0,T_0;L^{p}(\Omega;X_{\delta}))}.
\end{equation}\par

For the estimate in the $\gamma$-radonifying norm we shall use the 
Besov embedding of Section \ref{sec:besov} and Lemma \ref{lem:wBesovEst} in the Appendix. 
We begin with a simple observation. If $\Psi\in L^{\infty}(0,T_0;Y)$ 
for some Banach space $Y$ and $\alpha\in (0,1]$, then:
\begin{align*}
&\Big\n s\mapsto \int_{I_s} \Psi(u)\,du\Big\n_{C^{\alpha}(0,T_0;Y)}\\
& \quad \leq \sup_{s\in [0,T_0]}|I_s|\n \Psi\n_{L^{\infty}(0,T_0;Y)} 
+ \sup_{0\leq s < t\leq T_0}(t-s)^{-\alpha}
\Big\n \int_{I_t}\Psi(u)\,du -\int_{I_s} \Psi(u)\,du \Big\n_{Y}\\
& \quad \leq (\tfrac{T}{n}\minsym T_0)\n \Psi\n_{L^{\infty}(0,T_0;Y)} 
+ \sup_{0\leq s < t\leq T_0}(t-s)^{-\alpha}
\Big\n \int_{I_t}\Psi(u)\,du - \int_{I_s} \Psi(u)\,du \Big\n_{Y}.
\end{align*}\par

If $t-s\geq \frac{T}{n}$, then 
\begin{align*}
(t-s)^{-\alpha}\Big\n \int_{I_t}\Psi(u)\,du - \int_{I_s} \Psi(u)\,du \Big\n_{Y} 
&  \leq 2(t-s)^{-\alpha}\sup_{s\in [0,T_0]}|I_s|\n \Psi\n_{L^{\infty}(0,T_0;Y)}\\
& \leq 2(\tfrac{T}{n}\minsym T_0)^{1-\alpha}\n \Psi\n_{L^{\infty}(0,T_0;Y)}.
\end{align*}
On the other hand, if $t-s\leq \frac{T}{n}$, then:
\begin{align*}
& (t-s)^{-\alpha}\Big(\Big\n \int_{I_t}\Psi(u)\,du - \int_{I_s} \Psi(u)\,du \Big\n_{Y}\Big) \\
& \qquad \leq (t-s)^{-\alpha}\Big(\Big\n \int_{s}^{t}\Psi(u)\,du \Big\n_{Y}  +  
\Big\n \int_{(s-\frac{T}{n})^+}^{(t-\frac{T}{n})^+} \Psi(u)\,du \Big\n\Big)\\
& \qquad \leq 2(t-s)^{1-\alpha}\n \Psi\n_{L^{\infty}(0,T_0;Y)} 
\leq 2(\tfrac{T}{n}\minsym T_0)^{1-\alpha}\n \Psi\n_{L^{\infty}(0,T_0;Y)}.
\end{align*}
It follows that:
\begin{equation}
\begin{aligned}\label{HoelderIsEst}
\Big\n s\mapsto \int_{I_s} \Psi(u)\,du\Big\n_{C^{\alpha}(0,T_0;Y)}
&\leq 3(\tfrac{T}{n}\minsym T_0)^{1-\alpha}\n \Psi\n_{L^{\infty}(0,T_0;Y)}.
\end{aligned}
\end{equation}\par

Note that as $p\geq 2$ the type of $L^p(\Omega,X)$ is the same as the type $\tau$ of $X$. 
Without loss of generality we may assume that $\tau<2$. Fix $q\geq 2$ such that 
$\inv{q}<\inv{\tau}-\alpha$. By isomorphism \eqref{eq:gFub}, 
the Besov embedding \eqref{BesovEmbed}, and Lemma \ref{lem:wBesovEst} 
there exists an $\epsilon_0>0$ such that we have: 
\begin{align*}
& \sup_{t\in [0,T_0]}\Big\n s\mapsto (t-s)^{-\alpha}\int_{I_s} 1_{\{a(s)
\leq u \leq b(s)\}}\Phi(u)\, du 
\Big\n_{L^p(\Omega;\gamma(0,t;X_{\delta}))}\\
& \qquad \eqsim \sup_{t\in [0,T_0]}\Big\n s\mapsto (t-s)^{-\alpha}\int_{I_s} 
1_{\{a(s)\leq u \leq b(s)\}}\Phi(u)\, du 
\Big\n_{\gamma(0,t;L^p(\Omega;X_{\delta}))}\\
& \qquad \lesssim \sup_{t\in [0,T_0]}\Big\n s\mapsto (t-s)^{-\alpha} \int_{I_s} 
1_{\{a(s)\leq u \leq b(s)\}}\Phi(u)\, du 
\Big\n_{B_{\tau,\tau}^{\inv{\tau}-\inv{2}}(0,T_0;L^p(\Omega;X_{\delta}))}\\
& \qquad \lesssim T_0^{\epsilon_0} \Big\n s\mapsto \int_{I_s} 1_{\{a(s)\leq u 
\leq b(s)\}}\Phi(u)\, du 
\Big\n_{L^{\infty}(0,T_0;L^p(\Omega;X_{\delta}))}\\
& \qquad \quad + T_0^{\epsilon_0} \Big\n s\mapsto \int_{I_s} 1_{\{a(s)\leq u 
\leq b(s)\}}\Phi(u)\, du 
\Big\n_{B_{q,\tau}^{\inv{\tau}-\inv{2}}(0,T_0;L^p(\Omega;X_{\delta}))}\\
& \qquad \lesssim T_0^{\epsilon_0} \Big\n s\mapsto \int_{I_s} 1_{\{a(s)\leq u 
\leq b(s)\}}\Phi(u)\, du 
\Big\n_{C^{\inv{\tau}-\inv{2}+\eps}(0,T_0;L^p(\Omega;X_{\delta}))}\\
& \qquad \lesssim \big(\tfrac{T}{n} \wedge T_0\big)^{\frac32-\inv{\tau}-\eps}
\n \Phi\n_{L^{\infty}(0,T_0;X_{\delta})},
\end{align*}
where in the final estimate we used \eqref{HoelderIsEst}.
\addtocontents{toc}{\SkipTocEntry}\subsection*{Part 4: Proof of Claim 2}
As before let $a,b:[0,T]\rightarrow \R$ be measurable and satisfy $a\leq b$. 
It suffices to prove that:
\begin{equation}\label{remainder_est_2}
\begin{aligned} & \Big\n s\mapsto \int_{I_s} 1_{\{a(s)\leq u \leq b(s)\}}\Phi(u)\, dW_H(u) 
\Big\n_{\Winfp{\alpha}([0,T_0]\times \Omega;X_{\delta})}
 \\ & \hskip4cm
 \lesssim \big(\tfrac{T}{n}\minsym T_0\big)^{\alpha}
\n \Phi\n_{L^{\infty}(0,T_0;L^{p}(\Omega;X_{\delta}))}.
\end{aligned}\end{equation}\par

Note that for any $s\in [0,T_0]$ we have:
\begin{align*}
& \Big\n \int_{I_s} 1_{\{a(s)\leq u \leq b(s)\}}\Phi(u)dW_H(u) 
\Big\n_{L^p(\Omega;X_{\delta})} \leq \Big\n \int_{I_s} \Phi(u)dW_H(u) 
\Big\n_{L^p(\Omega;X_{\delta})}\\
&\quad  = \Big\n\int_{0}^{s} \Phi(u)dW_H(u) - 
\int_{0}^{(s-\frac{T}{n})^+} \Phi(u)dW_H(u) \Big\n_{L^p(\Omega;X_{\delta})} \\
&\quad \leq (s-(s-\tfrac{T}{n})^+)^{\alpha} \Big\n s \mapsto \int_{0}^{s} 
\Phi(u)dW_H(u) \Big\n_{C^{\alpha}(0,T_0;L^p(\Omega;X_{\delta}))}.
\end{align*}
Thus by \eqref{stochIntCont} we have:
\begin{equation}\label{Cl2:Linfty}
\begin{aligned}
& \Big\n s\mapsto \int_{I_s} 1_{\{a(s)\leq u \leq b(s)\}}\Phi(u)\, dW_H(u) 
\Big\n_{L^{\infty}(0,T_0;L^{p}(\Omega;X_{\delta}))} \\
& \qquad  \lesssim \sup_{0\leq r\leq T_0}|I_r|^{\alpha}  \Big\n s\mapsto
\int_{0}^{s} \Phi(u)\,dW_H(u)
\Big\n_{C^{\alpha}(0,T_0;L^p(\Omega;X_{\delta}))}\\
& \qquad \lesssim \big(\tfrac{T}{n}\minsym T_0\big)^{\alpha}
\sup_{0\leq t\leq T_0}\n   s  \mapsto (t-  s  )^{-\alpha} 
\Phi(s)\n_{L^p(\Omega,\gamma(0,t;H,X_{\delta}))}.
\end{aligned}
\end{equation}\par
Let $t\in [0,T_0]$. We wish to apply Lemma \ref{lem:h1} with $$f(r,u)(s) =
(t-s)^{-\alpha}(t-u)^{\alpha}1_{\{(s-\frac{T}{n})^+\leq u\leq s
\}}1_{\{a(s)\leq u \leq b(s)\}},$$ $R=[0,1]$, $S=[0,t]$ (both with the Lebesgue measure), 
$X_1=X_2=X_{\delta}$, $\Phi_2\equiv I$ and
$\Phi_1(u)=(t-u)^{-\a}\Phi(u)$. Note that:
\begin{align*}
\sup_{(r,u)\in [0,1]\times [0,t]} \n f(r,u)\n_{L^2(0,t)}^2 
&  \leq \sup_{u\in  [0,t]} (t-u)^{\alpha} 
\big\n s\mapsto (t-s)^{-\alpha}1_{\{(s-\frac{T}{n})^+\leq u\leq s\}} \big\n_{L^2(0,t)}^2\\
& \leq \sup_{u\in  [0,t]} (t-u)^{2\alpha} 
\int_u^{(u+\frac{T}{n})\wedge t} (t-s)^{-2\alpha}\,ds.
\end{align*}
If $t -u \ge \frac{2T}{n}$, then $t-(u+\frac{T}{n})\geq \inv{2}(t-u)$ 
and
$$ \int_u^{(u+\frac{T}{n})\wedge t} (t-s)^{-2\alpha}\,ds
\leq \tfrac{T}{n}(t-(u+\tfrac{T}{n}))^{-2\alpha} \leq  2^{2\alpha}\tfrac{T}{n}(t - u)^{-2\a},
$$
while if $t - u \le \frac{2T}{n}$, then
$$ \int_u^{(u+\frac{T}{n})\wedge t} (t-s)^{-2\alpha}\,ds
 \le \int_u^t (t-s)^{-2\alpha}\,ds = \tfrac{1}{1-2\alpha}(t-u)^{1-2\alpha}
\leq \tfrac{1}{1-2\alpha}\tfrac{2T}{n}(t-u)^{-2\alpha}. 
$$
In both cases, we also have the estimate 
$$\int_u^{(u+\frac{T}{n})\wedge t} (t-s)^{-2\alpha}\,ds
 \le \int_u^t (t-s)^{-2\alpha}\,ds \eqsim (t-u)^{1-2\a} \le T_0(t-u)^{-2\a}.
$$ 
Combining this with the previous estimates we find:
\begin{equation}\label{2cases} \sup_{(r,u)\in [0,1]\times [0,t]} \n f(r,u)\n_{L^2(0,t)}
\lesssim \big(\tfrac{T}{n} \wedge T_0\big)^{\frac12}.
\end{equation}
Thus Lemma \ref{lem:h1} gives:
\begin{equation}\label{combinewith2cases}
\begin{aligned}
&
\Big\n s\mapsto
(t-s)^{-\alpha}\int_{I_s}1_{\{a(s)\leq u \leq b(s)\}}\Phi(u)\,dW_H(u)
\Big\n_{L^p(\Omega;\gamma(0,t;X_{\delta}))}\\
& \quad \lesssim \big(\tfrac{T}{n}\minsym T_0\big)^{\a}
\n u\mapsto (t-u)^{-\alpha} \Phi(u)\n_{L^p(\Omega,\gamma(0,t;H,X_{\delta}))}.
\end{aligned}
\end{equation}
Taking the supremum over $t\in [0,T_0]$ above and combining the result 
with \eqref{Cl2:Linfty} we obtain  \eqref{remainder_est_2}.
This completes the proof of the claim.
\end{proof}
\section{Approximating semigroup operators}\label{sec:gammaHP}
In this section we prove a $\gamma$-boundedness result for families of operators
defined in terms of the so-called Hille-Phillips functional calculus of an operator $A$
that generates an analytic $C_0$-semigroup on a Banach space $X$. This result
will be used in the next section, where we prove an abstract convergence result
for time discretization schemes of \eqref{SEE}.\par
Let $(\mu_{n})_{n\ge N}$ be a sequence of non-negative finite measures on
$[0,\infty)$ and let $R\geq 0$ be given. For $j\in \N$ let $\munj = \mu_n *
\dots * \mu_n$ denote the $j$-fold convolution product of $\mu_n$ with itself.
Consider the following properties: 
\begin{itemize}
\item[\MII{}] For all $n\ge 1$ we have:
\begin{align*}
& \int_{0}^{\infty}t\,d\mu_{n}(t) = \tfrac{1}{n};
\end{align*}
\item[\MI{}] There exists an $N\geq 1$ such that for all $n\ge N$, all
$j=1,\ldots,n$, and every $\a\in (-1,1]$ we have:
\begin{align*}
& \int_{0}^{\infty} t^\a e^{R t}\,d\munj (t) < \infty;
\end{align*}
\item[\MIII{}] For every $\a\in (-1,1]$ we have:
\begin{align*}
 \sup_{n\geq N} \sup_{1\le j\le n} & \Big| j\int_{0}^{\infty} \big[ 1 -
\big(\tfrac{tn}{j}\big)^{\alpha}\big]e^{R t}\,d\munj (t)\Big| < \infty.
\end{align*}
\end{itemize}
\par
Let $A$ be the generator of an analytic $C_0$-semigroup $S$ on $X$ and let
$\omega\geq 0$ 
be such that $(e^{-\omega t}S(t))_{t\geq 0}$ is uniformly bounded. Fix $T>0$.
Let $(\mu_{n})_{n\ge N}$ 
be a sequence of non-negative $\sigma$-finite measures on $[0,\infty)$ that
satisfy \MII, \MI, \MIII{}
for $R=\omega T$. Let $N$ be as in \MI. For $n\ge N$ define $E(\tfrac{T}{n}) \in \calL(X)$ by
\begin{equation}\label{eq:def-E}
E(\tfrac{T}{n}) x:=\int_{0}^{\infty} S(t)x \,d\mu_{n}(t/T), \quad x\in X,
\end{equation}
where for $n\in \N$, $j\in \N$ we define:
\begin{align*}
\int_{0}^{\infty} f(t) \,d\munj(t/T):=\int_{0}^{\infty} f(tT) \,d\munj(t).
\end{align*}
By \MI, the right-hand side of \eqref{eq:def-E} is well defined as a Bochner integral in $X$.
It is an easy consequence of the semigroup property that, for all $j\ge 1$,
\begin{equation}\label{defEn}
\Enj x:= [E(\tfrac{T}{n})]^j x = \int_{0}^{\infty} S(t)x \,d\munj (t/T), \quad
x\in X.
\end{equation}
We supplement these definitions by putting $E(0):= I$.\par

\begin{example}
In Section \ref{ss:MEuler} below we will demonstrate that the family of measures
$$d\mu_{n}(t)=
n e^{-n t}\,dt$$ satisfy \MII, \MI, \MIII{}. For these measures we have 
$$\Enj x =(I-\tfrac{T}{n}A)^{-j}x,$$ which means that
$\Enj$ is the $j^{\textrm{th}}$ Euler approximation of $S(t_{j}^{(n)})$.
\end{example}\par
The following proposition and corollary give the $\gamma$-boundedness estimates 
for the differences $\Enj -S(t_{j}^{(n)})$
that play the same role in the proof of
Theorem \ref{t:euler} as the estimates of Lemma \ref{lem:analyticRbound} played
in the proof of Theorem \ref{thm:convV}.
\begin{proposition}\label{p:Eulergbdd} Let the setting be as described above.
\begin{enumerate}
 \item 
For all $\delta\in (-1,1]$ there exists a constant $C$ such that for all $n\ge
N$:
 $$\sup_{j=1,\ldots,n} \Big\n (t_j^{(n)})^{{1-\delta}}[\Enj -S(t_{j}^{(n)})]
\Big\n_{\calL(X_{\delta};X)} \le Cn^{-1}.$$
 \item 
For all $\delta\in (-1,1]$, $0\le \beta\leq 1-\delta$, and $\epsilon>0$ there
exists a constant $C$ such that for all
$n\ge N$: 
$$\gamma_{[X_{\delta},X]}\Big\{(t_j^{(n)})^{\beta}[\Enj
-S(t_{j}^{(n)})]:j=1,\ldots,n\Big\} 
\leq Cn^{-\beta-\delta+\epsilon}.$$
\end{enumerate}
\end{proposition}
\begin{remark}
Stronger $\gamma$-bounded\-ness estimates can be obtained by imposing strong\-er
conditions on the 
measures $\mu_n$. 
Such conditions would correspond
to using higher-order numerical approximation schemes. However, this will not
improve the overall convergence rates as
proven in Theorem \ref{t:euler} because the bottle-neck for convergence rate is
the noise approximation.
\end{remark}\par 
\begin{remark}
The first part of Proposition \ref{p:Eulergbdd}, concerning the uniform boundedness, 
has been known since the 1970's for the case that $\delta=0$ and $\Enj$ is the 
Euler approximation. Generally such results are proven by functional calculus
methods. Our proof may be read as an
extension of the approach taken by Bentkus and Paulauskas \cite{BentPaul04},
which is of more probabilistic nature and
seems the most suitable for
our needs.
\end{remark}\par 
Before turning to the proof of Proposition \ref{p:Eulergbdd}, we state a simple
corollary. 
\begin{corollary}\label{c:Eulergbdd} Let the setting be as described above.
\begin{enumerate}
 \item 
For all $\delta\in (-1,0]$
there exists a constant $C$ such that for all $n\ge N$:
 $$ \sup_{j=1,\ldots,n}
\Big\n (t_j^{(n)})^{-\delta} \Enj \Big\n_{\calL(X_{\delta};X)} \le C .$$
 \item 
For all $\delta\in (-1,0]$, $-\delta < \beta\leq 1-\delta$, and
$0<\epsilon< \b+\d$ there exists a constant $C$ such that for all $n\ge N$ and
all
$k=1,\dots, n$ we have
$$\gamma_{[X_{\delta},X]} \Big\{(t_j^{(n)})^{\beta} \Enj :j=1,\ldots,k\Big\}
\leq C(t_{k}^{(n)})^{\beta+\delta-\epsilon}.$$
\end{enumerate}
\end{corollary}
\begin{proof}
By the first part of Proposition \ref{p:Eulergbdd},
for all $1\le j\le n$
\begin{align*}
\big\n (t_j^{(n)})^{-\delta}\big(\Enj - S(t_{j}^{(n)})\big)\big
\n_{\calL(X_{\delta};X)} 
& \leq Cn^{-1}(t_j^{(n)})^{-1} \lesssim 1.
\end{align*}
Moreover, by the analyticity of the semigroup $S$ (i.e., estimate \eqref{analyticDiff}),
 and the fact that $\delta \leq 0$, 
\begin{align*}
\sup_{0\leq j\leq n} (t_j^{(n)})^{-\delta}\n
S(t_{j}^{(n)})\n_{\calL(X_{\delta};X)}
& \lesssim 1
\end{align*}
This proves the first part of the corollary.\par

By Part (1) of Lemma \ref{lem:analyticRbound} and Part (2) of Proposition
\ref{p:Eulergbdd}, 
observing that $\beta>-\delta\geq0$, we have:
\begin{align*}
\ & 
\gamma_{[X_{\delta},X]}\big\{(t_j^{(n)})^{\beta}E(t_{j}^{(n)})\,:\,j=1,\ldots,
k\big\} \\
& \qquad \lesssim
\gamma_{[X_{\delta},X]}\big\{(t_j^{(n)})^{\beta}S(t_{j}^{(n)})\,:\,j=1,\ldots,
k\big\}
\\ & \qquad\qquad + 
\gamma_{[X_{\delta},X]}\big\{(t_j^{(n)})^{\beta}\big(E(t_{j}^{(n)})-S(t_{j}^{(n)
})\big)\,:\,j=1,\ldots,k\big\}\\
& \qquad \lesssim  (t_j^{(n)})^{\beta+\delta} + n^{-\b -\d+\epsilon}
\\ & \qquad \lesssim (t_{k}^{(n)})^{\beta+\delta-\epsilon} ,
\end{align*}
with implied constants independent of $k$ and $n$ (although they may depend on
$T$). 
\end{proof}

In order to prove Proposition \ref{p:Eulergbdd} we shall make use the following simple 
observation. Suppose $\mu$ is a
probability measure on $[0,\infty)$, $t=\int_0^{\infty} s \,d\mu(s)$, 
and $f:[0,\infty)\rightarrow X$ is twice continuously
differentiable. By integration by parts one has:
\begin{equation}\label{Taylor}
\int_0^{\infty} f(s) \,d\mu(s)-f(t) = \int_{0}^{\infty} \int_{t}^{s}
(s-r)f''(r)\,dr\,d\mu(s).
\end{equation}\par
We substitute $f(s)=S(s)x$, $x\in X$, and $\mu=\munj $ for
$n>N$ and $j\in
\{ 1,\ldots,n\}$ in the above. From \MII{} we have that
$\int_0^{\infty} s \,d\munj (s/T) = t_{j}^{(n)}$. Thus by setting
$t=t_{j}^{(n)}$ in \eqref{Taylor} we obtain, for $x\in X$:
\begin{equation}
\begin{aligned}\label{funcalAppr}
\Enj x  - S(t_{j}^{(n)})x & = \int_{0}^{\infty} S(s)x \,d\munj (s/T) -
S(t_{j}^{(n)})x \\
&= \int_0^{\infty}
\int_{t_{j}^{(n)}}^{s} (s-r)A^2 S(r)x \,dr\,d\munj (s/T).
\end{aligned}
\end{equation}\par
\begin{proof}[Proof of Proposition \ref{p:Eulergbdd}.]
We first prove the statement concerning $\gamma$-boundedness.
Let $N$ be as in assumption \MI{}. Let $n>N$. Without loss of generality we may
assume $\delta-\epsilon \neq 0$ and $\delta-\epsilon\neq -1$. For
$j=1,2,\ldots,n$ define $\phi_{j}:[0,\infty)\times [0,\infty)\rightarrow \R$ by
\begin{align*}
\phi_j(s,r):=  (t_j^{(n)})^{\beta}r^{-2+\delta-\epsilon}(s-r) e^{\omega r}
(1_{\{t_{j}^{(n)}\leq r \leq
s\}}-1_{\{s\leq r\leq t_{j}^{(n)}\}}).
\end{align*}
By equality \eqref{funcalAppr} we have, for $x\in X$:
\begin{equation}
\begin{aligned}\label{Strans}
& (t_j^{(n)})^{\beta} [\Enj x -S(t_{j}^{(n)})x]\\
& \qquad =\int_0^{\infty} \int_{0}^{\infty} \phi_j(s,r)r^{2-\delta+\epsilon}
e^{-\omega r}A^2 S(r)x 
 \,dr\,d\munj (s/T),
\end{aligned}
\end{equation}
$j=1,2,\ldots,n.$\par

In a similar fashion as used for Lemma \ref{lem:analyticRbound} in Section
\ref{sec:analytic} one may prove that for $\delta\leq 2$ the set
\begin{align*}
\{ r^{2-\delta+\epsilon}e^{-\omega r}A^2S(r): r\in [0,\infty)\}
\end{align*}
is $\g$-bounded in $\calL(X_{\delta};X)$. By \cite[Proposition 2.5]{Wei} and
equation \eqref{Strans} it follows that
\begin{equation}\label{hest}
\begin{aligned}
& \gamma_{[X_{\delta},X]}\Big\{(t_j^{(n)})^{\beta}[\Enj
-S(t_{j}^{(n)})]\,:\,j=1,\ldots,n\Big\} \\ &\qquad \qquad \lesssim \sup_{1\leq j\leq n} \n
\phi_j\n_{L^1([0,\infty)\times
[0,\infty),\,\munj (\cdot/T)\times \lambda)},
\end{aligned}
\end{equation}
with implied constant independent of $n$, where $\lambda$ is the Lebesgue
measure on $[0,\infty)$.\par

Observe that for all $j=1,2,\ldots,n$ one has, because $\omega\geq 0$;
\begin{align*}
&\n \phi_j\n_{L^1([0,\infty)\times [0,\infty),\,\munj (\cdot/T)\times \lambda)} 
= (t_j^{(n)})^{\beta} \int_{0}^{\infty} \int_{t_{j}^{(n)}}^{s}
r^{-2+\delta-\epsilon}(s-r) e^{\omega r} 
\,dr\, d\munj (s/T)\\
&\qquad \qquad  \leq (t_j^{(n)})^{\beta} \int_{0}^{\infty} e^{(s\maxsym
T)\omega}\int_{t_{j}^{(n)}}^{s}  r^{-2+\delta-\epsilon}(s-r) \,dr\, d\munj (s/T).
\end{align*}
As $\delta-\epsilon\neq 0$ and $\delta-\epsilon\neq -1$, basic calculus gives:
\begin{equation}\label{simple_integration}
\begin{aligned}
\int_{t_{j}^{(n)}}^{s} r^{-2+\delta-\epsilon}(s-r) \,dr 
& = \frac{(t_{j}^{(n)})^{\delta-\epsilon} }{1-\delta+\epsilon} \bigg[
\inv{\delta-\epsilon}
\Big(1-\Big(\frac{s}{t_j^{(n)}}\Big)^{\delta-\epsilon}\Big) + 
\frac{s}{t_j^{(n)}}-1 \bigg].
\end{aligned}
\end{equation}\par

In the final estimate below we apply assumption \MIII{} (recall that we have
$R=\omega T$). Due to that assumption there exists a constant
$C$ independent of $n$ and $j\in \{1,\ldots,n\}$ such that:
\begin{equation}
\begin{aligned}\label{phiest2}
&\n \phi_j\n_{L^1([0,\infty)\times [0,\infty),\,\munj (\cdot/T)\times
\lambda)}\\
& \  \leq \frac{T^{\delta-\epsilon}}{1-\delta+\epsilon}
(t_j^{(n)})^{\beta+\delta-\epsilon} \int_{0}^{\infty}
e^{(s\maxsym T)\omega}\big[
\tinv{\delta-\epsilon}
\big(1-\big(\tfrac{s}{t_j^{(n)}}\big)^{\delta-\epsilon}\big) + 
\tfrac{s}{t_j^{(n)}}-1 \big]
\,d\munj (s/T)\\
& \  =  \frac{T^{\delta-\epsilon}}{1-\delta+\epsilon} 
(t_j^{(n)})^{\beta+\delta-\epsilon}\int_{0}^{\infty} e^{(s\maxsym
1)\omega T}[
\tinv{\delta-\epsilon} ((\tfrac{sn}{j})^{\delta-\epsilon}-1) +  \tfrac{sn}{j}-1
\big]
\,d\munj (s)\\
& \  \lesssim (t_j^{(n)})^{\beta+\delta-\epsilon}j^{-1}\lesssim 
n^{-\beta-\delta+\epsilon}.
\end{aligned}
\end{equation}
This in combination with estimate \eqref{hest} completes the proof, as
$\beta+\delta-\epsilon<1$.\par

As for the statement concerning uniform boundedness, by \eqref{funcalAppr} and
analyticity (estimate \eqref{analyticDiff}) we have, for $\delta\neq 0$:
\begin{align*}
& \big( t_{j}^{(n)}\big)^{1-\delta} \n \Enj
-S(t_{j}^{(n)})\n_{\calL(X_{\delta};X)} \\
& \quad \lesssim \big( t_{j}^{(n)}\big)^{1-\delta}\int_0^{\infty}
\int_{t_j^{(n)}}^{s} (s-r) \n A^{2-\delta}
S(r)\n_{\calL(X)} \,dr\,d\munj (s/T)\\
&\quad \lesssim \big( t_{j}^{(n)}\big)^{1-\delta}\int_0^{\infty}
\int_{t_j^{(n)}}^{s} r^{-2+\delta}(s-r)e^{\omega r}
\,dr\,d\munj (s/T)\\
& \quad \lesssim n^{-1},
\end{align*}
where the final estimate follows by similar arguments as used to 
obtain \eqref{phiest2}. In the case that $\delta =0$ or $\delta=1$ the
evaluation of the 
integral in \eqref{simple_integration} will contain a
logarithmic term, which can be estimated in a suitable manner by observing that
$\ln x \leq x-1$ 
for all $x>0$. We leave the details to the reader.\par
\end{proof}
\subsection{Examples}\label{ss:MEuler}
We have two main examples in mind, which lead to a splitting scheme with discretized noise and
the implicit Euler scheme, respectively.

\begin{example}[Splitting with discretized noise]\label{ex:discr-splitting}
The simplest example obtained by taking $$\mu_n:=\delta_{\inv{n}},$$
which correspond to the trivial choice $$\En{\tfrac{T}{n}}:=S(\tfrac{T}{n}).$$
The conditions \MII, \MI, \MIII\ are trivially fulfilled for any $R\geq 0$. 
\end{example}

\begin{example}[Implicit Euler]\label{ex:Euler}
We will show that the measures  $$d\mu_{n}(t)=
n e^{-n t}\,dt$$
fulfill assumptions \MII, \MI, \MIII\  for any $R\geq 0$. These measures give rise to the operators
$$ \En{\tfrac{T}{n}} = (I - \tfrac{T}{n}A)^{-1}.$$

To start the proof, first note that by induction,
$$d\munj (t)=\frac{(n t)^{j-1}}{(j-1)!}n 
e^{-n t}\,dt,$$ and thus, for all $\alpha> -j$ and all $n> \omega T$, one has:
\begin{equation}\label{expmoments}
\begin{aligned}
\int_{0}^{\infty} t^\alpha e^{\omega T t} \,d\munj (t) & = 
\frac{n^{j}}{(j-1)!(n-\omega T)^{j+\alpha}} \int_{0}^{\infty} u^{j+\alpha-1}e^{-u}\,du
\\ & =
\frac{n^{j}}{(n-\omega T)^{j+\alpha}}\frac{\Gamma(j+\alpha)}{\Gamma(j)}.
\end{aligned}
\end{equation}
This proves that \MII{} and \MI{} are satisfied with $N> \omega T$. 

As for \MIII{}, by \eqref{expmoments} we have for $\alpha\in (-1,1]$ and $n\geq N\geq 2\omega T$:
\begin{align}\label{MIIIcheck}
\begin{aligned}
& \int_{0}^{\infty}\big[ 1 - \big(\tfrac{tn}{j}\big)^{\alpha}\big]e^{\omega T t}\,d\munj (t)\\
& \qquad  = \Big(\frac{n}{n-\omega T}\Big)^{j}\Big(1 -\Big(\frac{n}{n-\omega T}\Big)^{\alpha} \Big) 
+
\Big(\frac{n}{n-\omega T}\Big)^{j+\alpha} 
\Big[ 1 - \frac{\Gamma(j+\alpha)}{j^{\alpha}\Gamma(j)}\Big].
\end{aligned}
\end{align}
As $(\tfrac{n}{n-\omega T})^{n+1} \rightarrow e^{\omega T}$ as $n\rightarrow \infty$, 
there exists a constant
$M$ such that $$\sup_{n\in\N}\sup_{s\in [0,n+1]}\Big|\Big(\frac{n}{n-\omega T}\Big)^{s}\Big|
=\sup_{n\in\N}\Big(\frac{n}{n-\omega T}\Big)^{n+1} \leq M.$$
Moreover, for $n\geq 2\omega T$ we have:
\begin{align*}
\Big| 1 -\Big(\frac{n}{n-\omega T}\Big)^{\alpha} \Big|
&= |\alpha|\int_{1}^{1+\frac{\omega T}{n-\omega T}} 
s^{\alpha-1} \,ds
 \leq |\alpha|\frac{\omega T}{n-\omega T} \leq 2|\alpha| \frac{\omega T}{n}.
\end{align*}\par

From \eqref{MIIIcheck} and the above estimates we thus obtain:
\begin{align}\label{MIIIcheck2}
\begin{aligned}
\Big| \int_{0}^{\infty}\big[ 1 - \big(\tfrac{tn}{j}\big)^{\alpha}\big]e^{\omega T t}
\,d\munj (t)\Big| 
& \leq  2M|\alpha| \frac{\omega T}{n} +
M \Big[ 1 - \frac{\Gamma(j+\alpha)}{j^{\alpha}\Gamma(j)}\Big].
\end{aligned}
\end{align}

For $j\geq 2$ define $g_j:[-1,1]\rightarrow \R$ by:
$$
g_j(x) = \left\{
\begin{aligned}
\displaystyle \frac1{x}\Big(\displaystyle 1-\frac{\Gamma(j+x)}{j^x \Gamma(j)}\Big); 
& \qquad x\neq 0,\\
\ln j -\Psi(j); & \qquad x=0,
\end{aligned}
\right.
$$
where $\Psi=(\ln(\Gamma))'$ is the di-gamma function.\par
Assumption \MIII{} follows from \eqref{MIIIcheck2} for $N\geq 2\omega T$ once
the following claim is established:
\subsubsection*{Claim} For $j\geq 2$ we have $g_j(x)\in [0,\tfrac{1}{j-1}]$ for
all $x\in [-1,1]$.
\subsubsection*{Proof of Claim}
As $g_j(-1)=\inv{j-1}$ and $g_j(1)=0$ for all $j\geq 2$, it suffices
to show that $g_j$ is non-increasing on $[-1,1]$. 

For $j\geq 2$ define the
function
$h_j:[-1,1]\rightarrow \R$ by $h_j(x):=1-j^{-x}\frac{\Gamma(j+x)}{\Gamma(j)}$.
For $x\in [-1,1]$ and $j\geq 2$ we have:
\begin{align*}
h_j'(x)&= \frac{\Gamma(j+x)}{j^{x}\Gamma(j)}\big[\ln
j-\Psi(j+x)\big]=(1-h_j(x))(\ln j -\Psi(j+x));\\
h_j''(x)&= \frac{-\Gamma(j+x)}{j^{x}\Gamma(j)}\big[\big(\Psi(j+x)-\ln
j\big)^2+\Psi'(j+x)\big].
\end{align*}
As the $\Gamma$-function is log-convex on $(0,\infty)$, we have that $\Psi'$ is
positive on
that interval. As $j\geq 2$ and $x\in [-1,1]$ we have that $j+x>0$ and thus
$h_j''(x)\leq 0$ for $x \in [-1,1]$.\par

One may check that $g_j$ is continuously differentiable and
$$
g_j'(x) = \left\{
\begin{aligned}
\frac{1}{x^2}(x h_j'(x) - h_j(x)); & \qquad x\neq 0,\\        
-\tfrac12\big[(\Psi(j)-\ln j )^2+\Psi'(j)\big]; & \qquad x=0.
\end{aligned}\right.
$$
To prove that $g_j$ is non-increasing on $[-1,1]$ it suffices to prove that
$$xh_j'(x)-h_j(x)\leq 0, \textrm{ for all }
x\in [-1,1].$$ Observe that $g_j'(0) \leq 0$, hence it suffices to prove that
$x\mapsto xh_j'(x)-h_j(x)$ is
non-decreasing on $[-1,0]$ and non-increasing on $[0,1]$. This follows from the
fact that
$\frac{d}{dx}[xh_j'(x)-h_j(x)]=xh_j''(x)$ and $h_j''\leq 0$ on $[-1,1]$.
\end{example}
\section{An abstract time discretization}\label{sec:absEuler}
In this section we prove a convergence result for a general class of approximation schemes
for \eqref{SEE} involving the operators $\Enj$ as defined
in \eqref{eq:def-E} and discretized noise. In particular, 
the convergence result contains the implicit Euler scheme as the special  
case that $\Enj = (I - \frac{T}{n}A)^{-j}$ (see Example 
\ref{ex:Euler}).

Throughout this section we consider the problem \eqref{SEE} under the
assumptions \MA{}, \MF{}, \MG. 
On the part of $X$ we shall assume that it is an \textsc{umd} space with
Pisier's property $(\alpha)$
introduced in \cite{Pis}. For an extensive discussion of this property and its
use in the theory
of stochastic evolution equations we refer to \cite{KalWei07, NeeWei08}.
Examples of Banach spaces 
with property $(\a)$ are the Hilbert spaces and the spaces $L^p(\mu)$ with $1\le
p<\infty$ and $\mu$ $\sigma$-finite. 
Here we need the fact (see \cite{KalWei07, NeeWei08} for the proof and
generalizations)
that if
$X$ has property $(\alpha)$, then for any two measure spaces 
$(R_1,\mathcal{R}_1,\mu_1)$ and $(R_2,\mathcal{R}_2,\mu_2)$ and any Hilbert
space $H$ we have a natural isomorphism
\begin{equation}\label{ggiso}
\gamma(R_1;\gamma(R_2;H,X)) \simeq \gamma(R_1\times R_2;H,X),
\end{equation}
which is given by  the mapping $f_1\otimes ((f_2\otimes h)\otimes x) \mapsto
((f_1\otimes f_2)\otimes h)\otimes x)$.

Set:
\begin{equation}
\begin{aligned}\label{eq:zeta-max}
\zeta_{\max}& := \min\{1 - (\tinv{\tau} - \tfrac12) +(\theta_F\wedge 0),\tinv{2}+(\theta_G\wedge
0)\},
\end{aligned}
\end{equation}
where $\tau\in (1,2]$ is the type of $X$. In addition to the assumptions \MF{},
\MG{}
we shall assume:
\begin{itemize}
\item[\MFII{}] There exists a constant $C$ such that for all $x\in X$ and $s,t\in [0,T]$ we
have: 
\begin{align*}
\n F(t,x)-F(s,x)\n_{X_{\theta_F\wedge 0}} & \leq
C|t-s|^{\zeta_{\max}}(1+\n x\n_{X}).
\end{align*}
\item[\MGII{}] 
There exists a constant $C$ such that for all $x\in X$ and $s,t\in [0,T]$ we
have: 
\begin{align*}
\n G(t,x) - G(s,x)\n_{\gamma(H,X_{\theta_G\wedge 0}))} & \leq
C|t-s|^{\zeta_{\max}+\inv{\tau}-\inv{2}}(1+\n x\n_{X}).
\end{align*}
Moreover, for every $t\in [0,T]$ there exists a constant $C_t$ 
such that:
\begin{align*}
\n G(t,x)-G(t,y) \n_{\gamma(H,X_{\theta_G\wedge 0})} & \leq C_t \n x-y\n_{X}.
\end{align*}
\end{itemize}\par
\begin{remark}
Clearly, condition \MFII{} is automatically satisfied if $F$ is not
time-dependent and satisfies \MF.\par
Condition \MGII{} is also automatically satisfied if $G$ is not time-dependent and satisfies \MG,
provided we assume 
$G(0)\in \g(H,X_{\theta_G\wedge 0})$. Indeed, in that case, from \cite[Lemma 5.3]{NVW08}
it follows that 
 $G$ takes values in $\g(H,X_{\theta_G\wedge 0})$ and the conditions of \MGII{}
concerning H\"older 
continuity in the first variable and Lipschitz continuity the second are
satisfied.
\end{remark}\par
The reader will have noticed that the above assumptions are phrased in terms 
of $\theta_F\wedge 0$ and $\theta_G\wedge 0$.
The reason for this is explained in Remark \ref{r:thetaNeg} below. Because of this, 
{\em for the rest of this section, without loss of generality we shall 
assume that $\theta_F,\theta_G\ge 0$}.
The other assumptions on $\theta_F$ and $\theta_G$ remain in force. Explicitly, we 
assume 
$$-1 + (\tfrac1\tau-\tfrac12) < \theta_F\le 0,\quad -\tfrac12< \theta_G \le 0.$$
Once this convention is in force, of course one has $\zeta_{\max}=\eta_{\max}$. In order to
remind the reader of the convention,
we shall continue the use of $\zeta_{\max}$.

Let us now introduce the discrete-time approximation scheme that will be studied
in this section. Fix $T>0$ and let $(\mu_n)_{n=1}^{\infty}$ be a family of measures 
satisfying 
\MII, \MI, \MIII{} for $R=\omega T$, where $\omega\geq 0$ is such that 
$e^{-\omega t}S(t)$ is uniformly bounded in $t\in [0,\infty)$. Let $\Enj$ be 
defined by \eqref{defEn}.
We fix $p>2$ and let $U$ be the mild solution to \eqref{SEE} with initial value
$x_0\in L^p(\Omega,\calF_0;X)$. 
We fix another initial value $y_0\in L^p(\Omega,\calF_0;X)$ (in the applications
below, the typical situation is that 
$y_0$ is a close approximation to $x_0$). Let $n\geq N$, where $N$ is as in \MI{}. Set 
$V^{(n)}_{0}:=y_0$ and define $V_{j}^{(n)}$, $j=1,\ldots,n$, inductively as
follows:
\begin{equation}\label{appr_scheme}
\begin{aligned}
V_{j}^{(n)}&:= E(\tfrac{T}{n}) [V^{(n)}_{j-1} +
\tfrac{T}{n}F(t_{j-1}^{(n)},V^{(n)}_{j-1}) 
+ G(t_{j-1}^{(n)},V^{(n)}_{j-1})\Delta W_j^{(n)}].
\end{aligned}
\end{equation}
Here,
$$ \Delta W_j^{(n)} := W_H(t_j^{(n)}) - W_H(t_{j-1}^{(n)}).$$
The rigorous interpretation of the term $G(t_{j-1}^{(n)},V^{(n)}_{j-1})\Delta
W_j^{(n)}$ proceeds in three steps. 

{\em Step 1:} Let us first fix an operator 
$R\in \gamma(H,X_{\theta_G})$. By standard results on $\g$-radonifying operators
(see, e.g., \cite{Nee-survey})
 may write $$R=\sum_{k=1}^{\infty}h_k\otimes x_k$$  
for some orthonormal sequence $(h_k)_{k=1}^{\infty}$ in $H$ and a 
sequence $(x_k)_{k=1}^{\infty}$ in $X_{\theta_G}$ 
(the convergence of the sum being in $\gamma(H,X_{\theta_G})$). 
For sets $B\in \F_{j-1}^{(n)}:=\F_{t_j^{(n)}}$
we now define
\begin{equation}\label{DeltaWsimple}
(1_{B}\otimes R)\Delta W_j^{(n)} := 1_{B} \sum_{k=1}^{\infty} W_{H}(h_k \otimes
1_{(t_{j-1}^{(n)},t_{j}^{(n)}]})\otimes x_k.
\end{equation}
The sum on the right-hand side above converges in $L^p(\Omega;X_{\theta_G})$ 
 since $W_H$ extends to a bounded operator from
$\gamma(L^2(0,T;H);X_{\theta_G})$ into $L^p(\O;X_{\theta_G})$
(see \cite{Nee-survey}). By
the independence of $W_{H}(h_k \otimes 1_{(t_{j-1}^{(n)},t_{j}^{(n)}]})$ and
$\F_{j-1}^{(n)}$, the product of
$1_B$ and this sum converges in $L^p(\Omega;X_{\theta_G})$ as well.  Moreover,
by the Kahane-Khintchine inequality,
\begin{align*}
& \n (1_{B}\otimes R)\Delta W_j^{(n)}\n_{L^p(\O;X_{\theta_G})} \\
&\qquad \qquad  \lesssim (\E (1_B))^\frac1p \Big\n 
\sum_{k=1}^{\infty} (h_k \otimes 1_{(t_{j-1}^{(n)},t_{j}^{(n)}]})\otimes
x_k\Big\n_{\gamma(L^2(0,T;H);X_{\theta_G})}
\\ &\qquad \qquad  = (t_j^{(n)} - t_{j-1}^{(n)})^{\frac{1}{2}} (\E
(1_B))^\frac1p \Big\n 
\sum_{k=1}^{\infty} h_k \otimes x_k\Big\n_{\gamma(H,X_{\theta_G})} \\ 
 & \qquad \qquad = (\tfrac{T}{n})^\frac12 (\E (1_B))^\frac1p \n R
\n_{\gamma(H,X_{\theta_G})}
\end{align*}
with implied constants depending on $p$ only.

{\em Step 2:} Now fix a simple random variable
$\phi\in L^p(\Omega,\F_{j-1}^{(n)};\g(H,X_{\theta_G}))$, say
$ \phi = \sum_{j=1}^k 1_{B_j} \otimes R_j$ with the sets $B_j \in
\F_{j-1}^{(n)}$ disjoint. By the above, 
\begin{equation}\label{DeltaWsimpleNorm}
\begin{aligned}
\n \phi \Delta W_j^{(n)} \n_{L^p(\Omega;X_{\theta_G})} & \lesssim
(\tfrac{T}{n})^{\frac12}\sum_{j=1}^k  (\E (1_{B_j}))^\frac1p \n R_j
\n_{\gamma(H,X_{\theta_G})}
\\ & = (\tfrac{T}{n})^{\frac12} \n \phi \n_{L^p(\Omega;\gamma(H,X_{\theta_G}))}.
\end{aligned} 
\end{equation}

{\em Step 3}: Let $(v_m)_{m\in\N}$ be a sequence of simple $X$-valued random
variables approximating $V^{(n)}_{j-1}$ in 
$L^p(\Omega,\F_{j-1}^{(n)};X)$. By the norm estimate \eqref{DeltaWsimpleNorm}
and the Lipschitz condition in \MGII{}
it follows that the sequence 
$(G(t_{j-1}^{(n)},v_m)\Delta W_{j}^{(n)})_{m\in \N}$ is Cauchy in
$L^p(\Omega;X_{\theta_G})$.
Now we define
$$  G(t_{j-1}^{(n)},V^{(n)}_{j-1})\Delta W_j^{(n)} := \lim_{m\to\infty}
G(t_{j-1}^{(n)},v_m)\Delta W_{j}^{(n)}.$$
This completes the construction.

Returning to the abstract scheme \eqref{appr_scheme}, 
we have the following explicit expression for
$V_{j}^{(n)}$: 
\begin{align*}
V^{(n)}_{j} 
= \En{t_{j}^{(n)}} y_0 & + \tfrac{T}{n}\sum_{k=1}^{j}
\En{t_{j-k+1}^{(n)}}F(t_{k-1}^{(n)},V^{(n)}_{k-1}) \\
& + \sum_{k=0}^{j} \En{t_{j-k+1}^{(n)}}G(t_{k-1}^{(n)},V^{(n)}_{k-1})\Delta
W_{k}^{(n)}, \quad
j=0,\dots,n.
\end{align*}
We define, for $s\in [0,T]$,
\begin{align}\label{defEulerAbs}
V^{(n)}(s)&=\sum_{j=1}^{n}V^{(n)}_{j}1_{I_j}(s),
\end{align}
where, as always, $I_j = [t_{j-1}^{(n)}, t_j^{(n)})$. This process satisfies the
identity
\begin{equation}
\begin{aligned}\label{EulerConv}
V^{(n)}(s) & = \En{\ov{s}} y_0 + \int_{0}^{\ov{s}}
\En{\ov{s}-\un{u}}F(\un{u},V^{(n)}(\un{u}))\,du \\
& \phantom{ = \En{\ov{s}} y_0}\,\, + \int_{0}^{\ov{s}}
\En{\ov{s}-\un{u}}G(\un{u},V^{(n)}(\un{u}))\,dW_H(u) \\
& = \En{\ov{s}} y_0 + \int_{0}^{\ov{s}}
\En{\ov{s}-\un{u}}F(\un{u},V^{(n)}(u))\,du \\
& \phantom{ = \En{\ov{s}} y_0 }\,\, + \int_{0}^{\ov{s}}
\En{\ov{s}-\un{u}}G(\un{u},V^{(n)}(u))\,dW_H(u).
\end{aligned}
\end{equation}\par

The main result of this section reads as follows. Recall the definition of 
$\zeta_{\max}$ in \eqref{eq:zeta-max} and let $N$ be as in \MI.
\begin{theorem}\label{t:euler}
Let $\eta>0$ be such that $0\le \eta <
\zeta_{\max}$, and suppose
$0\le \a < \frac12$. Then for all $p\in (2,\infty)$ there exists a constant
$C$ such that for all 
$x_0\in L^p(\Omega,\calF_0;X_{\eta})$, $y_0\in L^p(\Omega,\calF_0;X)$, and
$n\geq N$ we have
\begin{equation}
\begin{aligned}
\n U-V^{(n)} \n_{\Winfp{\alpha}([0,T]\times\Omega;X)} &\le C\n
x_0-y_0\n_{L^p(\Omega;X)} 
+ Cn^{-\eta}(1+\n x_0\n_{L^p(\Omega;X_{\eta})}).
\end{aligned}
\end{equation}
\end{theorem}\par
\begin{remark}\label{r:thetaNeg}
Unlike the case in Theorem \ref{thm:convV}, the convergence rate of the Euler
approximations does not improve if $\theta_F$ and $\theta_G$ increase above $0$. 
This is mainly
due to the time-discretization of the noise. More precisely, the estimates on the 
sixth and
tenth term in \eqref{eulersplit} below do not improve if $\theta_F$ and $\theta_G$ 
increase above $0$. The
estimate on third term in \eqref{eulersplit} as presented here also does not improve 
if $\theta_F$
increases above $0$, but we believe this is just an artifact of our proof.
\end{remark}
\begin{remark}\label{r:MFIIMGII}
The H\"older conditions
of \MFII{} and \MGII{} can be weakened: in order to obtain 
convergence rate $\eta$ in Theorem \ref{t:euler} it suffices that 
the H\"older exponent in \MFII{} is $\eta$ instead of $\zeta_{\max}$, 
and that the exponent in \MGII{} is $\eta+\inv{\tau}-\inv{2}$ instead 
of $\zeta_{\max}+\inv{\tau}-\inv{2}$.
\end{remark}\par
\begin{proof}[Proof of Theorem \ref{t:euler}]
We begin by observing that, due to the assumption $2<p<\infty$, the spaces
$X$ and $L^p(\O;X)$ has the same type $\tau$. 
\addtocontents{toc}{\SkipTocEntry}\subsection*{Part 1.}
The main issue is to prove that
there exists $T_0\in (0,T]$ and a constant $C$ such that for all $n\in \N$ and
$j\in \{0,\dots,n\}$
we have:
\begin{align*}
 \ & \n U-V^{(n)} \n_{\Winfp{\alpha}([t_j^{(n)},t_j^{(n)}+T_0]\times\Omega;X)} 
 \\ & \qquad \lesssim C\n U(t_j^{(n)}) - V_j^{(n)}\n_{L^{p}(\Omega;X)} +
Cn^{-\eta}
(1+\n U(t_j^{(n)})\n_{L^{p}(\Omega;X_{\eta})}).
\end{align*}
This statement is entirely analogous to the result obtained in Part 3 of the
proof of Theorem \ref{thm:convV}. 
Once it has been established, the extension to the interval $[0,T]$ can be
obtained in precisely the same way 
as in Part 4 of Theorem \ref{thm:convV}.\par

Until further notice we fix $n\geq N$ and $T_0\in [0,T]$.

Let $(U^{(n)}_{j})_{j=1}^{n}$ be the modified splitting scheme as defined by
\eqref{SDE_noA} in Section
\ref{sec:split}, with initial datum $x_0$.
Let $U^{(n)}$ be the corresponding process as defined by \eqref{splitdef} in
Section \ref{sec:split}. 
By Theorem
\ref{thm:convV} we have, for all $j$: 
\begin{equation}\label{convSmallIntervalSplit}
\begin{aligned}
& \n U-V^{(n)} \n_{\Winfp{\alpha}([t_j^{(n)},t_j^{(n)}+T_0]\times\Omega;X)} \\
& \quad  \leq \n U - U^{(n)} \n_{\Winfp{\alpha}([t_j^{(n)},t_j^{(n)}+T_0]\times
\Omega;X)} + \n
U^{(n)}-V^{(n)} \n_{\Winfp{\alpha}([t_j^{(n)},t_j^{(n)}+T_0]\times \Omega;X)}\\
& \quad \lesssim n^{-\eta}(1+\n U(t_j^{(n)})\n_{L^p(\Omega;X_{\eta})}) + \n
U^{(n)}-V^{(n)}
\n_{\Winfp{\alpha}([t_j^{(n)},t_j^{(n)}+T_0]\times \Omega;X)},
\end{aligned}
\end{equation}
with implied constants independent of $n$ and $j$. Thus 
it suffices to show that 
\begin{equation}\label{Euler-Split}
\begin{aligned}
\ & \n U^{(n)}-V^{(n)}\n_{\Winfp{\alpha}([t_j^{(n)},t_j^{(n)}+T_0]\times
\Omega;X)} 
\\ & \qquad \lesssim C\n U(t_j^{(n)}) - V_j^{(n)}\n_{L^{p}(\Omega;X)} +
Cn^{-\eta}
(1+\n U(t_j^{(n)})\n_{L^{p}(\Omega;X_{\eta})}).
\end{aligned}
\end{equation}
\addtocontents{toc}{\SkipTocEntry}\subsection*{Part 2.}
For simplicity we shall prove this for $j=0$ (careful examination of the
proof reveals that the 
other $t_j^{(n)}$ do not generate extra
difficulties). In that case we have $U(t_j^{(n)}) = U(0) = x_0$ and $V^{(n)}_{j}=V_0^{(n)} =
y_0$.

From the integral representations
\eqref{splitConv} and \eqref{EulerConv} we have, for $n\geq N$:
\begin{equation}\label{eulersplit}
\begin{aligned}
U^{(n)}(t)- V^{(n)} (t) & = S(\ov{t})x_0 - \En{\ov{t}} y_0 \\
&\quad + \int_{0}^{t} S(\ov{t}-\un{s})F\big(s,U^{(n)}(s)\big)\,ds \\ 
&\quad  - \int_0^{\ov{t}} \En{\ov{t}-\un{u}}F(\un{u},V^{(n)}(u))\,du \\
&\quad  + \int_{0}^{t} S(\ov{t}-\un{s})G\big(s,U^{(n)}(s)\big)\,dW_H(s) \\ 
&\quad  - \int_{0}^{\ov{t}} \En{\ov{t}-\un{u}}G(\un{u},V^{(n)}(u))\,dW_H(u)\\
& = (S(\ov{t})-\En{\ov{t}})x_0 + \En{\ov{t}}(x_0-y_0) \\
&\quad  + \int_{0}^{\ov{t}}[S(\ov{t}-\un{s})-\En{\ov{t}-\un{s}}]F(s,U^{(n)}(s))\,ds \\
&\quad + \int_{0}^{\ov{t}}\En{\ov{t}-\un{s}}[F(s,U^{(n)}(s))-F(s,V^{(n)}(s))]\,ds \\
&\quad + \int_{0}^{\ov{t}}\En{\ov{t}-\un{s}}[F(s,V^{(n)}(s))-F(\un{s},V^{(n)}(s))]\,ds \\
&\quad + \int_{t}^{\ov{t}} S(\tfrac{T}{n})F(s,U^{(n)}(s))\,ds\\
&\quad + \int_{0}^{\ov{t}}[S(\ov{t}-\un{s})-\En{\ov{t}-\un{s}}]G(s,U^{(n)}(s))\,dW_H(s)\\
&\quad + \int_{0}^{t}\En{\ov{t}-\un{s}}[G(s,U^{(n)}(s))-G(s,V^{(n)}(s))]\,dW_H(s)\\
&\quad + \int_{0}^{t}\En{\ov{t}-\un{s}}[G(s,V^{(n)}(s))-G(\un{s},V^{(n)}(s))]\,dW_H(s)\\
&\quad + \int_{t}^{\ov{t}} S(\tfrac{T}{n})G(s,U^{(n)}(s))\,dW_H(s).
\end{aligned}
\end{equation}
We shall estimate each of the ten terms on the right-hand side above
separately. 
The implied constants in these
estimates may depend on $T$, although this will not be stated explicitly. 
However, for the fourth, fifth, eighth and ninth term
(Part 2d and 2g below) it will be necessary to keep track of the dependence upon
$T_0$.\par

Without loss of generality we may assume that
$\tau\in (1,2)$. We fix $0<\e<\frac12$ such that 
\begin{equation}\label{eq:cond-eta-eps} \eps< \max\{\zeta_{\max}-\eta,
1-2\alpha\},\end{equation}
where $\zeta_{\max}$ is defined as in \eqref{eq:zeta-max}.
As $\Winfp{\alpha}\hookrightarrow \Winfp{\beta}$ for $\alpha>\beta$,
we may also assume
that 
\begin{equation}\label{eq:extra-e}
\tfrac1{2}-\tfrac23\eps< \alpha<\tfrac1{2}-\tfrac12\eps.
\end{equation}

\addtocontents{toc}{\SkipTocEntry}\subsection*{Part 2a.}
For the first term on the right-hand side of \eqref{eulersplit} we have, 
by the uniform boundedness estimate of Proposition \ref{p:Eulergbdd} with
$\delta=\eta$, pointwise in $\omega\in \Omega$: 
\begin{equation}\label{Eu1a}
\n s\mapsto (S(\ov{s})-\En{\ov{s}})x_0\n_{L^{\infty}(0,T_0;X)}  \lesssim 
 n^{-1} \sup_{1\le j\le n} (t_j^{(n)})^{1-\eta}\n x_0\n_{X_{\eta}}
 \lesssim n^{-\eta}\n x_0\n_{X_{\eta}}.
\end{equation}

Let $t\in [0,T_0]$. By the $\gamma$-boundedness result of Proposition
\ref{p:Eulergbdd} with $\beta = \epsilon = \frac12\eps$ and
$\delta=\eta$,
the $\g$-multiplier theorem (Theorem \ref{t:KW}), and \eqref{eq:gFub}
we have, pointwise in $\omega\in\Omega$:
\begin{align*}
& \n s\mapsto (t-s)^{-\alpha}(S(\ov{s})-\En{\ov{s}})x_0\n_{\gamma(0,t;X)}\\
&\qquad \lesssim n^{-\eta} \n s\mapsto (t-s)^{-\alpha}\ov{s}^{-\frac12\eps }
x_0\n_{\gamma(0,t;X_\eta)}\\
&\qquad \eqsim n^{-\eta} \n s\mapsto (t-s)^{-\alpha}\ov{s}^{-\frac12\eps }
x_0\n_{\gamma(0,t;X_\eta)}\\
&\qquad = n^{-\eta} \n s\mapsto (t-s)^{-\alpha}\ov{s}^{-\frac12\eps } \n_{L^2(0,t)} \n
x_0\n_{X_{\eta}}\\
&\qquad \eqsim n^{-\eta} t^{\inv{2}-\alpha-\frac12\eps }\n x_0\n_{X_{\eta}},
\end{align*}
with implied constants independent of $n$, $T_0$ and $x_0$. 
As $\inv{2}-\alpha-\frac12\eps>0$, we have $t^{\inv{2}-\alpha-\frac12\eps}
\leq T^{\inv{2}-\alpha-\frac12\eps}$. By taking the supremum over 
$t\in [0,T_0]$ in the above, combining the result with \eqref{Eu1a}, 
and then taking $p^{\textrm{th}}$ moments, one obtains :
\begin{align}\label{Eu1}
\n s\mapsto (S(\ov{s})-\En{\ov{s}})x_0\n_{\Winfp{\alpha}([0,T_0]\times\Omega;X)} & \lesssim
n^{-\eta}\n
x_0\n_{L^p(\Omega;X_{\eta})},
\end{align}
with implied constants independent of $n$, $T_0$, and $x_0$.
\addtocontents{toc}{\SkipTocEntry}\subsection*{Part 2b.}
Concerning the second term on the right-hand side of \eqref{eulersplit} recall
that as $A$ generates an analytic
$C_0$-semigroup, there exists a constant $M$ such that
$$\sup_{n\geq N} \sup_{k\in \{ 1,\ldots,n\}}\n
\En{t_{k}^{(n)}}\n_{\calL(X)} \leq M.$$ Thus pointwise in $\omega\in\Omega$ we
have:
\begin{align*}
\n s\mapsto \En{\ov{s}} (x_0-y_0)\n_{L^{\infty}(0,T_0;X)} & \lesssim \n x_0-y_0
\n_{X}.
\end{align*}
Let $t\in [0,T_0]$. We apply the second part of Corollary \ref{c:Eulergbdd} with
$\beta=\eps$, $\delta =0$, $\epsilon=\frac12\eps $,
$k=n$. Arguing as in the previous estimate we obtain:
\begin{align*}
& \n s\mapsto (t-s)^{-\alpha} \En{\ov{s}}(x_0-y_0) \n_{\gamma(0,t;X)}\\
& \qquad  \lesssim \n s\mapsto (t-s)^{-\alpha}\ov{s}^{-\eps}
\n_{L^2(0,t)} \n x_0-y_0\n_{X}\\
& \qquad \lesssim \n x_0-y_0 \n_{X},
\end{align*}
with implied constants independent of $n$, $T_0$, $x_0$ and $y_0$.\par
As $t\in [0,T_0]$ was arbitrary, by taking $p^{\textrm{th}}$ moments it follows
that:
\begin{equation}\label{Eu2}
\begin{aligned}
\n s\mapsto \En{\ov{s}} (x_0-y_0)\n_{\Winfp{\alpha}([0,T_0]\times\Omega;X)} &
\lesssim \n x_0-y_0 \n_{L^p(\Omega;X)}
\end{aligned}
\end{equation}
with implied constants independent of $n$, $T_0$, $x_0$, and $y_0$.
\addtocontents{toc}{\SkipTocEntry}\subsection*{Part 2c.}
By the uniform boundedness estimate of Proposition \ref{p:Eulergbdd} with
$\delta=\theta_F$
one has, for $s\in [0,T_0]$ fixed:
\begin{equation*}
\begin{aligned}
& \Big\n \int_{0}^{\ov{s}}[S(\ov{s}-\un{u})-\En{\ov{s}-\un{u}}]F(u,U^{(n)}(u))\,du
\Big\n_{L^p(\Omega;X)}\\
& \qquad \leq \int_{0}^{\ov{s}}\n
[S(\ov{s}-\un{u})-\En{\ov{s}-\un{u}}]F(u,U^{(n)}(u))\n_{L^p(\Omega;X)} \,du \\
& \qquad \lesssim  \int_{0}^{\ov{s}}
n^{-1}(\ov{s}-\un{u})^{-1+\theta_{F}}\,du\n
F(\cdot, U^{(n)})\n_{L^{\infty}(0,\ov{T_0};L^p(\Omega;X_{\theta_F}))} \\
& \qquad \le \frac1n \sum_{j=1}^n (T/n) \cdot (jT/n)^{-1+\theta_{F}}\n F(\cdot,
U^{(n)})\n_{L^{\infty}(0,\ov{T_0};L^p(\Omega;X_{\theta_F}))}\\
& \qquad \eqsim n^{-1-\theta_F} \sum_{j=1}^n j^{-1+\theta_{F}} 
(1 + \n U^{(n)}\n_{L^{\infty}(0,\ov{T_0};L^p(\Omega;X))})\\
& \qquad \lesssim n^{-1-\theta_F+\frac12\e} (1 + \n x_0\n_{L^p(\Omega;X)}),
\end{aligned}
\end{equation*}
where we used the linear growth condition in \MF{} and 
that
$\sum_{j=1}^n j^{-1+\theta_{F}}\lesssim 1$ if $\theta_F < 0$
and 
$\sum_{j=1}^n j^{-1+\theta_{F}}\lesssim \ln n \lesssim n^{\frac{1}{2}\e}
$ if $\theta_F = 0$. In the last step we
used Corollary \ref{cor:Lpest} with $\delta = 0$. The implied
constants are independent of $n$, $T_0$, and $x_0$.\par
By taking the supremum 
over $s\in [0,T_0]$ we
obtain:
\begin{equation}\label{Linftyest1c}\begin{aligned}
& \Big\n s\mapsto \int_{0}^{\ov{s}}[S(\ov{s}-\un{u})-\En{\ov{s}-\un{u}}]F(u,U^{(n)}(u))\,du
\Big\n_{L^{\infty}(0,T_0;L^p(\Omega;X))}\\
& \qquad \qquad \lesssim n^{-1-\theta_F+\frac12\e} (1 + \n x_0\n_{L^p(\Omega;X)}).
\end{aligned}\end{equation}\par

For the estimate in the weighted $\gamma$-space we shall use 
Lemma \ref{lem:wBesovEst}. Define $\Psi:[0,T_0]\rightarrow
L^p(\Omega;X)$ by $$\Psi(s)=
\int_{0}^{\ov{s}}[S(\ov{s}-\un{u})-\En{\ov{s}-\un{u}}]F(u,U^{(n)}(u))\,du.$$ Let
$t\in
[0,T_0]$ and let $q=(\inv{\tau}-\inv{2}+\frac12\eps )^{-1}$ (so
$\inv{\tau}-\inv{2}<\inv{q}<\inv{\tau}-\alpha$). By Lemma
\ref{lem:wBesovEst} we have:
\begin{equation}\label{wGammaEst2c}
\begin{aligned}
& \sup_{t\in [0,T_0]}\n  s\mapsto (t-s)^{-\alpha} \Psi(s)
\n_{\gamma(0,t;L^p(\Omega;X))}\\
& \qquad \qquad \lesssim \n  \Psi
\n_{B_{q,\tau}^{\inv{\tau}-\inv{2}}([0,T_0];L^p(\Omega;X))} + \n \Psi
\n_{L^{\infty}(0,T_0;L^p(\Omega;X))},
\end{aligned}
\end{equation}
with implied constant independent of $T_0$.\par

Let $\rho \in [0,1]$ and let $0<|h|< \rho$. We have, with $I= [0,T_0]$,
\begin{align*}
\n T_h^I \Psi(s) - \Psi(s) \n_{L^p(\Omega;X)} & \leq \left\{ 
\begin{array}{ll} 
0, &  \ov{s+h}=\ov{s}, s+h\in [0,T_0];\\
2\n \Psi \n_{L^{\infty}(0,T_0;L^p(\Omega;X))},
& \ov{s+h}\neq \ov{s} \textrm{ or } s+h \notin [0,T_0].
\end{array}\right.
\end{align*}
Suppose $|h|<\frac{T}{n}$. Define $I_h=\{ s\in [0,T_0]: \ov{s+h}\neq \ov{s}\}$
and observe that $|I_h|\leq n|h|$. 
Thus by the definition of $q$ and by
\eqref{Linftyest1c}:
\begin{align*}
\n T_h^I \Psi - \Psi \n_{L^{q}(0,T_0;L^p(\Omega;X))} 
& \lesssim (n|h|)^{\inv{q}}\n \Psi
\n_{L^{\infty}(0,T_0;L^p(\Omega;X))}\\
& \lesssim |h|^{\inv{\tau}-\inv{2}+\frac12\eps }n^{-\frac{3}{2}+\inv{\tau}-\theta_F+\eps}
(1+\n x_0\n_{L^p(\Omega;X)}).
\end{align*}
On the other hand, if $|h|\geq \frac{T}{n}$ then we have:
\begin{align*}
\n T_h^I\Psi-\Psi\n_{L^q(0,T_0;L^p(\Omega;X))} 
& \lesssim \n \Psi \n_{L^{\infty}(0,T_0;L^p(\Omega;X))} 
\\ & \lesssim
n^{-1-\theta_F+\frac12\e} (1 + \n x_0\n_{L^p(\Omega;X)}) 
\\ & = |h|^{\inv{\tau}-\inv{2}+\frac12\eps }n^{-\frac{3}{2}+\inv{\tau}-\theta_F+\eps}(1 + \n
x_0\n_{L^p(\Omega;X)}).
\end{align*}\par

Combining the two cases and recalling that $\eta<\frac32-\inv{\tau}+\theta_F-\eps$ 
by \eqref{eq:cond-eta-eps} we obtain:
\begin{align*}
\sup_{0<|h|\leq \rho}\n T_h^I \Psi(s) - \Psi(s)
\n_{L^{q}(0,T_0;L^p(\Omega;X))} & \le 
\rho^{\inv{\tau}-\inv{2}+\frac12\eps }n^{-\eta}
(1+ \n x_0\n_{L^p(\Omega;X)}).
\end{align*}
By the definition of $B_{q,\tau}^{\inv{\tau}-\inv{2}}$ and equation \eqref{Linftyest1c}, 
this gives:
\begin{align*}
& \n  \Psi \n_{B_{q,\tau}^{\inv{\tau}-\inv{2}}([0,T_0];L^p(\Omega;X))} \\
& \qquad \lesssim \n \Psi
\n_{L^q(0,T_0;L^p(\Omega;X))} + n^{-\eta}\int_0^{1}\rho^{\inv2\eps\tau -1}d\rho(1+\n
x_0\n_{L^p(\Omega,X)})\\
& \qquad \lesssim \n \Psi \n_{L^{\infty}(0,T_0;L^p(\Omega;X))} + n^{-\eta}(1+\n
x_0\n_{L^p(\Omega,X)})\\
&\qquad  \lesssim n^{-\eta}(1 + \n x_0\n_{L^p(\Omega;X)}),
\end{align*}
with implied constant independent of $n$, $T_0$, and $x_0$.\par
Inserting the above and \eqref{Linftyest1c} into \eqref{wGammaEst2c} we obtain:
\begin{equation}
\begin{aligned}\label{BesovEst1c}
& \sup_{t\in [0,T_0]}
\n  s\mapsto (t-s)^{-\alpha}\Psi(s)
\n_{\gamma(0,t;L^p(\Omega;X))} \lesssim n^{-\eta}(1 + \n x_0\n_{L^p(\Omega;X)}).
\end{aligned}
\end{equation}\par

Finally, by combining \eqref{Linftyest1c} and \eqref{BesovEst1c} we obtain that
\begin{equation}\label{Eu3}
\begin{aligned}
 \n \Psi \n_{\Winfp{\alpha}([0,T_0]\times\Omega;X)} 
\lesssim n^{-\eta}(1 + \n x_0\n_{L^p(\Omega;X)}),
\end{aligned}
\end{equation}
with implied constants independent of $n$, $T_0$, and $x_0$.\par
\addtocontents{toc}{\SkipTocEntry}\subsection*{Part 2d.}
In this part we provide estimates for the fourth and fifth term in
\eqref{eulersplit}. In order to do so, we shall prove
that there exists an $\eps_1>0$ such that for any $\Phi\in
L^{\infty}(0,T;L^p(\Omega;X_{\theta_F}))$ we have:
\begin{equation}\label{EnConvEst}
\begin{aligned}
\Big\n  s\mapsto \int_{0}^{\ov{s}}\En{\ov{s}-\un{u}}\Phi(u) \,du
\Big\n_{\Winfp{\alpha}([0,T_0]\times\Omega;X)}
& \lesssim \ov{T_0}^{\eps_1} \n \Phi
\n_{L^{\infty}(0,\ov{T_0};L^p(\Omega;X_{\theta_F}))},
\end{aligned}
\end{equation}
with implied constants independent of $n$, $T_0$, and $\Phi$.
\par

Once \eqref{EnConvEst} is obtained, we immediately get: 
\begin{equation}\label{Eu4a}
\begin{aligned}
 & \Big\n  s\mapsto
\int_{0}^{\ov{s}}\En{\ov{s}-\un{u}}[F(u,U^{(n)}(u))-F(u,V^{(n)}(u))] \,du
\Big\n_{\Winfp{\alpha}([0,T_0]\times\Omega;X)}\\
& \qquad \lesssim \ov{T_0}^{\eps_1} \n V^{(n)} -U^{(n)}
\n_{L^{\infty}(0,\ov{T_0};L^p(\Omega;X))},
\end{aligned}
\end{equation}
by \MF{}. Moreover, by \MFII{} we get:
\begin{equation}\label{Eu4b}
\begin{aligned}
 & \Big\n  s\mapsto \int_{0}^{\ov{s}}\En{\ov{s}-\un{u}}[F(u,V^{(n)}(u)) -
F(\un{u},V^{(n)}(u))] \,du
\Big\n_{\Winfp{\alpha}([0,T_0]\times\Omega;X)}\\
& \lesssim \ov{T_0}^{\e_1} n^{-\zeta_{\max}} \big(1+\n
V^{(n)} \n_{L^{\infty}(0,\ov{T_0};L^p(\Omega;X))}\big)\\
& \lesssim \ov{T_0}^{\e_1} \big[
\n V^{(n)} -U^{(n)}\n_{L^{\infty}(0,\ov{T_0};L^p(\Omega;X))} 
+ n^{-\eta}\big(1+\n U^{(n)}\n_{L^{\infty}(0,\ov{T_0};L^p(\Omega;X))}\big)\big]\\
& \lesssim \ov{T_0}^{\e_1} \big[
\n V^{(n)} -U^{(n)}\n_{L^{\infty}(0,\ov{T_0};L^p(\Omega;X))} 
+ n^{-\eta}\big(1+\n x_0\n_{L^p(\O;X)}\big)\big],
\end{aligned}
\end{equation}
the last step being again a consequence of Corollary \ref{cor:Lpest} 
with $\delta = 0$. 
\par

It remains to prove \eqref{EnConvEst}. We fix $\Phi \in
L^{\infty}(0,T;L^p(\Omega;X_{\theta_F}))$.
By the uniform boundedness estimate of Corollary \ref{c:Eulergbdd} with
$\delta=\theta_F$ we obtain:
\begin{equation}\label{Linftyest1d}
\begin{aligned}
& \Big\n  s\mapsto \int_{0}^{\ov{s}}\En{\ov{s}-\un{u}}\Phi(u)\,du
\Big\n_{L^{\infty}(0,T_0;L^p(\Omega;X))}\\
&\qquad  \leq \sup_{0\leq s\leq T_0} \int_{0}^{\ov{s}}\n
\En{\ov{s}-\un{u}}\Phi(u)\n_{L^p(\Omega;X)} \,du \\
&\qquad  \lesssim \sup_{0\leq s\leq T_0}\int_{0}^{\ov{s}}
(\ov{s}-\un{u})^{\theta_F} \,du \n
\Phi\n_{L^{\infty}(0,\ov{T_0};L^p(\Omega;X_{\theta_F}))}\\
&\qquad  \lesssim \ov{T_0}^{1+\theta_F}\n
\Phi\n_{L^{\infty}(0,\ov{T_0};L^p(\Omega;X_{\theta_F}))},
\end{aligned}
\end{equation}
where the last step follows from a similar calculation as in
\eqref{Linftyest1c},
the difference being that now we consider the terms `up to $T_0$'. 
The implied constants are independent of $n$, $T_0$, and $\Phi$.\par

For the estimate in the weighted $\gamma$-space we shall again use Lemma
\ref{lem:wBesovEst}. Define
$\Psi:[0,T_0]\rightarrow L^p(\Omega;X)$ by 
\begin{equation}\label{def-Psi}\Psi(s):=
\int_{0}^{\ov{s}}\En{\ov{s}-\un{u}}\Phi(u)\,du.
\end{equation} Let $t\in [0,T_0]$ and let
$q=(\inv{\tau}-\inv{2}+\frac12\eps )^{-1}$ (so
$\inv{\tau}-\inv{2}<\inv{q}<\inv{\tau}-\alpha$). 
Combining Lemma \ref{lem:wBesovEst} and \eqref{Linftyest1d} we obtain, 
for some $\eps_0>0$:
\begin{equation}\label{wBesovEst2d}
\begin{aligned}
& \sup_{t\in [0,T_0]}\n  s\mapsto (t-s)^{-\alpha}
\Psi(s)\n_{\gamma(0,t;L^p(\Omega;X))}\\
& \qquad\qquad \lesssim T_0^{\eps_0} \big(\n  \Psi
\n_{B_{q,\tau}^{\inv{\tau}-\inv{2}}([0,T_0];L^p(\Omega;X))} + \n \Psi
\n_{L^{\infty}(0,T_0;L^p(\Omega;X))}\big) \\
& \qquad\qquad \lesssim T_0^{\eps_0} \big(\n  \Psi
\n_{B_{q,\tau}^{\inv{\tau}-\inv{2}}([0,T_0];L^p(\Omega;X))} + 
\ov{T_0}^{1+\theta_F}\n \Phi
\n_{L^{\infty}(0,\ov{T_0};L^p(\Omega;X_{\theta_F}))}\big),
\end{aligned}
\end{equation}
with implied constant independent of $T_0$ and $\Psi$.\par

In order to estimate the Besov norm in the right-hand side, let us first fix
$s\in [0,T_0]$ and $k\in
\{1,\ldots,n\}$ such that $s+t_{k}^{(n)}\leq T_0$. We have, using \eqref{defEn},
\begin{equation}\label{BesovEst2d}
\begin{aligned}
\Psi(s+t_{k}^{(n)}) -\Psi(s)
& = \int_{0}^{\ov{s}}(\En{t_{k}^{(n)}}-I)\En{\ov{s}-\un{u}}\Phi(u)\,du\\
& \qquad  + \int_{\ov{s}}^{\ov{s}+t_{k}^{(n)}}
\En{\ov{s}+t_{k}^{(n)}-\un{u}}\Phi(u)\,du.
\end{aligned}
\end{equation}

By Lemma \ref{lem:analyticRbound} (3) and the uniform boundedness result of
Proposition 
\ref{p:Eulergbdd} with $\delta=1+\theta_F-\frac12\eps $ (note that $\d>0$ by
\eqref{eq:cond-eta-eps}) we have:
\begin{align*}\label{EIcomp}
\big\n \En{t_{k}^{(n)}}-I \big\n_{\calL(X_{1+\theta_F-\frac12\eps };X)} & \lesssim
(t_{k}^{(n)})^{1+\theta_F-\frac12\eps }.
\end{align*}\par

Moreover, by the uniform boundedness result in Corollary \ref{c:Eulergbdd} with
$\delta= -1+\frac12\eps $
we have:
\begin{align*}
\big\n \En{t_{k}^{(n)}} \big\n_{\calL(X_{\theta_F};X_{1+\theta_F-\frac12\eps })} &
\lesssim
(t_{k}^{(n)})^{-1+\frac12\eps },
\end{align*}
and with $\delta=\theta_F$:
\begin{align*}
\big\n \En{t_{k}^{(n)}} \big\n_{\calL(X_{\theta_F};X)} & \lesssim
(t_{k}^{(n)})^{\theta_F}.
\end{align*}\par

Thus for the first term in \eqref{BesovEst2d} we have:
\begin{align*}
& \Big\n\int_{0}^{\ov{s}} (\En{t_{k}^{(n)}}-I)\En{\ov{s}-\un{u}}\Phi(u)
 \,du\Big\n_{L^p(\Omega;X)}\\
& \qquad \leq \int_{0}^{\ov{s}} \big\n \En{t_{k}^{(n)}}-I
\big\n_{\calL(X_{1+\theta_F-\frac12\eps };X)} \big\n
\En{\ov{s}-\un{u}}\big\n_{\calL(X_{\theta_F};X_{1+\theta_F-\frac12\eps })}\\
& \qquad \hskip6cm \times \big\n \Phi
\big\n_{L^\infty(0,\ov{T_0};L^p(\Omega;X_{\theta_F}))} \,du\\
& \qquad \lesssim (t_{k}^{(n)})^{1+\theta_F-\frac12\eps } \int_{0}^{\ov{s}}
(\ov{s}-\un{u})^{-1+\frac12\eps } \,du
\big\n
\Phi \big\n_{L^{\infty}(0,\ov{T_0};L^p(\Omega;X_{\theta_{F}}))}\\
& \qquad \lesssim (t_{k}^{(n)})^{1+\theta_F-\frac12\eps } \big\n \Phi
\big\n_{L^{\infty}(0,\ov{T_0};L^p(\Omega;X_{\theta_{F}}))}.
\end{align*}\par
For the second term in \eqref{BesovEst2d} we have:
\begin{align*}
& \Big\n \int_{\ov{s}}^{\ov{s}+t_{k}^{(n)}} 
\En{\ov{s}+t_{k}^{(n)}-\un{u}}\Phi(u)\,du\Big\n_{L^p(\Omega;X)}\\
& \qquad \lesssim \int_{0}^{t_{k}^{(n)}} (t_{k}^{(n)}-\un{u})^{\theta_F} \,du
\big\n
\Phi \big\n_{L^{\infty}(0,\ov{T_0};L^p(\Omega;X_{\theta_F}))}\\
& \qquad \lesssim (t_{k}^{(n)})^{1+\theta_F}\n \Phi
\n_{L^{\infty}(0,\ov{T_0};L^p(\Omega;X_{\theta_F}))}.
\end{align*}\par
Combining the two estimates above we obtain:
\begin{equation}\label{BesovEst2d_2}
\begin{aligned}
 \n \Psi(s+t_{k}^{(n)}) -\Psi(s)\n_{L^p(\Omega;X)} 
\lesssim (t_{k}^{(n)})^{1+\theta_F} \n \Phi
\n_{L^{\infty}(0,\ov{T_0};L^p(\Omega;X_{\theta_F}))}.
\end{aligned}
\end{equation}

This enables us to find the right estimate for the Besov norm in
\eqref{wBesovEst2d}. 
Fix $\rho \in [0,1]$ and
$0<|h|< \rho$. Set $I= [0,T_0]$. Suppose first that $|h|\leq \frac{T}{n}$. In
that case we have, by \eqref{BesovEst2d_2}:
\begin{align*}
& \n T_h^I \Psi(s) - \Psi(s) \n_{L^p(\Omega;X)} \\
& \qquad \leq \left\{ 
\begin{array}{ll} 
0, &  \ov{s+h}=\ov{s}\textrm{ and } s+h\in [0, T_0];\\
\n \Psi(s+\tfrac{T}{n})-\Psi(s) \n_{L^p(\Omega;X)},
& \ov{s+h}\neq \ov{s} \textrm{ and } s+h\in [0, T_0];\\
\n \Psi(s) \n_{L^p(\Omega;X)}, & s+h\notin [0,T_0]
\end{array}\right.\\
& \qquad \lesssim \left\{ 
\begin{array}{ll} 
0, &  \ov{s+h}=\ov{s}\textrm{ and } s+h\in [0, T_0];\\
n^{-1-\theta_F} \n \Phi
\n_{L^{\infty}(0,\ov{T_0};L^p(\Omega;X_{\theta_F}))},
& \ov{s+h}\neq \ov{s} \textrm{ and } s+h\in [0, T_0];\\
\n \Phi
\n_{L^{\infty}(0,\ov{T_0};L^p(\Omega;X_{\theta_F}))}, & s+h\notin [0,T_0].
\end{array}\right.
\end{align*}
In the above we used $\n \Psi(s) \n_{L^p(\Omega;X)} \leq \n
\Psi\n_{L^{\infty}(0,T_0;L^p(\Omega;X))}$ and
\eqref{Linftyest1d}.\par
Define $I_h=\{ s\in [0,T_0]\,:\, \ov{s+h}\neq \ov{s}\}$ and observe that
$|I_h|\leq n|h|$. 
Moreover $|\{s\in [0,T_0]\,:\,
s+h\notin[0,T_0]\}|\leq |h|$. Thus by the definition of $q$ and by 
\eqref{BesovEst2d_2},
\begin{align*}
& \n T_h^I \Psi - \Psi \n_{L^{q}(0,T_0;L^p(\Omega;X))} \\
& \quad \lesssim \big[(n|h|)^{\inv{q}} n^{-1-\theta_F} + |h|^\frac1q\big]\n \Phi
\n_{L^{\infty}(0,\ov{T_0};L^p(\Omega;X_{\theta_F}))}\\
& \quad \lesssim |h|^{\inv{\tau}-\inv{2}+\e/2}\n \Phi
\n_{L^{\infty}(0,\ov{T_0};L^p(\Omega;X_{\theta_F}))}.
\end{align*}

Next let $|h|>\frac{T}{n}$. Then, $|h|/2< \un{|h|}\leq |h|$ and
$|h|<\ov{|h|}\leq 2|h|$.
Let us deal with the case $h>\frac{T}{n}$; the case $-h>\frac{T}{n}$ is dealt
with entirely analogously.
 It follows from the definition of $\Phi$ in \eqref{def-Psi}
that for each $s\in [0,T_0]$ we either have 
$\Psi(\ov{s+h}) = \Psi(\ov{s}+\ov{h})$ or   
$\Psi(\ov{s+h}) = \Psi(\ov{s}+\un{h})$. Hence, by \eqref{BesovEst2d_2}:
\begin{align*}
& \n T_h^I \Psi - \Psi \n_{L^{q}(0,T_0;L^p(\Omega;X))} \\
& \qquad \leq \n 1_{\{s+h\not\in [0,T_0]\}}\Psi\n_{L^{q}(0,T_0;L^p(\Omega;X))}
\\
& \qquad\qquad +
\n T_{\ov{h}}^I \Psi - \Psi \n_{L^{q}(0,T_0;L^p(\Omega;X))} + \n T_{\un{h}}^I
\Psi - \Psi
\n_{L^{q}(0,T_0;L^p(\Omega;X))}
 \\
& \qquad \lesssim \big( h^\frac1q + \un{h}^{1+\theta_F} + \ov{h}^{1+\theta_F}
\big) \n \Phi
\n_{L^{\infty}(0,\ov{T_0};L^p(\Omega;X_{\theta_F}))}\\
& \qquad \lesssim h^{\inv{\tau}-\inv{2}+\frac12\eps }\n \Phi
\n_{L^{\infty}(0,\ov{T_0};L^p(\Omega;X_{\theta_F}))},
\end{align*}
where we use that $1+\theta_F\ge \inv{\tau}-\inv{2}+\frac12\eps $ (by
\eqref{eq:cond-eta-eps}).\par
Thus we have:
\begin{align*}
\sup_{|h|\leq \rho}\n T_h \Psi - \Psi \n_{L^{q}(0,t;L^p(\Omega;X))} 
& \lesssim \rho^{\inv{\tau}-\inv{2}+\frac12\eps }\n \Phi
\n_{L^{\infty}(0,\ov{T_0};L^p(\Omega;X_{\theta_F}))}.
\end{align*}
With \eqref{Linftyest1d} it follows that 
\begin{align*}
\n \Psi \n_{B_{q,\tau}^{\inv{\tau}-\inv{2}}([0,t];L^p(\Omega;X))} & \lesssim 
\n \Psi \n_{{L^{q}(0,T_0;L^p(\Omega;X))}} +
\n \Phi \n_{L^{\infty}(0,\ov{T_0};L^p(\Omega;X_{\theta_F}))} 
\\ & \lesssim \n \Psi \n_{{L^{\infty}(0,T_0;L^p(\Omega;X))}} +
\n \Phi \n_{L^{\infty}(0,\ov{T_0};L^p(\Omega;X_{\theta_F}))}
\\ & \lesssim 
\n \Phi \n_{L^{\infty}(0,\ov{T_0};L^p(\Omega;X_{\theta_F}))},
\end{align*}
and thus, by \eqref{wBesovEst2d},
\begin{equation}\label{BesovEst1d}
\begin{aligned}
\sup_{t\in [0,T_0]}\n  s\mapsto (t-s)^{-\alpha} \Psi(s)
\n_{\gamma(0,t;L^p(\Omega;X))}
& \lesssim T_0^{\eps_0} \n \Phi \n_{L^{\infty}(0,\ov{T_0};L^p(\Omega;X_{\theta_F}))}.
\end{aligned}
\end{equation}
This completes the proof of \eqref{EnConvEst}.

\addtocontents{toc}{\SkipTocEntry}\subsection*{Part 2e}
For the sixth term we again use Besov embeddings. First of all observe that by
\eqref{analyticDiff}, the linear growth condition on $F$ in \MF{} and Corollary
\ref{cor:Lpest}, for all $s\in [0,T_0]$ we have:
\begin{align*}
\Big\n \int_{s}^{\ov{s}} 
S(\tfrac{T}{n})F(u,U^{(n)}(u))\,du \Big\n_{L^{p}(\Omega;X)} 
& \lesssim (\tfrac{T}{n})^{\theta_F}  \int_{s}^{\ov{s}} 
\n F(u,U^{(n)}(u))\n_{L^{p}(\Omega;X_{\theta_F})} \,du 
\\ & \lesssim 
(\tfrac{T}{n})^{\theta_F} \int_{s}^{\ov{s}}  \n U^{(n)}(u)\n_{L^{p}(\Omega;X)}
\,du 
\\ & \lesssim (\ov{s}-s) 
n^{-\theta_F}(1+\n x_0\n_{L^p(\Omega;X)})
\\ & \le n^{-1-\theta_F}(1+\n x_0\n_{L^p(\Omega;X)})
\end{align*}
and therefore
\begin{equation}\label{Linfty1e}
\Big\n s\mapsto \int_{s}^{\ov{s}} 
S(\tfrac{T}{n})F(u,U^{(n)}(u))\,du \Big\n_{L^{\infty}(0,T_0;L^{p}(\Omega;X))} 
\lesssim
n^{-1-\theta_F}(1+\n x_0\n_{L^p(\Omega;X)}).
\end{equation}
Similarly, for $h\in (0,\frac{T}{n}]$:
\begin{equation}\label{Linftyh1e}
\begin{aligned}
& \Big\n s\mapsto \int_{s}^{s+h} S(\tfrac{T}{n})F(u,U^{(n)}(u))\,du
\Big\n_{L^{\infty}(0,T_0;L^{p}(\Omega;X))} 
\\ & \qquad \lesssim hn^{-\theta_F}(1+\n x_0\n_{L^p(\Omega;X)}) \\
& \qquad \leq
h^{\inv{\tau}-\inv{2}+\frac12\eps }n^{-\frac{3}{2}+\inv{\tau}-\theta_F+\frac12\eps }(1+\n
x_0\n_{L^p(\Omega;X)}).
\end{aligned}
\end{equation}
Define $\Psi:[0,T_0]\rightarrow L^p(\Omega;X)$ by 
$$\Psi(s):=  \int_{s}^{\ov{s}} S(\tfrac{T}{n})F(u,U^{(n)}(u))\,du.$$ Fix
$\rho
\in [0,1]$ and let $0\le h<\rho$ (the case that $-\rho < h \le 0$ is entirely
analogous). Suppose 
first that $h\leq \frac{T}{n}$.
Then, for $s\in I:= [0,T_0]$:
\begin{align*}
& \n T_h^I \Psi(s) - \Psi(s) \n_{L^p(\Omega;X)} \\
& \qquad \lesssim \left\{ 
\begin{aligned} 
& \Big\n s\mapsto \int_{s}^{s+h} S(\tfrac{T}{n})F(u,U^{(n)}(u))\,du
\Big\n_{L^{p}(\Omega;X)}, 
&&  \ov{s+h}=\ov{s}, \ s+h\in [0,T_0];\\
& 2\n \Psi \n_{L^{\infty}(0,T_0;L^p(\Omega;X))},
&& \hbox{otherwise}. 
\end{aligned}
\right.
\end{align*}
Recall that 
$|\{s\in [0,T_0]: \ov{s}\neq\ov{s+h}\}|\leq nh$
and 
$|\{s\in [0,T_0]: {s+h}\not\in [0,T_0]\}|\leq h$. 
Let $q=(\inv{\tau}-\inv{2}+\frac12\eps )^{-1}$. By \eqref{Linfty1e}
and \eqref{Linftyh1e} we have:
\begin{align*}
&\n T_h^I \Psi - \Psi \n_{L^q(0,T_0;L^p(\Omega;X))} \\
& \qquad \lesssim
(h^{\inv{\tau}-\inv{2}+\frac12\eps }n^{-\frac{3}{2}+\inv{\tau}-\theta_F+\frac12\eps }
+((n+1)h)^{\inv{q}}n^{-1-\theta_F}
)(1+\n x_0\n_{L^{p}(\Omega;X)}) \\
& \qquad \lesssim h^{\inv{\tau}-\inv{2}+\frac12\eps }n^{-1-\theta_F+\inv{q}} (1+\n
x_0\n_{L^{p}(\Omega;X)}).
\end{align*}
On the other hand, if $h>\frac{T}{n}$, then by \eqref{Linfty1e}:
\begin{align*}
\n T_h^I \Psi - \Psi \n_{L^q(0,T_0;L^p(\Omega;X))} & \lesssim 2\n \Psi
\n_{L^q(0,T_0;L^p(\Omega;X))} \\
& \lesssim n^{-1-\theta_F}(1+\n x_0\n_{L^p(\Omega;X)})
\\
& \lesssim h^{\inv{\tau}-\inv{2}+\frac12\eps } n^{-1-\theta_F+\inv{q}} (1+\n
x_0\n_{L^p(\Omega;X)}).
\end{align*}
As in Part 2c, using that by \eqref{eq:cond-eta-eps} we have
$\eta + \eps< \frac32-\frac1\tau+\theta_F < 1 + \theta_F - \frac1q + \eps$, this
implies 
\begin{align*}
\n \Psi \n_{B^{\inv{\tau}-\inv{2}}_{q,\tau}([0,T_0];L^p(\Omega;X))} & \lesssim
n^{-\eta}(1+\n x_0\n_{L^{p}(\Omega;X)}),
\end{align*}
with implied constants independent of $n$, $x_0$ and $T_0$. By Lemma
\ref{lem:wBesovEst} it now follows that
\begin{align*}
& \sup_{t\in [0,T_0]} \Big\n s\mapsto (t-s)^{-\alpha}\int_{s}^{\ov{s}}
S(\tfrac{T}{n})F(u,U^{(n)}(u))\,du
\Big\n_{\gamma(0,t;L^p(\Omega;X))} \\
&\qquad \qquad  \lesssim n^{-\eta}(1+\n x_0\n_{L^{p}(\Omega;X)}),
\end{align*}
with implied constants independent of $n$, $x_0$ and $T_0$. Combining this with
\eqref{Linfty1e} we obtain:
\begin{equation}\label{Eu5}
\Big\n s\mapsto \int_{s}^{\ov{s}} S(\tfrac{T}{n})F(u,U^{(n)}(u))\,du 
\Big\n_{\Winfp{\alpha}([0,T_0]\times\Omega;X)} 
\lesssim
n^{-\eta}(1+\n x_0\n_{L^{p}(\Omega;X)}),
\end{equation}
with implied constants independent of $n$, $x_0$ and $T_0$.

\addtocontents{toc}{\SkipTocEntry}\subsection*{Part 2f.}
By Theorem \ref{t:stochint} we have, for any $s\in [0,T_0]$:
\begin{align*}
& \Big\n 
\int_{0}^{\ov{s}}[S(\ov{s}-\un{u})-\En{\ov{s}-\un{u}}]G(u,U^{(n)}(u))\,dW_H(u)
\Big\n_{L^p(\Omega;X)}\\
& \qquad \qquad \lesssim \big\n u\mapsto
[S(\ov{s}-\un{u})-\En{\ov{s}-\un{u}}]G(u,U^{(n)}(u))\big\n_{L^p(\Omega;\gamma(0,
\ov{s};H,X))}.
\end{align*}
By the second part of Proposition \ref{p:Eulergbdd} with $\delta=\theta_G$,
$\epsilon= \frac13\eps$, 
$\beta=\inv{2}-\frac23\eps$, and Theorem \ref{t:KW} and \eqref{eq:cond-eta-eps} we have: 
\begin{equation*}
\begin{aligned}
& \big\n u\mapsto
[S(\ov{s}-\un{u})-\En{\ov{s}-\un{u}}]G(u,U^{(n)}(u))\big\n_{L^p(\Omega;\gamma(0,
\ov{s};H,X))} \\
& \qquad \lesssim n^{-\inv{2}-\theta_G+\eps} \big\n u\mapsto
(\ov{s}-\un{u})^{-\inv{2}+\frac23\eps}
G(u,U^{(n)}(u))\big\n_{L^p(\Omega;\gamma(0,\ov{s};H,X_{\theta_G}))}\\
& \qquad \lesssim n^{-\eta} \big\n
U^{(n)}\big\n_{\Winfp{\frac12-\frac23\eps}([0,\ov{T_0}]\times\Omega;X)}\\
& \qquad \lesssim n^{-\eta} \big\n
U^{(n)}\big\n_{\Winfp{\alpha}([0,\ov{T_0}]\times\Omega;X)}\\
& \qquad \lesssim n^{-\eta} (1+\n x_0 \n_{L^p(\Omega;X)}),
\end{aligned}
\end{equation*}
where we also used \MG{} in the sense of \eqref{GLipschitzV1}, the fact that
$\alpha>\inv{2}-\frac23\eps$ (whence $\Winfp{\alpha}([0,\ov{T_0}]\times\Omega;X)
\hookrightarrow
\Winfp{\inv{2}-\frac23\eps }([0,\ov{T_0}]\times\Omega;X)$), and we used Corollary
\ref{cor:Lpest}. Note that the implied
constants are independent of $n$, $T_0$ and
$x_0$. As $s\in [0,T_0]$ was arbitrary, it follows that
\begin{equation}\label{eu1fLinfty}
\begin{aligned}
& \Big\n  s\mapsto
\int_{0}^{\ov{s}}[S(\ov{s}-\un{u})-\En{\ov{s}-\un{u}}]G(u,U^{(n)}(u))\,dW_H(u)
\Big\n_{L^{\infty}(0,T_0;L^p(\Omega;X))}\\
& \qquad \lesssim n^{-\eta} (1+\n x_0 \n_{L^p(\Omega;X)}),
\end{aligned}
\end{equation}\par
Next we estimate the part concerning the weighted $\gamma$-radonifying norm. We
begin by recalling that, since $X$ is a
\textsc{umd} Banach space, $\gamma(0,t;H,X)$ is a \textsc{umd} Banach space for
any
$t>0$ (by noting that this space embeds into $L^2(\tilde\O;X)$ isometrically
whenever $(\tilde \O,\tilde \P)$
is a probability space supporting a Gaussian sequence; see, e.g.,
\cite{Nee-survey}).
Thus, by Theorem \ref{t:stochint} (applied with state space $\gamma(0,\ov{t};X)$)
and isomorphism
\eqref{ggiso}, for all $t\in [0,T_0]$ we obtain
\begin{align*}
& \Big\n  s\mapsto (t-s)^{-\alpha}
\int_{0}^{\ov{s}}[S(\ov{s}-\un{u})-\En{\ov{s}-\un{u}}]G(u,U^{(n)}(u))\,dW_H(u)
\Big\n_{L^p(\Omega;\gamma(0,t;X))}\\
& \quad = \Big\n  \int_{0}^{\ov{t}} \big[s\mapsto 1_{\{0\leq u\leq
\ov{s}\}}(t-s)^{-\alpha}  \\
& \quad \qquad \qquad \times
[S(\ov{s}-\un{u})-\En{\ov{s}-\un{u}}]G(u,U^{(n)}(u))\big] \,dW_H(u)
\Big\n_{L^p(\Omega;\gamma(0,t;X))}\\
& \quad \eqsim \big\n  u\mapsto \big( s\mapsto 1_{\{0\leq u\leq
\ov{s}\}}(t-s)^{-\alpha}\\
& \quad \qquad \qquad \times
[S(\ov{s}-\un{u})-\En{\ov{s}-\un{u}}]G(u,U^{(n)}(u))\big)\big\n_{
L^p(\Omega;\gamma(0,\ov{t};H,\gamma(0,{t};X)))}\\
& \quad \eqsim \big\n  (s,u)\mapsto 1_{\{0\leq u\leq \ov{s}\}} (t-s)^{-\alpha}\\
& \quad \qquad \qquad \times 
[S(\ov{s}-\un{u})-\En{\ov{s}-\un{u}}]G(u,U^{(n)}(u))\big\n_{L^p(\Omega;\gamma([0
,t]\times [0,\ov{t}];H,X))}.
\end{align*}
By the second part of Proposition \ref{p:Eulergbdd} with $\delta=\theta_G$,
$\epsilon=\frac12\eps$, $\beta=\inv{2}-\frac12\eps$, Theorem
\ref{t:KW}, isomorphism \eqref{ggiso}, once again Theorem \ref{t:KW} combined with Theorem
\ref{t.KaiWei}, 
Lemma \ref{lem:L2unif}, Corollary
\ref{cor:Lpest} and \eqref{eq:cond-eta-eps} we have: 
\begin{align*}
&\big\n  (s,u)\mapsto 1_{\{0\leq u\leq \ov{s}\}}(t-s)^{-\alpha} \\
&\quad \qquad \qquad \times
[S(\ov{s}-\un{u})-\En{\ov{s}-\un{u}}]G(u,U^{(n)}(u))\big\n_{L^p(\Omega;\gamma([0
,t]\times
[0,\ov{t}];H,X))}\\
&\quad \lesssim n^{-\inv{2}-\theta_G+\eps}\big\n  (s,u)\mapsto 1_{\{0\leq u\leq
\ov{s}\}} (t-s)^{-\alpha} \\
&\quad \qquad \qquad \times
(\ov{s}-\un{u})^{-\inv{2}+\frac12\eps}G(u,U^{(n)}(u)))\big\n_{L^p(\Omega;\gamma(
[0,t]\times
[0,\ov{t}];H,X_{\theta_G}))}\\
&\quad \lesssim n^{-\eta}\big\n  u\mapsto \big( s\mapsto 1_{\{0\leq u\leq
\ov{s}\}}(t-s)^{-\alpha} \\
&\quad \qquad \qquad \times (\ov{s}-\un{u})^{-\inv{2}+\frac12\eps}
G(u,U^{(n)}(u))\big)\big\n_{L^p(\Omega;\gamma(0,\ov{t};H,\gamma(0,t;X_{\theta_G}
)))}\\
&\quad \lesssim
n^{-\eta}\sup_{u\in[0,\ov{t}]}\left\{(t+\tfrac{T}{n}-u)^{\alpha}\n s\mapsto
(t-s)^{-\alpha}(\ov{s}-\un{u})^{-\inv{2}+\frac12\eps} \n_{L^2(\un{u},t)}\right\}
\\
& \quad \qquad \qquad \times \big\n u\mapsto (t+\tfrac{T}{n}-u)^{-\alpha}
G(u,U^{(n)}(u))\big\n_{L^p(\Omega;\gamma(0,\ov{t};H,X_{\theta_G}))}\\
&\quad \lesssim n^{-\eta}(1+ \n U^{(n)}
\n_{\Winfp{\alpha}([0,\ov{t}]\times\Omega;X)}) \\
& \quad \lesssim n^{-\eta}(1 + \n x_0\n_{L^p(\Omega;X)}).
\end{align*}
The implied constants above are independent of $x_0$, $n$, $t$ and $T_0$.\par
Combining the above with \eqref{eu1fLinfty} above one obtains:
\begin{equation}
\begin{aligned}\label{Eu6}
&\Big\n  s\mapsto
\int_{0}^{\ov{s}}[S(\ov{s}-\un{u})-\En{\ov{s}-\un{u}}]G(u,U^{(n)}(u))\,dW_H(u)
\Big\n_{\Winfp{\alpha}([0,T_0]\times \Omega;X)}\\
& \qquad \qquad \lesssim n^{-\eta}(1 + \n x_0\n_{L^p(\Omega;X)}),
\end{aligned}
\end{equation}
with implied constant independent of $n$, $T_0$ and $x_0$.
\addtocontents{toc}{\SkipTocEntry}\subsection*{Part 2g.}
The estimate for the eighth and ninth term in \eqref{eulersplit} is similar to
Part 2f, except that one needs to keep
track of dependence on $T_0$.\par

We shall prove that for any $\Phi\in L^p(\Omega;\gamma(0,T;H,X_{\theta_G}))$
we have:
\begin{equation}\label{EnStConvEst}
\begin{aligned}
& \Big\n  \int_{0}^{\ov{s}}\En{\ov{s}-\un{u}}\Phi(u)\,dW_H(u)
\Big\n_{\Winfp{\alpha}([0,\ov{T_0}]\times
\Omega;X_{\theta_{G}})}\\
& \qquad \lesssim \ov{T_0}^{\inv{2}-\theta_G-\eps}\sup_{0\leq t\leq
\ov{T_0}}\big\n s\mapsto (t-s)^{-\alpha}  \Phi(s)
\big\n_{L^p(\Omega;\gamma(0,t;H,X_{\theta_{G}}))},
\end{aligned}
\end{equation}
with implied constant independent of $n$ and $T_0$, provided the right-hand side above is
finite.\par

The estimate for the eighth term in \eqref{eulersplit} follows immediately
from \eqref{EnStConvEst} and \eqref{GLipschitzV1} (i.e., \MG{}): 
\begin{equation}\label{Eu7a}
\begin{aligned}
&\Big\n  s\mapsto
\int_{0}^{\ov{s}}\En{\ov{s}-\un{u}}[G(u,U^{(n)}(u))-G(u,V^{(n)}(u))]\,dW_H(u)
\Big\n_{\Winfp{\alpha}([0,\ov{T_0}]\times \Omega;X)}\\
& \qquad \lesssim \ov{T_0}^{\inv{2}-\theta_G-\eps}\n
U^{(n)}-V^{(n)}\n_{\Winfp{\alpha}(
[0,\ov{T_0}]\times \Omega;X)}.
\end{aligned}
\end{equation}\par

The estimate for the ninth term in \eqref{eulersplit} follows immediately 
from \eqref{EnStConvEst}
in combination with \eqref{eq:cond-eta-eps} and Lemma \ref{l:GgammaEst} 
(i.e., \MGII{}), with $B_j=V^{(n)}_j$ noting that $V^{(n)}(u)=
V^{(n)}(\un{u})=V^{(n)}_{\un{u}n/T}$:
\begin{equation}\label{Eu7b}
\begin{aligned}
&\Big\n  s\mapsto
\int_{0}^{\ov{s}}\En{\ov{s}-\un{u}}[G(u,V^{(n)}(u))-G(\un{u},V^{(n)}(u))]\,
dW_H(u)
\Big\n_{\Winfp{\alpha}([0,\ov{T_0}]\times \Omega;X)}\\
& \ \lesssim \ov{T_0}^{\inv{2}-\theta_G-\eps}n^{-\eta}\big( 1 +\n
V^{(n)}\n_{\Winfp{\alpha}([0,\ov{T_0}]\times
\Omega;X)}\big)\\
& \ \lesssim \ov{T_0}^{\inv{2}-\theta_G-\eps}\n U^{(n)} -
V^{(n)}\n_{\Winfp{\alpha}([0,\ov{T_0}]\times \Omega;X)} +
n^{-\eta} \big(1+\n U^{(n)}\n_{\Winfp{\alpha}([0,\ov{T_0}]\times
\Omega;X)}\big).
\end{aligned}
\end{equation}\par

It remains to prove \eqref{EnStConvEst}. By Theorem \ref{t:stochint} we have,
for $s\in [0,T_0]$:
\begin{align*}
\Big\n  \int_{0}^{\ov{s}}\En{\ov{s}-\un{u}}\Phi(u)\,dW_H(u)
\Big\n_{L^p(\Omega;X)}
& \lesssim \big\n  u\mapsto
\En{\ov{s}-\un{u}}\Phi(u)\big\n_{L^p(\Omega;\gamma(0,\ov{s};X))}.
\end{align*}\par

By the second part of Corollary \ref{c:Eulergbdd} with $\delta=\theta_G$,
$\epsilon=\frac13\eps$, $\beta=\inv{2}-\frac23\eps$, and Theorem
\ref{t:KW} we have:
\begin{equation*}
\begin{aligned}
& \big\n u\mapsto
\En{\ov{s}-\un{u}}\Phi(u)\big\n_{L^p(\Omega;\gamma(0,\ov{s};X))} \\
& \qquad \qquad  \lesssim \ov{T_0}^{\inv{2}+\theta_G-\eps} \big\n  u\mapsto
(\ov{s}-\un{u})^{-\alpha}\Phi(u)\big\n_{L^p(\Omega;\gamma(0,\ov{s};X_{\theta_G}
))},
\end{aligned}
\end{equation*}
where we used that
$\alpha>\inv{2}-\frac23\eps$. Note that the implied constants are independent of
$n$ and $T_0$. As $s\in [0,T_0]$
was arbitrary, it follows that:
\begin{equation}\label{eu1gLinfty}
\begin{aligned}
& \Big\n  s\mapsto \int_{0}^{\ov{s}}\En{\ov{s}-\un{u}}\Phi(u) \,dW_H(u)
\Big\n_{L^{\infty}(0,T_0;L^p(\Omega;X))}\\
& \quad \lesssim \ov{T_0}^{\inv{2}+\theta_G-\eps} \sup_{0\leq t\leq \ov{T_0}}
\big\n s\mapsto (t-s)^{-\alpha} \Phi(s)
\big\n_{L^p(\Omega;\gamma(0,\ov{T_0};H,X_{\theta_{G}}))}.
\end{aligned}
\end{equation}\par

As for the part concerning the weighted $\gamma$-radonifying norm, as before we
have, for $t\in [0,T_0]$:
\begin{align*}
& \Big\n  s\mapsto (t-s)^{-\alpha} \int_{0}^{\ov{s}}\En{\ov{s}-\un{u}}
\Phi(u)\,dW_H(u)
\Big\n_{L^p(\Omega;\gamma(0,t;X))}\\
& \quad \lesssim \big\n (s,u) \mapsto 1_{\{0\leq u\leq
\ov{s}\}}(t-s)^{-\alpha}\En{\ov{s}-\un{u}}\Phi(u)
\big\n_{L^p(\Omega;\gamma([0,t]\times
[0,\ov{t}];H,X))}.
\end{align*}\par

By the second part of Corollary \ref{c:Eulergbdd} with $\delta=\theta_G$,
$\epsilon= \frac12 \eps$, 
$\beta=\inv{2}-\frac12\eps$, Theorem
\ref{t:KW}, isomorphism \eqref{ggiso}, once again Theorem \ref{t:KW} combined
with Theorem \ref{t.KaiWei} and Lemma \ref{lem:L2unif} we have:
\begin{align*}
&\big\n  (s,u) \mapsto 1_{\{0\leq u\leq \ov{s}\}}(t-s)^{-\alpha}\En{\ov{s}-\un{u}}
\Phi(u) \big\n_{L^p(\Omega;\gamma([0,t]\times [0,\ov{t}];H,X))}\\
&\quad  \lesssim \ov{T_0}^{\inv{2}+\theta_G-\eps}\big\n  (s,u) \mapsto
1_{\{0\leq u\leq
\ov{s}\}}\\
& \qquad\qquad \qquad \qquad \times
(t-s)^{-\alpha}(\ov{s}-\un{u})^{-\inv{2}+\frac12\eps}
\Phi(u)\big)\big\n_{L^p(\Omega;\gamma([0,t]\times
[0,\ov{t}];H,X_{\theta_G}))}\\
&\quad  \lesssim \ov{T_0}^{\inv{2}+\theta_G-\eps} \big\n
u\mapsto (t+\tfrac{T}{n}-u)^{-\alpha}\Phi(u)
\big\n_{L^p(\Omega;\g(0,\ov{t};X))}\\
&\quad  \lesssim \ov{T_0}^{\inv{2}+\theta_G-\eps} \big\n
u\mapsto (t-u)^{-\alpha}\Phi(u) \big\n_{L^p(\Omega;\g(0,\ov{t};X))}.
\end{align*}
Taking the supremum over $t\in [0,T_0]$ and combining the above with \eqref{eu1gLinfty} one
obtains
\eqref{EnStConvEst}.
\addtocontents{toc}{\SkipTocEntry}\subsection*{Part 2h.}
As for the final term in \eqref{eulersplit}, first observe that because 
$\theta_G\leq 0$ we have, by
\eqref{analyticDiff}:
\begin{equation}\label{eq:2h-begin}\begin{aligned}
& \Big\n  s\mapsto \int_{s}^{\ov{s}} S(\tfrac{T}{n})G(u,U^{(n)}(u))\,dW_H(u)
\Big\n_{\Winfp{\alpha}([0,T_0]\times
\Omega;X)}
\\
&\qquad \lesssim n^{-\theta_G} \Big\n s\mapsto
\int_{s}^{\ov{s}}G(u,U^{(n)}(u))\,dW_H(u)
\Big\n_{\Winfp{\alpha}([0,T_0]\times \Omega;X_{\theta_G})},
\end{aligned}
\end{equation}
with implied constants independent of $n$ and $x_0$.\par

By \eqref{stochIntCont} (take $\tilde{\alpha}=\inv{2}-\eps/2$ and $\tilde{\eps}=\eps/2$), 
\eqref{GLipschitzV1}, and Corollary
\ref{cor:Lpest} we have, for $s\in [0,T_0]$:
\begin{equation}\label{eq:2h}
\begin{aligned}
&\Big\n \int_{s}^{\ov{s}}G(u,U^{(n)}(u))\,dW_H(u)
\Big\n_{L^{p}(\Omega;X_{\theta_G})}\\
&\qquad \leq (\ov{s}-s)^{-\inv{2}+\eps}  \Big\n s\mapsto
\int_{0}^{s}G(u,U^{(n)}(u))\,dW_H(u)
\Big\n_{C^{\inv{2}-\eps/2}(0,T_0;L^p(\Omega;X_{\theta_G}))}\\
&\qquad \lesssim n^{-\inv{2}+\eps}\sup_{0\le t\le T_0}\n u\mapsto
(t-u)^{-\frac12+\frac12\e} 
G(u,U^{(n)}(u))\n_{L^p(\O;\g(0,t;H,X_{\theta_G}))}\\
&\qquad \lesssim n^{-\inv{2}+\eps}\big(1+\n
U^{(n)}\n_{\Vpinf{\inv{2}-\frac12\eps}([0,T_0]\times \Omega;X)}\big)\\
& \qquad \lesssim n^{-\inv{2}+\eps}\big(1+\n x_0 \n_{L^p(\Omega;X)}\big),
\end{aligned}\end{equation}
with implied constants independent of $n$, $T_0$, and $x_0$. We have shown that
\begin{equation}\label{eq:2h-1}\Big\n s\mapsto
\int_{s}^{\ov{s}}G(u,U^{(n)}(u))\,dW_H(u)
\Big\n_{L^\infty(0,T_0;L^p(\Omega;X_{\theta_G}))} \lesssim
n^{-\inv{2}+\eps}\big(1+\n x_0 \n_{L^p(\Omega;X)}\big).
\end{equation}\par

Next fix $t\in [0,T_0]$. By Lemma \ref{lem:h1} 
(with $R=(0,1)$ and $S = (0,t)$ with the Lebesgue measure, 
$f(r,u)(s) = (t-s)^{-\alpha}(t-u)^{\alpha}1_{\{s \leq u\leq \ov{s} \}}$, 
$\Phi_2\equiv I$  and $\Phi_1(u)=(t-u)^{-\a}G(u,U^{(n)}(u))$) 
we obtain 
\begin{align*}
& \Big\n s\mapsto (t-s)^{-\alpha}\int_{s}^{\ov{s}} G(u,U^{(n)}(u))\,dW_H(u)
\Big\n_{L^p(\Omega;\gamma(0,t;X_{\theta_G}))}\\
& \qquad \lesssim \sup_{u\in [0,t]}(t-u)^{\alpha}\n s\mapsto
(t-s)^{-\alpha}1_{\{ \un{u}\leq s \leq u\}}
\n_{L^2(0,t)}\\
& \qquad \hskip3cm \times \n u\mapsto
(t-u)^{-\alpha}G(u,U^{(n)}(u))\n_{L^p(\Omega;\gamma(0,t;H,X_{\theta_G}))}\\
& \qquad \lesssim n^{-\inv{2}} \n u\mapsto
(t-u)^{-\alpha}G(u,U^{(n)}(u))\n_{L^p(\Omega;\gamma(0,t;H,X_{\theta_G}))},
\end{align*}
with implied constants independent of $x_0$, $n$ and $T_0$.\par
From here we proceed as in \eqref{eq:2h} and take the supremum over $t\in
[0,T_0]$ 
to arrive at the estimate
\begin{equation}\label{eq:2h-2}
\begin{aligned}
&  \sup_{t\in [0,T_0]}
 \Big\n s\mapsto (t-s)^{-\alpha}\int_{s}^{\ov{s}} G(u,U^{(n)}(u))\,dW_H(u)
\Big\n_{L^p(\Omega;\gamma(0,t;X_{\theta_G}))}
\\ & \qquad \qquad \lesssim n^{-\inv{2}} \big(1+\n x_0\n_{L^p(\O;X)}\big). 
\end{aligned}\end{equation}\par

Combining \eqref{eq:2h-1} and \eqref{eq:2h-2} with \eqref{eq:2h-begin} and recalling 
\eqref{eq:cond-eta-eps}
we obtain:
\begin{equation}\begin{aligned}\label{Eu8}
& \Big\n  s\mapsto \int_{s}^{\ov{s}} S(\tfrac{T}{n})G(u,U^{(n)}(u))\,dW_H(u)
\Big\n_{\Winfp{\alpha}([0,T_0]\times
\Omega;X)}
\\
&\qquad \lesssim n^{-\inv{2}-\theta_G+\inv{p}+\eps} \big(1+\n x_0
\n_{L^p(\Omega;X)}\big) 
\\ & \qquad \leq n^{-\eta}\big(1+\n x_0
\n_{L^p(\Omega;X)}\big).
\end{aligned}\end{equation}
\addtocontents{toc}{\SkipTocEntry}\subsection*{Part 3.}
By combining equations \eqref{eulersplit}, \eqref{Eu1}, \eqref{Eu2},
\eqref{Eu3}, \eqref{Eu4a}, \eqref{Eu4b}, \eqref{Eu5}, \eqref{Eu6}, \eqref{Eu7a},
\eqref{Eu7b} and
\eqref{Eu8}, we obtain that there exist constants $C>0$ and $\epsilon>0$,
independent of $n$,
$x_0$ and $y_0$, such that for all $n\geq N$ and $T_0\in (0,T]$:
\begin{equation}\label{EulerComb}
\begin{aligned}
& \n U^{(n)}-V^{(n)}\n_{\Winfp{\alpha}([0,\ov{T_0}]\times\Omega;X)}  \leq C \n
x_0-y_0\n_{L^p(\Omega;X)} \\
& \quad + C n^{-\eta}(1+\n x_0\n_{L^p(\Omega;X_{\eta})}) + C
\ov{T_0}^{\epsilon}\n
U^{(n)}-V^{(n)} \n_{\Winfp{\alpha}([0,\ov{T_0}]\times \Omega;X)}.
\end{aligned}
\end{equation}\par

Define $c_0=\inv{2}(2C)^{-\inv{\epsilon}}$
and let $N_0\in \N$ be such that $N_0>\max\{N,T/c_0\}$, this implies that for $n\geq N_0$ we
have $c_0\leq \ov{c_0}\leq
2c_0$, and thus $\ov{c_0}^{\epsilon}\leq (2c_0)^{\epsilon} = (2C)^{-1}$. For
$n\geq N_0$ we obtain by taking $T_0=\ov{c_0}$ in \eqref{EulerComb}:
\begin{equation*}
\begin{aligned}
\n U^{(n)}-V^{(n)}\n_{\Winfp{\alpha}([0,\ov{c_0}]\times\Omega;X)}
&\leq 2C\big(\n x_0-y_0\n_{L^p(\Omega;X)} + n^{-\eta}(1+\n
x_0\n_{L^p(\Omega;X_{\eta})})\big),
\end{aligned}
\end{equation*}
and thus there exists a constant $\tilde{C}$ such that for all $n\geq N$ we
have:
\begin{equation*}
\begin{aligned}
\n U^{(n)}-V^{(n)}\n_{\Winfp{\alpha}([0,c_0]\times\Omega;X)} \leq \tilde{C}\big(\n
x_0-y_0\n_{L^p(\Omega;X)} + n^{-\eta}(1+\n
x_0\n_{L^p(\Omega;X_{\eta})})\big),
\end{aligned}
\end{equation*}
which is precisely estimate \eqref{Euler-Split}.
\end{proof} 
\section{Proof of Theorems \ref{t:euler_concrete_intro} 
and \ref{t:split_intro}}\label{sec:pathwise}
We shall present the proof of Theorem \ref{t:euler_concrete_intro};
the proof of Theorem \ref{t:split_intro} and of the analogues of 
Theorem \ref{t:split_intro}
for classical splitting scheme and the splitting scheme with discretized noise of 
Example \ref{ex:discr-splitting} is entirely analogous. 

Set $u:=(U(t_j^{(n)}))_{j=0}^{n}$, where $U$ is the solution to \eqref{SDE} 
with initial datum $x_0\in L^p(\Omega,\F_{0},X_\eta)$ for some $p>2$ and $\eta\ge 0$ 
such that $\inv{p}<\eta<\zeta_{\max}$, with $\zeta_{\max}$ as defined in \eqref{eq:zeta-max}.\par
In order to prove Theorem
\ref{t:euler_concrete_intro}, we shall in fact consider $(V^{(n)}_j)_{j=0}^{n}$ being defined by
the abstract scheme of Section \ref{sec:absEuler}, where $\Enj{}$ is defined in terms of a family of
measures $(\mu_n)_{n\geq \N}$ satisfying \MII{}, \MI{}, and \MIII{}. We shall also assume that
$V_0=y_0\in L^p(\Omega,\calF_0;X)$. This more general case does not require extra arguments and
Theorem \ref{t:euler_concrete_intro} follows immediately by Example \ref{ex:Euler} and by taking
$x_0=y_0$.\par
The proofs of Theorems \ref{t:euler_concrete_intro} and \ref{t:split_intro}
are based on the following version of Kolmogorov's continuity criterion 
(see, e.g., \cite[Theorem I.2.1]{RevYor}):
\begin{proposition}[Kolmogorov's continuity criterion] Let $Y$ be a Banach space.
For all $\alpha>0$, $q\in (\inv{\alpha},\infty)$ and $0\le\gamma<\alpha-\inv{q}$ 
there exists a constant $K$ such that for all $T>0$, $k\in \N,$ and 
$u,v\in c_{\alpha}^{(2^k)}([0,T];L^q(\Omega;Y))$ we have:
\begin{align*}
\big(\E  \n u - v \n_{c_{\gamma}^{(2^k)}([0,T];Y)}^q\big)^{\inv{q}} 
& \leq K \n u-v \n_{c_{\alpha}^{(2^k)}([0,T];L^q(\Omega;y))}. 
\end{align*}
\end{proposition}
\begin{proof}[Proof of Theorem \ref{t:euler_concrete_intro}]
Let $T>0$ and $n\in \N$ be fixed.
Upon replacing $\eta$ by a smaller value, we may assume that 
$\gamma+\delta+\tinv{p} < \eta< \zeta_{\max}$ with $\zeta_{\max}$ as in \eqref{eq:zeta-max}. 

Let $k\in \N$ be such that $2^{k-1}< n \leq 2^k$. Then $T \leq \tfrac{2^k T}{n} < 2T$. 
For $j\in \{0,\ldots,2n\}$ set $$d^{(n)}_j:=U(t_j^{(n)})-V^{(n)}_j,$$ 
using that there exists a 
unique mild solution $U$ to \eqref{SEE} on $[0,2T]$; the definition of 
$V^{(n)}_j$ for $j\in \{n+1,\ldots,2n\}$ is straightforward.\par

By Theorem \ref{t:euler} applied to the interval
$[0,2T]$ with $2n$ time steps and with $\eta$ as fixed above we have, because 
$|t_j^{(n)}-t_i^{(n)}|\geq \frac{T}{n}$,
\begin{align*}
 \sup_{0\le i < j\le 2n} \frac{\n d^{(n)}_j- d^{(n)}_i\n_{L^p(\Omega;X)}}{|t_j^{(n)}
-t_i^{(n)}|^{\eta-\delta}}  
& \leq \big(\tfrac{n}{T}\big)^{\eta-\delta}\!\!\sup_{0\le i < j\le n}\! 
\big(\n d^{(n)}_j\n_{L^p(\Omega;X)} 
+ \n d^{(n)}_i \n_{L^p(\Omega;X)}\big)\\
& \lesssim n^{\eta-\delta}n^{-\eta}(1+\n
x_0\n_{L^p(\Omega;X_{\eta})})\\
& = n^{-\delta}(1+\n x_0\n_{L^p(\Omega;X_{\eta})}),
\end{align*}
with implied constant independent of $n$ and $x_0$. In particular:
\begin{align*}
\sup_{0\le j\le 2n} \n d^{(n)}_j \n_{L^p(\Omega;X)} & \lesssim \n d^{(n)}_0\n_{L^p(\Omega;X)} 
+ n^{-\delta}(1+\n x_0\n_{L^p(\Omega;X_{\eta})})\\
& = \n x_0 - y_0 \n_{L^p(\Omega;X)} + n^{-\delta}(1+\n x_0\n_{L^p(\Omega;X_{\eta})}).
\end{align*}\par
It follows that
\begin{align*}
\big\n \big(d^{(n)}_j\big)_{j=0}^{2^k}
\n_{c_{\eta-\delta}^{(2^k)}([0,t_{2^k}^{(n)}];L^p(\Omega;X))}
 & \lesssim \n x_0 - y_0 \n_{L^p(\Omega;X)} +n^{-\delta}(1+\n x_0\n_{L^p(\Omega;X_{\eta_{\max}})}).
\end{align*}
Thus, by Kolmogorov's criterion, using that $\eta-\delta>\gamma+\inv{p}$, and the fact 
that $T \leq 
\tfrac{2^kT}{n} < 2T$;
\begin{align*}
\big(\E\n u -v^{(n)} \n_{c_{\gamma}^{(n)}([0,T];X)}^p\big)^{\inv{p}} & 
= \big(\E\big\n \big(d^{(n)}_j\big)_{j=0}^{n} \n_{c_{\gamma}^{(n)}([0,T];X)}^p\big)^{\inv{p}}\\
& \leq \big(\E\big\n \big(d^{(n)}_j\big)_{j=0}^{2^k} 
\n_{c_{\gamma}^{(2^k)}([0,t_{2^k}^{(n)}];X)}^p\big)^{\inv{p}}\\
& \lesssim \big\n \big(d^{(n)}_j\big)_{j=0}^{2^k} 
\n_{c_{\eta-\delta}^{(2^k)}([0,t_{2^k}^{(n)}];L^p(\Omega;X))}\\
& \lesssim \n x_0 - y_0 \n_{L^p(\Omega;X)} + n^{-\delta}(1+\n x_0\n_{L^p(\Omega;X_{\eta})}).
\end{align*}
Theorem \ref{t:euler_concrete_intro} now follows from the fact that in there we assumed $y_0=x_0$.
\end{proof}

The corollary below is obtained from Theorem \ref{t:euler_concrete_intro} by a Borel-Cantelli
argument. An analogous corollary may be derived for the general abstract discretization schemes of
Section \ref{sec:absEuler}, as well as for the modified and classical splitting schemes
(under the condition $\gamma+\delta
+\frac{2}{p} < \min\{\eta_{\max},\eta,1\}$), assuming that $\theta_F,\theta_G\ge 0$ 
in the case of the latter.
\begin{corollary}\label{cor:pathwiseAS}
Let $\g,\d\ge 0$, $\eta>0$, and $p \in [2,\infty)$ 
be such that $\gamma+\delta + \frac{2}{p} 
< \min\{\zeta_{\max}, \eta\}$.
Suppose that $x_0=y_0\in L^p(\Omega,\F_0;X_{\eta})$. 
Then there exists a random variable $\chi\in L^{0}(\Omega)$, 
independent of $x_0$ and $n$ such that:
\begin{align*}
\n u - v^{(n)} \n_{c_{\gamma}^{(n)}([0,T];X)} & \leq \chi 
n^{-\delta}(1+\n x_0\n_{L^p(\Omega;X_{\eta})}).
\end{align*}
\end{corollary}
\begin{proof}
We may assume that $\gamma+\delta + \frac{2}{p} 
< \eta<\min\{\zeta_{\max},1\}$. 
Pick $\overline{\delta}>0$ such that $\delta+\inv{p}<\ov{\delta} < \zeta_{\max}-\gamma-\inv{p}$. 
By Chebychev's inequality and Theorem \ref{t:euler_concrete_intro} (with $\ov{\delta}$ 
instead of $\delta$)
we have $\P(\Omega_n)\lesssim n^{-(\ov{\delta}-\delta)p}$, where
\begin{align*}
\Omega_{n}& :=\Big\{\omega\in\Omega\, :\, \n u(\omega) - v^{(n)}(\omega) 
\n_{c_{\gamma}^{(n)}([0,T];X)} >  n^{-\delta }(1+\n x_0\n_{L^p(\Omega;X_{\eta})}) \Big\}
\end{align*}
By assumption we have $\ov{\delta}-\d >\frac1p$, and therefore
$\sum_n \P(\Omega_n) < \infty$. 
The corollary now follows from the  Borel-Cantelli lemma.
\end{proof}\par
In particular, if  $\g,\d\ge 0$ satisfy $\gamma+\delta < \min\{\zeta_{\max}, \eta\}$ for the
discretization method of Section \ref{sec:absEuler} ($\gamma+\delta < \min\{\eta_{\max},\eta,1\}$ 
for the splitting methods), then for initial values $x_0 = y_0\in L^{\infty}(\Omega,\calF_0;X_\eta)$
we obtain:
\begin{align*}
\n u - v^{(n)} \n_{c_{\gamma}^{(n)}([0,T];X)} & \leq \chi 
n^{-\delta}(1+\n x_0\n_{L^{\infty}(\Omega;X_\eta)}).
\end{align*}

\begin{remark}
It is clear from the proof of Theorem \ref{t:euler_concrete_intro} and Corollary
\ref{cor:pathwiseAS} that the assertions remain valid if $V^{(n)}$ starts
from an initial value $y_0^{(n)}\in L^p(\O,\F_0;X)$, provided that for all $n\in\N$ we have 
$\n x_0 - y_0^{(n)} \n_{L^p(\Omega;X)} \lesssim n^{-\delta}
$.  
\end{remark}

\begin{remark}\label{r:pathwise}
At the cost of some extra work, for the approximations obtained 
by either the modified or the classical splitting scheme 
it is possible to obtain pathwise convergence over $[0,T]$ instead of over the grid points, 
i.e.,\ convergence in $L^p(\Omega;C([0,T];X))$ and in $C([0,T];X)$ almost surely. 
The details are presented in \cite{cox:thesis}.
\end{remark}
\section{The local case}\label{sec:local}
The pathwise convergence result of Corollary \ref{cor:pathwiseAS} remains valid 
if $F$ and $G$ are merely locally Lipschitz and satisfy linear growth conditions. 
We shall demonstrate how this extension is obtained for the Euler scheme (or, more 
generally, the abstract time discretizations of Section \ref{sec:absEuler}). Analogous 
results may be obtained for modified and classical splitting scheme, but in order to 
do so one needs the extra regularity results mentioned in Remark \ref{r:pathwise}.\par
Consider \eqref{SEE} in a \textsc{umd} Banach space $X$ with property $(\alpha)$,
 where $A$ satisfies \MA{} and $F$ and $G$ satisfy:
\begin{itemize}
\item[\MFloc{}] For some $\theta_F>  -1 + (\inv{\tau}-\frac12)$, where $\tau$ is
the type of $X$, the function 
$F:[0,T]\times X\rightarrow X_{\theta_F}$ is measurable in the sense that for
all $x\in X$ the mapping $F(\cdot,x):[0,T]\rightarrow X_{\theta_F}$ is strongly
measurable. Moreover, $F$ is locally Lipschitz continuous and uniformly of
linear growth in its second variable.
That is to say, 
there exist constants $C_{0,m}$, $m\in \N$, and $C_1$ 
such that for all $t\in [0,T]$, all $m\in \N$ and all $x_1,x_2,x\in X$ such that 
$\n x_1\n,\n x_2\n
\leq m$:
\begin{align*}
\n F(t,x_1) - F(t,x_2) \n_{X_{\theta_F}} & \leq C_{0,m} \n x_1-x_2\n_{X}, \\ 
\n F(t,x)\n_{X_{\theta_F}} &\leq 
C_1(1+\n x\n_{X}).
\end{align*}

\item[\MGloc{}] For some $\theta_G>-\inv{2}$, the function $G : [0,T]\times
X\rightarrow \calL(H,X_{\theta_G})$ 
is measurable in the sense that for all $h\in H$ and $x\in X$ the mapping
$G(\cdot,x)h:[0,T]\rightarrow X_{\theta_G}$ is strongly measurable. Moreover,
$G$ is locally $L^2_{\gamma}$-Lipschitz continuous and uniformly of linear
growth in its second variable. That is to say, 
there exist constants  
$K_{G,m}$, $m\in 
N$, and $C_1$ such that for 
all $\alpha\in [0,\inv{2})$, all $t\in [0,T]$, all $m\in \N$ and all simple functions
$\phi_1$, $\phi_2$, $\phi: [0,T]\to X$ such that $\n \phi_1 \n_{L^{\infty}(0,T;X)}, \n \phi_2
\n_{L^{\infty}(0,T;X)}, \n \phi \n_{L^{\infty}(0,T;X)}\leq m$ one has:
\begin{align*}
& \qquad \n s\mapsto (t-s)^{-\alpha}[G(s,\phi_1(s)) -G(s,\phi_2(s))]
\n_{\gamma(0,t;H,X_{\theta_G})} \\
& \qquad \qquad \qquad \le  C_{0,m} \n s\mapsto (t-s)^{-\alpha}[\phi_1 -\phi_2] 
\n_{L^2(0,t;X)\,\cap \,\gamma(0,t;X)}; \\
& \qquad \n s\mapsto (t-s)^{-\alpha}
G(s,\phi(s))\n_{\gamma(0,t;H,X_{\theta_G})}\\
& \qquad \qquad \qquad  \leq C_1\big(1+ \sup_{0\leq t\leq T} \n s\mapsto
(t-s)^{-\alpha}\phi(s)\n_{L^2(0,t;X)\,\cap\, \gamma(0,t;X)}\big).
\end{align*}
\end{itemize}

Following \cite[Section 7]{NVW08}, for $T>0$, $p\in [1,\infty)$ and $\alpha \in [0,\inv{2})$ we
define $V^{\alpha,0}_{p}([0,T]\times\Omega; X)$ to be the space of continuous adapted
processes $\Phi:[0,T]\times\Omega\rightarrow X$ such that almost surely,
\begin{align*}
\n \Phi\n_{C([0,T];X} + \Big( \int_{0}^{T} \n s\mapsto
(t-s)^{-\alpha}\Phi(s)\n_{\gamma(0,t;X)}^p\,dt\Big)^{\inv{p}} < \infty. 
\end{align*}\par

It has been proven in \cite{NVW08} that if one assumes \MFloc{} and \MGloc{} 
instead of \MF{} and \MG{}, and moreover assumes that $x_0\in L^{0}(\Omega,\mathcal{F}_0,X)$, then
for every $p>2$ and $\alpha\in [0,\inv{2})$ satisfying $\inv{p}<\alpha+\theta_G$ equation
\eqref{SDE} has a unique mild solution on in $V^{\alpha,0}_{p}([0,T]\times \Omega;X)$ for all $T>0$.
The solution is constructed by approximation; uniqueness is proven separately.\par
The approximations are obtained as follows. For $m\in \N$ we define 
$F_{m}(t,x):=F(t,(1\minsym \frac{m}{\n x\n})x)$ and $G_{m}(t,x):=G(t,(1\minsym \frac{m}{\n x\n})x)$. 
Clearly $F_{m}$ and $G_{m}$ satisfy \MF{} and \MG{}. By \cite[Theorem 6.2]{NVW08}, 
for all $p\in (2,\infty)$ and $\alpha\in [0,\inv{2})$ satisfying $\inv{p}<\alpha+\theta_G$ there
exists a $U_{m}\in V_{c}^{\alpha,p}([0,T]\times\Omega;X)$ that is a mild solution to:
\begin{equation}\label{SEEm}
\left\{ \begin{aligned} dU_{m}(t) & = AU_{m}(t)\,dt + F_{m}(t,U_{m}(t))\,dt \\
& \qquad + G_{m}(t,U_{m}(t))\,dW_H(t);\quad t\in
[0,T],\\
U_{m}(0)& = 1_{\{\n x_0\n \le m\}} x_0  \end{aligned}\right.
\end{equation}
Note that by uniqueness this solution corresponds to the solution given by Theorem
\eqref{thm:NVW08}.

Fix $T>0$ and set 
\begin{align*}
\tau_{m}^T(\omega) := \inf\{t\geq 0\, :\, \n U_{m}(t,\omega)\n_X\geq m\},
\end{align*}
with the convention that $\inf(\emptyset)=T$.
By \cite[Section 8]{NVW08} we have, due to the linear growth conditions on $F$ and $G$, that 
$$\lim_{m\rightarrow \infty} \tau_{m}^T = T \quad\textrm{almost surely.}$$ 
In fact, due to the fact that this holds for arbitrary $T>0$, there exists a set 
$\Omega_0\subseteq \Omega$ of measure one such that for all $\omega\in\Omega_0$ there exists an 
$m_{\omega}$ such that $\tau_{m}^T(\omega)=T$ for all $m\geq m_{\omega}$. Moreover, by a 
uniqueness argument one may show that for $m_1\leq m_2$ one has $U_{m_1}(t)=U_{m_2}(t)$ 
on $[0,\tau_{m_1}^T]$.\par

The mild solution $U$ to \eqref{SEE} with $F$ and $G$ satisfying \MFloc{} and \MGloc{} is 
defined by setting:
\begin{align*}
U(t,\omega) := \lim_{m\rightarrow \infty} U_{m}(t,\omega), \quad t\in [0,T], \ \omega\in \Omega_0,
\end{align*}
and $U(t,\omega):=0$ for $t\in [0,T]$ and $\omega \in \Omega\setminus \Omega_0$.
\par
Fix $x_0\in L^{0}(\Omega;X_\eta)$ for some $\eta>0$. Let $\g,\d\ge 0$ be such that 
$\gamma+\delta< \min\{\zeta_{\max}, \eta\}$, with $\zeta_{\max}$ as defined by \eqref{eq:zeta-max}. 
Let $R>0$ be such that $[R/T,\infty)\subseteq \rho(A)$. Fix $n\in \N$, $n\geq R$. 
Let $(V^{(n)})_{j=0}^{n}$, 
$(V^{(n)}_{m,j})_{j=0}^{n}$ be defined by Euler approximations to the equations \eqref{SEE} and 
\eqref{SEEm} with initial datum $x_0$ and $1_{\{\n x_0\n \le m\}} x_0$, 
both with step size $\frac{T}{n}$.\par
Set $u:=(U(t_j^{(n)}))_{j=0}^{n}$ and $u_{m}=(U_{m}(t_j^{(n)}))_{j=0}^{n}$,  
$v^{(n)}:=(V^{(n)}_j)_{j=0}^{n}$ and $v^{(n)}_{m}:=(V^{(n)}_{m,j})_{j=0}^{n}$. 
Fix $\omega\in \Omega_0$. Let $m_{\omega}$ be such that $\tau_{m}^T(\omega)=T$ for all 
$m\geq m_{\omega}$. Note that for all $m\geq m_{\omega}$: $$\n u_{m}(\omega)\n_{\ell_n^{\infty}(X)} 
= \n u(\omega)\n_{\ell_n^{\infty}(X)}\leq m_{\omega}.$$ 
However, a priori this does not guarantee that $v^{(n)}(\omega)=v_{m}^{(n)}(\omega)$ 
for $m\geq m_{\omega}$.\par

By Corollary \ref{cor:pathwiseAS} (with
$p>2$ such that $\gamma+\delta+\frac{2}{p}< 
\min\{\zeta_{\max}, \eta\}$) there exists a constant $C_{\omega}$ depending on $\omega$ 
and $m_{\omega}$, but independent of $n$, such that:
\begin{align*}
\n u_{2m_{\omega}}(\omega) - v_{2m_{\omega}}^{(n)}(\omega) \n_{c^{(n)}_{\gamma}([0,T];X)} 
& \leq C_{\omega}
n^{-\delta}(1+\n 1_{\{\n x_0\n \le 2m_{\omega}\}} x_0\n_{L^p(\Omega,X)})\\
& \leq
C_{\omega}n^{-\delta}(1+2m_{\omega}).
\end{align*}
In particular, 
for large enough $n\ge R$, say  $n\ge N_{\omega}$,
we have:
\begin{align*}
\n u_{2m_{\omega}}(\omega) - v_{2m_{\omega}}^{(n)}(\omega) \n_{c^{(n)}_{\gamma}([0,T];X)} 
& \leq m_{\omega}.
\end{align*}
Thus for $n\geq N_{\omega}$ we have $\n v_{2m_{\omega}}^{(n)}(\omega) \n_{\ell^{\infty}(X)} 
\leq 2m_{\omega}$, whence by definition of $F_{2m_{\omega}}$ and $G_{2m_{\omega}}$ we have
$v_{2m_{\omega}}^{(n)}(\omega)=v^{(n)}(\omega)$.\par

In conclusion we have, for $n\geq N_{\omega}$:
\begin{align*}
\n u(\omega) - v^{(n)}(\omega) \n_{c^{(n)}_{\gamma}([0,T];X)} & 
= \n u_{2m_{\omega}}(\omega) - v_{2m_{\omega}}^{(n)}(\omega) \n_{c^{(n)}_{\gamma}([0,T];X)} \\
& \leq C_{\omega}
n^{-\delta}.
\end{align*} 
It follows that there exists a random variable $\xi$ such that for all $n\geq R$ we have:
\begin{align*}
\n u(\omega) - v^{(n)}(\omega) \n_{c^{(n)}_{\gamma}([0,T];X)} & \leq \tilde{C}_{\omega}
n^{-\delta}.
\end{align*}
\begin{remark}\label{r:split_local}
If we desire results for the splitting scheme then we are faced with the difficulty that 
$\n u_{2m_{\omega}}^{(n)}(\omega) \n_{\ell^{\infty}(X)} \leq 2m_{\omega}$ does not imply 
$u_{2m_{\omega}}^{(n)}(\omega)=u^{(n)}(\omega)$. Instead, we need that $\n t\mapsto  
U^{(n)}_{2m_{\omega}}(t,\omega)\n_{C([0,T];X} \leq 2m_{\omega}$, which may be obtained 
from the result announced in Remark \ref{r:pathwise}.
\end{remark}
\section{Example: heat equation in one space dimension}\label{sec:example}
On the unit interval $[0,1]$ we consider the following stochastic heat equation:  
\begin{equation}\label{heateq}
\begin{aligned}
\left\{ 
\begin{array}{rcll}
\displaystyle{\frac{\partial u}{\partial t}(\xi,t)} & \!\!\!\!=\!\!\!\! & 
\displaystyle{a_2(\xi)\frac{\partial^2 u}{\partial \xi^2}(\xi,t)
+ a_1(\xi)\frac{\partial u}{\partial \xi}(\xi,t)}
\\
&&\displaystyle{+f(t,\xi,u(\xi,t))+g(t,\xi,u(\xi,t))\frac{\partial w}{\partial t}(\xi,t)}; 
& \!\!\! \xi\in (0,1), \ t\in (0,T],\\
\displaystyle{b_{1,\xi}\frac{\partial u}{\partial \xi}(\xi,t)} & \!\!\!\!
=\!\!\!\! &- b_{0,\xi}u(\xi,t); & \!\!\! \xi\in \{0,1\},\  t\in (0,T],\\
\\
u(0,\xi)  & \!\!\!\!=\!\!\!\ &u_0(\xi); & \!\!\! \xi\in [0,1],
\end{array}
\right.
\end{aligned}
\end{equation}
Here $w$ denotes a space-time white noise on $[0,T]\times[0,1]$
and $b_{i,\xi}$ are real numbers $(i,\xi\in \{0,1\})$.
It is assumed that 
$a_2\in C[0,1]$ is bounded away from $0$ and $a_1\in C[0,1]$. 
Moreover, $f:[0,T]\times[0,1]\times\R \rightarrow \R$ and 
$g:[0,T]\times[0,1]\times\R \rightarrow \R$ 
are jointly measurable and globally Lipschitz in
the second variable, uniformly in the first variable. More precisely, 
there exist constants $L_f$ 
and $L_g$ such that for all
$t\in[0,T]$, $\xi\in[0,1]$, and $x,y\in \R$ we have:
\begin{align*}
|f(t,\xi,x)-f(t,\xi,y)|& \leq L_f|x-y|, & |g(t,\xi,x)-g(t,\xi,y)|& \leq L_g|x-y|.
\end{align*}
We also impose the linear growth conditions
\begin{align*}
|f(t,\xi,x)| \le C(1+|x|), \qquad  |g(t,\xi,x)| \le C(1+|x|),
\end{align*}
with constant $C$ independent of $\xi\in [0,1]$,  $t\in[0,T]$, $x\in \R$.

Following the approach of \cite[Section 10]{NVW08} we may rewrite this equation to fit in 
the functional-analytic
framework of Section \ref{ss:setting}. 
For $\theta>0$ and $1<q<\infty$ we define
$ H^{\theta,q}_{B} := H^{\theta,q}(0,1)$
for $0<\theta<1+\tfrac1{q},$
and
$$ H^{\theta,q}_{B}:= 
\big\{ u\in H^{\theta,q}(0,1)\colon b_{1,\xi}\tfrac{\partial u}{\partial
\xi}(\xi,t)+b_{0,\xi}u(\xi,t)=0; \ \xi\in \{0,1\}\big\}$$
for $ 1+\tinv{q}<\theta<\infty.$

The operator $A:H_B^{2,q}\rightarrow L^q(0,1)$ defined by 
$$A
u:=a_2\frac{\partial^2 u}{\partial
\xi^2}+a_1\frac{\partial u}{\partial \xi}$$ 
generates an analytic $C_0$-semigroup on $L^q(0,1)$, see \cite[Section 3.1]{Lun}, which is based on \cite{AgDougNi:59}.\par

From now on we take $(2,\infty)$ and fix $\beta\in (\frac1{2q},\inv{4})$. The part of $A$ in the space
$$X:= H_B^{2\beta,q} = H^{2\beta,q}(0,1)$$ 
generates an analytic $C_0$-semigroup in $X$ and $-A$ has bounded imaginary
powers in $X$ by \cite[Example 4.2.3]{Lun99}. By abuse of notation we shall denote the
operator by $A$ again. As a consequence of 
\cite[Theorem 4.2.6]{Lun99} and reiteration of the complex interpolation method,
for $\theta\in (0,1)$, $2\b+2\theta\not= 1+\frac1q$, we have, for $0<\theta<1$,
$$ X_\theta = [X,D(A)]_\theta = \big[H_B^{2\beta,q}, H_B^{2\beta+2,q}\big]_\theta
= H_B^{2\beta+2\theta,q}.$$\par 

The reason for picking the space $X$ as our state space is two-fold. Firstly, we need a certain
amount of space-regularity ($\beta>\frac1{2q}$) for
proving that the Nemytskii operators $F$ and $G$ induced by $f$ and $g$ satisfy \MF{} and \MG{}. 
Secondly, as we shall see in Theorem
\ref{t:example}, there is a trade-off between the space regularity in which we 
consider convergence and 
the convergence rate: as $\beta$ increases to $\inv{4}$, the convergence 
rate decreases. 
Beyond this critical value we are no longer able to prove convergence.\par
Observe that $X$ is a \textsc{umd} space, and since we assume $q>2$ the type 
of $X$ equals $\tau=2$. 
Set $H:=L^2:= L^2(0,1)$. For $t\in [0,T]$, $u\in X$, and $h\in H$ we define the Nemytskii operators
\begin{align*}
F(t,u)(\xi)&:=f(t,\xi,u(\xi)); \\ (G(t,u)h)(\xi)&:=g(t,\xi,u(\xi))h(\xi).
\end{align*}
Set $\theta_{F}:=-\beta$ and pick $\eps>0$ sufficiently
small such that $\theta_G:=-\inv{4}-\beta-\eps>-\inv{2}$.
Under the above assumptions on $f$ and $g$, it was shown in the proof 
of \cite[Theorem 10.2]{NVW08} (here we use that $\b> \frac1{2q}$) that  
$F$ defines a mapping from $[0,T]\times X$ to $X_{\theta_F} = L^q(0,1)$ that satisfies \MF{}
and $G$ defines a mapping from $[0,T]\times X$ to $\gamma(H,X_{\theta_G}) = \g(L^2,H_B^{-\frac12-2\e,q})$ 
that satisfies \MG{}; the measurability conditions 
are satisfied due to the measurability of $f$ and $g$ 
(in the notation of \cite{NVW08} we take $E=L^q(0,1)$ and 
$\eta=\beta$, so that $E_{\eta}=X$).\par

Furthermore, the part of the $A$ in $X$ satisfies \MA{}. 
Modeling the space-time white noise as an $H$-cylindrical Brownian motion $W_H$, 
we may rewrite \eqref{heateq} as follows:
\begin{equation}\label{heatACP}
\left\{ \begin{aligned} dU(t) & = AU(t)\,dt + F(t,U(t))\,dt + G(t,U(t))\,dW_H(t);\quad
t\in [0,T],\\
U(0)&=u_0. \end{aligned}\right.
\end{equation}\par

In order to obtain convergence of the Euler scheme for $U$, we must ensure that 
\MFII{} and \MGII{} are
satisfied. This requires extra assumptions on $f$ and $g$.
Noting that $\eta_{\max} = \zeta_{\max}= \frac14-\beta-\e$, we assume that there
exists a constant $C$ such that for all
$s\in [0,1]$ and all $x\in \R$ we have:
\begin{align*}
\n t\mapsto f(t,s,x)\n_{C^{\frac14-\beta}([0,T])} & \leq C(1+|x|) ;
\end{align*}
and
\begin{align*}
\n t\mapsto g(t,s,x)\n_{C^{\frac14-\beta}([0,T])} & \leq C(1+|x|).
\end{align*}
By similar arguments that were used in \cite[Section 10]{NVW08} to prove that $F$ and $G$ 
satisfy \MF{} and \MG{}, one can use the above to prove that $F$ and $G$ satisfy 
\MFII{} and \MGII{}. 

Fix $T>0$ and $n\in \N$, let $U$ be the mild solution of 
\eqref{heatACP} on $[0,T]$ and set $u:=(U(t_j^{(n)})_{j=0}^{n}$ with $t_j^{(n)} = jT/n$.

\begin{theorem}\label{t:example}  
Let $p>4$, $q>2$, $\a>0$, $\beta\in (\inv{2q},\inv{4})$, 
 and $\g,\,\d \ge 0$
satisfy $$\beta+\g+\d+\tinv{p}< \min\{\tfrac14,\a\}.$$ 
Fix $T>0$. Let $U^{(n)}$ be defined by the modified splitting scheme 
with initial value $u_0\in
L^p(\Omega,\calF_0;H^{2\a,q}(0,1))$, and let $V^{(n)}$ be defined by the implicit Euler
scheme. Let $u^{(n)} :=(U^{(n)}(t_j^{(n)})_{j=0}^{n}$ and $v^{(n)}=(V^{(n)}_j)_{j=0}^{n}$.  
Then: 
\begin{align*}
\begin{aligned}
\big(\E \n u-u^{(n)} \n_{c_{\gamma}([0,T];H^{2\beta,q}(0,1))}^p\big)^{\inv{p}}
&\lesssim n^{-\delta}(1+\n
u_0\n_{L^p(\Omega;H^{2\a,q}(0,1))});\\
\big(\E \n u-v^{(n)} \n_{c_{\gamma}([0,T];H^{2\beta,q}(0,1))}^p\big)^{\inv{p}}
&\lesssim n^{-\delta}(1+\n
u_0\n_{L^p(\Omega;H^{2\a,q}(0,1))});
\end{aligned}
\end{align*}
with implied constant independent of $n$. 
\end{theorem}

\begin{proof}
This follows from Theorems \ref{t:split_intro} and \ref{t:euler_concrete_intro} with $X=H^{2\beta,q}$
and $\eta =\a-\beta-\inv{2q}$.
\end{proof}

By the Borel-Cantelli argument of the previous section,
almost sure convergence in $c_{\gamma}([0,T];H^{2\beta,q}(0,1))$ with rate $n^{-\d}$
holds for the same initial values under the stronger assumption 
\begin{align}\label{ass-coeff}
 \beta+\g+\d+\tfrac2p< \min\{\tfrac14,\a\}.
\end{align}\par

The Sobolev embedding theorem provides a continuous embedding 
$H^{2\b, q}(0,1)\embed C^\lambda[0,1]$ whenever $2\b>\lambda+\frac1q$. Hence, 
under assumption \eqref{ass-coeff} we
obtain that for all $\l\ge 0$ such that $\lambda + 2\g + 2\d +\frac4p+\inv{q} <
\min\{\tfrac12,2\a\}$,
almost surely we have:
\begin{align*}
\n u-u^{(n)} \n_{c_{\gamma}([0,T];C^\lambda[0,1])}
&\lesssim n^{-\delta}(1+\n u_0\n_{L^p(\Omega;H^{2\a,q}(0,1))}).
\end{align*}

Let us now take $\g=0$ and suppose that 
\begin{align}\label{ld}
\lambda + 2\d < \tfrac12.
\end{align}
Suppose $u_0\in L^{p}(\Omega,\calF_0;H^{\frac12,q}(0,1))$, i.e. we take  
$\a=\inv{4}$. 
By picking $p$ and $q$ large enough, we have $\lambda + 2\g + 2\d +\frac4p+\inv{q}
<\tfrac12=\min\{\inv{2},2\a\}$. By the above we then obtain almost sure uniform convergence (with
respect to the grid points $t_j^{(n)}$) in the space $C^\lambda[0,1]$ with rate $\delta$:
\begin{align*}
\sup_{0\le j\le n}\n u(t_j^{(n)})-u_j^{(n)}\n_{C^\lambda[0,1]}
&\lesssim n^{-\delta}(1+\n u_0\n_{L^p(\Omega;H^{\inv{2},q}(0,1))})\quad \textrm{almost surely.}
\end{align*}

\begin{remark}\label{r:optimalDG}
It is proven in \cite{DaGai:00} that the optimal convergence rate of a time discretization 
for the heat equation in one
dimension with additive space-time white noise based on $n$ equidistant time 
steps is $n^{-\inv{4}}$. 
This is under the
assumption that the noise approximation of the $n^{\textrm{th}}$ approximation is based 
only on linear combinations of
$(W_H(t_j^{(n)}))_{j=0}^{n}$. In the theorem above we obtain convergence rate 
$n^{-\inv{4}+\eps}$ for $\eps>0$
arbitrarily small by taking $\gamma=0$, $\beta$ sufficiently small and $p$, $q$ sufficiently
large.\par
In \cite{DaGai:00} the authors also provide optimal convergence rates for the heat equation 
in one
dimension with multiplicative space-time white noise, but these results concern simultaneous 
discretizations of time
and space and are therefore not applicable to our situation.
\end{remark}
\appendix
\section{Technical lemmas}\label{app:1}
Here we state and prove with two lemmas which give estimates for the $\gamma$-radonifying
norm of stochastic and deterministic integral processes.

\begin{lemma}\label{lem:gDetInt}
Let $q\in [1,\infty]$, $\inv{q}+\inv{q'}=1$, and let 
$(R,\mathcal{R},\mu)$ be a finite measure space and $(S,\mathcal{S},\nu)$ a
$\sigma$-finite measure space. Let $Y_1$ and $Y_2$ be Banach spaces,
 and suppose $\Psi_1\in L^q(R,\gamma(S;Y_1))$ and $\Psi_2\in
L^{q'}(R,\calL(Y_1,Y_2))$ such that $(r,s)\mapsto \Psi_2(r)\Psi_1(r,s)$ defines an
element of $L^1(R\times S;Y_2)$.
Then:
\begin{align*}
\Big\n s\mapsto \int_R \Psi_2(r)\Psi_1(r,s)\,d\mu(r)\Big\n_{\gamma(S;Y_2)}&\leq \n
\Psi_2 \n_{L^{q'}(S;\calL(Y_1,Y_2))}
\n \Psi_1\n_{L^q(R,\gamma(S;Y_1))} .
\end{align*}
\end{lemma}

\begin{proof}
We first consider the case $q\in [1,\infty)$. The $L^1$-assumption guarantees
that the integral on the left-hand side
exists as a Bochner integral
in $Y_2$ for $\nu$-almost all $s\in S$. By \eqref{eq:gFub} 
and the fact that $q<\infty$ we may identify $\Psi_1$ with an
element in $\gamma(S;L^q(R;Y_1))$,
 and by the H\"older inequality $\Psi_2$ induces a bounded operator from
$L^q(R;Y_1)$ to $Y_2$. 
Under these identifications, the expression inside the norm at 
left-hand side equals the operator $\Psi_2\circ \Psi_1 \in \gamma(S;Y_2)$ and the
desired estimate
is noting but the right ideal property for the $\g$-radonifying norm.

The case $q=\infty$ now follows by an approximation argument. Suppose first
$\Psi_1\in
L^1\cap L^{\infty}(R,\gamma(S;Y_1))$ and $\Psi_2\in L^1\cap
L^{\infty}(R,\calL(Y_1,Y_2))$. 
By the above we have:
\begin{align*}
& \Big\n s\mapsto \int_R \Psi_2(r)\Psi_1(r,s)\,d\mu(r)\Big\n_{\gamma(S;Y_2)}\\
& \qquad \leq \lim_{q\uparrow \infty,q'\downarrow 1}\n \Psi_2
\n_{L^{q'}(S;\calL(Y_1,Y_2))} \n
\Psi_1\n_{L^q(R,\gamma(S;Y_1))} \\
& \qquad = \n \Psi_2 \n_{L^{\infty}(S;\calL(Y_1,Y_2))} \n
\Psi_1\n_{L^1(R,\gamma(S;Y_1))}.
\end{align*}\par

The result for general $\Psi_1\in
L^{\infty}(R,\gamma(S;Y_1))$ and $\Psi_2\in L^{1}(R,\calL(Y_1,Y_2))$ follows by 
approximation.
\end{proof}

The above lemma can be applied to prove the following generalization of
\cite[Proposition 4.5]{NVW08}.
\begin{lemma}\label{lem:h1}
Let $X_1$ and $X_2$ be \textsc{umd} Banach spaces. Let 
$(R,\mathcal{R},\mu)$ be a finite measure space and $(S,\mathcal{S},\nu)$ a
$\sigma$-finite measure space.   
Let $\Phi_1:[0,T]\times \Omega \rightarrow \calL(H,X_1)$, let $\Phi_2\in
L^1(R;\calL(X_1,X_2))$, 
and let  $f\in L^{\infty}(R\times [0,T];L^2(S))$. If $\Phi_1$ is
$L^p$-stochastically integrable for some $p\in (1,\infty)$, then
\begin{align*}
& \Big\n  s\mapsto \int_{0}^{T} \int_R f(r,u)(s)\Phi_2(r)\Phi_1(u)\, d\mu(r) 
\,dW_H(u)\Big\n_{L^p(\Omega;\gamma(S;X_2))}\\
& \quad \lesssim  \esssup_{(r,u)\in R\times
[0,T]} \n f(r,u)\n_{L^2(S)} \n \Phi_2 \n_{L^1(R,\calL(X_1,X_2))}  \n \Phi_1
\n_{L^p(\Omega;\gamma(0,T;H,X_1))},
\end{align*}
with implied depending only on  $p$, $X_1$, $X_2$, provided the right-hand side
is finite.
\end{lemma}
\begin{proof}
By \cite[Corollary 2.17]{KunWei}, for almost all $s\in S$ the family
$\{T_{s,u}: \ u\in [0,T]\}$ is $\gamma$-bounded 
in $\calL(X_1,X_2)$, where
$$ T_{s,u}x = \int_R f(r,u)(s)\Phi_2(r)x\, d\mu(r). $$ 
Hence, by the $\g$-multiplier theorem (Theorem \ref{t:KW}),  for almost
all $s\in S$
the function
$u\mapsto \int_R f(r,u)(s)\Phi_2(r)\Phi_1(u)\, d\mu(r)$ belongs to
$L^p_{\calF}(\Omega;\gamma(0,T;H,X_2))$.\par

Moreover, by Theorem \ref{t:KW} in combination
with Theorem \ref{t.KaiWei} (note that \textsc{umd}
Banach spaces have non-trivial cotype) we have, for almost all $r\in R$; 
\begin{align*}
u\mapsto (s\mapsto f(r,u)(s)\Phi_1(u))\in L^p_{\calF}(\Omega;\gamma(0,T;\gamma(S,X_1))).
\end{align*}
By the stochastic Fubini theorem, the isomorphism \eqref{eq:gFub} and Lemma
\ref{lem:gDetInt} (with $q=\infty$, $Y_1= L^p(\Omega;X_1)$ and $Y_2=L^p(\Omega,X_2)$, to
$\Psi(r,s)=\int_{0}^{T} f(r,u)(s)\Phi_1(u)\,dW_H(u)$
and $\Psi_2=\Phi_2$) we have:
\begin{equation}\label{lem:h1_h1}
\begin{aligned}
& \Big\n  s\mapsto \int_{0}^{T} \int_R f(r,u)(s)\Phi_2(r)\Phi_1(u)\, d\mu(r) 
\,dW_H(u)\Big\n_{L^p(\Omega;\gamma(S;X_2))}\\
& \qquad \eqsim \Big\n  s\mapsto \int_R \Phi_2(r) \int_{0}^{T} 
f(r,u)(s)\Phi_1(u)\, 
\,dW_H(u)d\mu(r) \Big\n_{\gamma(S;L^p(\Omega;X_2))}\\
& \qquad \lesssim   \n \Phi_2 \n_{L^1(R,\calL(X_1,X_2))}  \Big\n s\mapsto
\int_{0}^{T}  f(r,u)(s)\Phi_1(u)\,dW_H(u)
\Big\n_{L^{\infty}(R,\gamma(S;L^p(\Omega;X_1)))}.
\end{aligned}
\end{equation}
By isomorphism \eqref{eq:gFub}, Theorem \ref{t:stochint}, and Theorem \ref{t:KW} in combination
with Theorem \ref{t.KaiWei} we have, for almost all $r\in R$ with
implicit constants independent of $r$:
\begin{align*}
& \Big\n s\mapsto \int_{0}^{T}  f(r,u)(s)\Phi_1(u)\,dW_H(u)
\Big\n_{\gamma(S;L^p(\Omega;X_1))} \\
& \qquad \eqsim \n u \mapsto (s\mapsto  f(r,u)(s)\Phi_1(u))
\n_{L^p(\Omega;\gamma(0,T;\gamma(S;X_1)))}\\
& \qquad \leq \esssup_{u\in [0,T]} \n f(r,u)\n_{L^2(0,T)} \n \Phi_1
\n_{L^p(\Omega;\gamma(0,T;X_1))}.
\end{align*}
The result now follows by inserting the above estimate into \eqref{lem:h1_h1}.
\end{proof}

We proceed with two lemmas on Besov embeddings. 
The proof of the first lemma is closely related to the proof of \cite[Lemma
3.1]{NVW08}.
\begin{lemma}\label{lem:wBesovEst}
Suppose $Y$ is a Banach space with type $\tau\in [1,2)$, and let $\alpha\in [0,\inv{2})$ and $q\in
(2,\infty)$ satisfy
$\inv{q}<\inv{\tau}-\alpha$. Let $\Phi\in
B^{\inv{\tau}-\inv{2}}_{q,\tau}(0,T;Y)\cap L^{\infty}(0,T;Y)$ 
and, for $t\in [0,T]$, 
define $\Phi_{\alpha,t}:(0,t) \rightarrow Y$ by
$$\Phi_{\alpha,t}(s)=(t-s)^{-\alpha}\Phi(s).$$ Then there
exists an $\eps_0>0$
such that for all $T_0\in [0,T]$: 
\begin{align}\label{wBesovEmbed}
\sup_{0\leq t\leq T_0} \n \Phi_{\alpha,t}
\n_{B_{\tau,\tau}^{\inv{\tau}-\inv{2}}(0,t;Y)} & \lesssim 
T_0^{\eps_0}\n \Phi
\n_{L^{\infty}(0,T_0;Y)\cap B^{\inv{\tau}-\inv{2}}_{q,\tau}(0,T_0;Y)}.
\end{align}
\end{lemma}
\begin{proof}
We shall in fact prove the following stronger
result, namely that there exists an $\eps_0>0$ such that for all $T_0\in [0,T]$:
\begin{align}\label{wBesovEmbedR}
\sup_{0\leq t\leq T_0} \n \Phi_{\alpha,t}
\n_{B_{\tau,\tau}^{\inv{\tau}-\inv{2}}(\R;Y)} 
& \lesssim 
T_0^{\eps_0}
\n \Phi
\n_{L^{\infty}(0,T_0;Y)\cap B^{\inv{\tau}-\inv{2}}_{q,\tau}(0,T_0;Y)}.
\end{align}
On the left-hand side above, we think of $\Phi_{\a,t}$ as being extended identically
zero outside the interval $(0,t)$.\par

Let $q'\in (1,\infty)$ be such that $\inv{q}+\inv{q'}=\inv{\tau}$. As we
assumed 
$\inv{q}<\inv{\tau}-\alpha$ it follows
that $\alpha q'<1$. Thus we can pick $\eps>0$ such that
$\eps<\min\{\frac12-\alpha,1-\alpha q'\}$.

Fix $t\in [0,T_0]$. Let $\rho\in (0,1]$ and let $0<h<\rho$ (we only consider the
case $h>0$; 
the case $h<0$ can be dealt with by observing that 
$\n T_h^\R f - f\n_{L^p(\R,Y)}=\n T_{-h}^\R f - f\n_{L^p(\R,Y)}$). First we
consider the case
that $h\leq t$. In that case we have:
\begin{align*}
&\n T_h^\R(\Phi_{\alpha,t})-\Phi_{\alpha,t} \n_{L^{\tau}(\R,Y)} \\
&\qquad \leq \n s\mapsto [(t-s-h)^{-\alpha}1_{[-h,t-h]}(s) -
(t-s)^{-\alpha}1_{[0,t-h]}(s)]\Phi(s+h)\n_{L^{\tau}(\R,Y)}\\
&\qquad \quad + \n s\mapsto
(t-s)^{-\alpha}1_{[0,t]}(s)[\Phi(s+h)1_{[0,t-h]}(s)-\Phi(s)]\n_{L^{\tau}(\R,Y)}
\\
&\qquad \leq \n s\mapsto [(t-s-h)^{-\alpha}1_{[-h,t-h]}(s) -
(t-s)^{-\alpha}1_{[0,t-h]}(s)]\n_{L^{\tau}(\R)} \n \Phi
\n_{L^{\infty}(0,T_0;Y)}\\
&\qquad \quad + \n s\mapsto (t-s)^{-\alpha}1_{[0,t]}(s)\n_{L^{q'}(\R,Y)}\n
T_h^\R(\Phi)1_{[0,t-h]}-\Phi \n_{L^q(0,t;Y)}.
\end{align*}
As $\alpha q' < 1$ we have:
\begin{align*}
\n s\mapsto (t-s)^{-\alpha}1_{[0,t]}(s)\n_{L^{q'}(\R,Y)}& \lesssim
T^{\inv{q'}-\alpha}.
\end{align*}
For $p\geq 1$ and $0\leq b\leq a$ one has $(a-b)^p\leq a^p-b^p$ and thus:
\begin{align*}
&\n s\mapsto [(t-s-h)^{-\alpha}1_{[-h,t-h]}(s) -
(t-s)^{-\alpha}1_{[0,t-h]}(s)]\n_{L^{\tau}(\R)}\\
&\qquad =\Big( \int_{-h}^{t-h} \big|(t-s-h)^{-\alpha}
- (t-s)^{-\alpha}1_{[0,t-h]}(s)\big|^{\tau}\,ds
\Big)^{\inv{\tau}} \\
&\qquad \leq \Big( \int_{-h}^{t-h} [(t-s-h)^{-\alpha\tau}
-(t-s)^{-\alpha\tau}1_{[0,t-h]}(s)] \,ds
\Big)^{\inv{\tau}}\\
&\qquad = (1-\alpha\tau)^{-\inv{\tau}} h^{\inv{\tau}-\alpha} 
\lesssim h^{\inv{\tau}-\alpha-\eps}T^{\eps},
\end{align*}
where the last inequality uses $ h\le t\le T_0$.\par

Putting together these estimates, 
\begin{align*}
& \n T_h^\R(\Phi_{\alpha,t})-\Phi_{\alpha,t} \n_{L^{\tau}(\R,Y)} 
\\ & \qquad  \le  h^{\inv{\tau}-\alpha-\eps}T_0^{\eps} \n \Phi
\n_{L^{\infty}(0,T_0;Y)} + T_0^{\inv{q'}-\alpha}\n T_h^\R(\Phi)1_{[0,t-h]}-\Phi
\n_{L^q(0,t;Y)}.
\\ & \qquad  =  h^{\inv{\tau}-\alpha-\eps}T_0^{\eps} \n \Phi
\n_{L^{\infty}(0,T_0;Y)} + T_0^{\inv{q'}-\alpha}\n T_h^I(\Phi)-\Phi \n_{L^q(0,t;Y)}.
\end{align*}

Next suppose $h>t$. In that case: 
\begin{align*}
\n T_h^\R(\Phi_{\alpha,t})-\Phi_{\alpha,t} \n_{L^{\tau}(\R,Y)} 
& = 2\n \Phi_{\alpha,t}
\n_{L^{\tau}(\R,Y)}\\
& \lesssim t^{\inv{\tau}-\alpha}\n \Phi \n_{L^{\infty}(0,T_0;Y)}
\leq h^{\inv{\tau}-\alpha-\eps}T_0^{\eps}\n \Phi \n_{L^{\infty}(0,T_0;Y)},
\end{align*}
this time using $t^{\inv{\tau}-\alpha}\le t^{\inv{\tau}-\alpha-\e}T_0^\e$ and
$t\le h$.

It follows that 
\begin{align*}
& \n \Phi_{\alpha,t} \n_{B_{\tau,\tau}^{\inv{\tau}-\inv{2}}(\R,Y)} \\
& \quad = \n \Phi_{\alpha,t}\n_{L^{\tau}(\R,Y)} + \Big(\int_{0}^{1} 
\rho^{-1+\frac{\tau}{2}}\sup_{|h|<\rho}\n
T_h^\R(\Phi_{\alpha,t})-\Phi_{\alpha,t} \n_{L^{\tau}(\R,Y)}^{\tau}
\frac{d\rho}{\rho}\Big)^{\inv{\tau}}\\
& \quad\lesssim T_0^{\inv{\tau}-\alpha} \n \Phi\n_{L^{\infty}(0,T_0;Y)} + T_0^{\frac{1}{q'}-\alpha}
\Big(\int_{0}^{1}\rho^{-1+\frac{\tau}{2}}\sup_{|h|<\rho}\n T_h^I(\Phi)-\Phi
\n_{L^{q}(0,t;Y)}^{\tau}
\frac{d\rho}{\rho}\Big)^{\inv{\tau}}\\
&\quad \qquad + T_0^{\eps} \Big(\int_{0}^{1} \rho^{\tau(\frac{1}{2} -\alpha -\e)}
\frac{d\rho}{\rho}\Big)^{\inv{\tau}}\n \Phi
\n_{L^{\infty}(0,T_0;Y)}\\
&\quad \lesssim (T_0^{\inv{\tau}-\alpha}\vee T_0^{\frac{1}{q'}-\alpha} \vee
T_0^\eps)\big( \n \Phi \n_{L^{\infty}(0,T_0;Y)} + \n \Phi
\n_{B^{\inv{\tau}-\inv{2}}_{q,\tau}(0,t;Y)}\big).
\end{align*}
This gives the result, noting that $T_0^{\alpha}\maxsym T_0^{\beta} \leq 
T_0^{\min\{\alpha,\beta\}}(1\maxsym T^{|\alpha-\beta|})$ for all $T_0\in [0,T]$.
\end{proof}\par
This lemma will be used to deduce the following estimate. 
\begin{lemma}\label{l:GgammaEst}
Let $X$ be a \textsc{umd} Banach space, $H$ a Hilbert space, and suppose $G:[0,T]\times
\Omega\rightarrow \calL(H,X)$ satisfies \MGII{} of Section \ref{sec:absEuler}. 
For all $0\le \alpha<\inv{2}$ and $\eps>0$ there for any $T>0$ there exists a 
constant $C>0$ such that for  
any $n\in \N$, and any sequence $(B_j)_{j=0}^{n}$
in $ L^p(\Omega;X)$, $p\in [2,\infty)$, we have:
\begin{align*}
& \sup_{0\leq t \leq T} \n s\mapsto
(t-s)^{-\alpha}\big[G(s,B_{\un{s}n/T}))-G(\un{s},B_{\un{s}n/T}))\big]\n_{L^p(\Omega;\gamma(0,
t;H,X_{\theta_G}))} \\
& \hskip5cm \leq Cn^{-\zeta_{\max}+\eps}\big(1+\sup_{0\leq j\leq n}\n B_j
\n_{L^p(\Omega;X)}\big).
\end{align*}
\end{lemma}

\begin{proof} Without loss of generality we may assume $\eps < \inv{2}-\alpha$. Set
$q=(\inv{\tau}-\inv{2}+\eps)^{-1}$, so that
$\inv{\tau}-\inv{2}<\inv{q}<\inv{\tau}-\alpha$. For $s\in [0,T)$ define
$$\Phi(s):=
G(s,B_{\un{s}n/T}))-G(\un{s},B_{\un{s}n/T}).$$ Note that as $p\geq 2$, the type of
$L^p(\Omega,X_{\theta_G})$ is the same as the type of $X$. By embedding \eqref{BesovEmbed}, Lemma
\ref{lem:wBesovEst} and isomorphism \eqref{eq:gFub} we have:
\begin{equation}\label{GBesov}
\begin{aligned}
& \sup_{0\leq t \leq T} \n s\mapsto
(t-s)^{-\alpha}\Phi(s)\n_{\gamma(0,t;H,L^p(\Omega;X_{\theta_G}))} \\
& \qquad \qquad \lesssim \n
\Phi\n_{B^{\inv{\tau}-\inv{2}}_{q,\tau}(0,T;L^p(\Omega;\gamma(H,X_{\theta_G})))} + \n
\Phi\n_{L^{\infty}(0,T;L^p(\Omega;\gamma(H,X_{\theta_G})))}.
\end{aligned}
\end{equation}\par

By the H\"older assumption of \MGII{} we have:
\begin{equation}\label{GinftyEst}
\begin{aligned}
\n \Phi\n_{L^{\infty}(0,T;L^p(\Omega;\gamma(H,X_{\theta_G})))} & \lesssim
n^{-\zeta_{\max}-\inv{\tau}+\inv{2}}
\big( 1 + \sup_{0\leq j\leq n}\n B_j\n_{L^{p}(\Omega;X)}\big)\\
& = n^{-\zeta_{\max}-\inv{q}+\eps}\big( 1 + \sup_{0\leq j\leq n}\n B_j\n_{L^{p}(\Omega;X)}\big).
\end{aligned}
\end{equation}\par

In order to estimate the Besov norm on the right-hand side of \eqref{GBesov} 
we fix $\rho\in (0,1)$, and let $|h|< \rho$. We have, with $I =
[0,T]$, 
\begin{equation*}
\begin{aligned}
\ & \n T_h^I \Phi(s) - \Phi(s) \n_{L^p(\Omega;\gamma(H,\theta_G))} 
\\ & \leq \left\{ 
\begin{aligned}
& \n G(s+h,B_{\un{s}n/T}) - G(s,B_{\un{s}n/T})\n_{L^p(\Omega;\g(H,X_{\theta_G}))}, &&
\un{s+h}=\un{s}, \ s+h\in [0,T],\\
& 2\n \Phi\n_{L^{\infty}(0,T;L^p(\Omega;\gamma(H,X_{\theta_G})))}, && \textrm{otherwise.}
\end{aligned}
\right.
\end{aligned}
\end{equation*}\par

For $|h|\geq \frac{T}{n}$ one never has $\un{s+h}=\un{s}$ and thus it follows
from the above and \eqref{GinftyEst} that
\begin{align*}
\n T_h^I\Phi -\Phi\n_{L^q(0,T;L^p(\Omega;\gamma(H,X_{\theta_G})))} & \lesssim
n^{-\zeta_{\max}-\inv{q}+\eps}\big( 1 + \sup_{0\leq j\leq n}\n B_j\n_{L^p(\Omega;X)}\big)\\
& \lesssim |h|^{\inv{q}}n^{-\zeta_{\max}+\eps}\big( 1 + \sup_{0\leq j\leq n}\n
B_j\n_{L^p(\Omega;X)}\big).
\end{align*}
On the other hand, for $h< \frac{T}{n}$ and $\un{s+h} = \un{s}$ we obtain, by
\MGII{}:
\begin{align*}
&  \n G(s+h,B_{\un{s+h}n/T}) - G(s,B_{\un{s}n/T})\n_{L^p(\Omega;\g(H,X_{\theta_G}))} \\
& \qquad \qquad \lesssim  |h|^{\zeta_{\rm max} + \frac1\tau-\frac12} 
(1+ \n B_{\un{s}n/T}\n_{L^p(\Omega;X)}) \\
&  \qquad \qquad \le |h|^{\inv{q}}(\tfrac{T}{n})^{\zeta_{\rm max} +
\frac1\tau-\frac12-\inv{q}} \big(1+ \sup_{0\leq j\leq n}\n B_j\n_{L^p(\Omega;X)}\big)\\
&  \qquad \qquad \lesssim |h|^{\inv{q}} n^{-\zeta_{\rm max} + \eps}\big(1+ 
\sup_{0\leq j\leq n}\n B_j\n_{L^p(\Omega;X)}\big).
\end{align*}
For $|h|<\frac{T}{n}$ observe that $|\{ s\in [0,T]: \un{s+h} \neq
\un{s}\}|=n|h|$. Thus for $|h|<\frac{T}{n}$ we have, by the above estimate and
\eqref{GinftyEst}:
\begin{align*}
& \n T_h^I\Phi -\Phi\n_{L^q(0,T;L^p(\Omega;\gamma(H,X_{\theta_G}))} \\
& \qquad \lesssim
(T-n|h|)^{\inv{q}}|h|^{\inv{q}} n^{-z\eta_{\rm max} + \eps} \big(1+ 
\sup_{0\leq j\leq n}\n B_j\n_{L^p(\Omega;X)}\big)\\
& \qquad \quad +(n|h|)^{\inv{q}}n^{-\zeta_{\max}-\inv{q}+\eps}
\big(1+ \sup_{0\leq j\leq n}\n B_j\n_{L^p(\Omega;X)}\big)\\
& \qquad \lesssim |h|^{\inv{q}} n^{-\zeta_{\rm max} + \eps} 
\big(1+ \sup_{0\leq j\leq n}\n B_j\n_{L^p(\Omega;X)}\big). 
\end{align*}\par

Collecting these estimates we find:
\begin{align*}
\sup_{|h|< \rho} \n T_h^I\Phi -\Phi\n_{L^q(0,T;L^p(\Omega;\gamma(H,X_{\theta_G}))} &
\lesssim
\rho^{\inv{q}}n^{-\zeta_{\max}+\eps}
\big( 1 + \sup_{0\leq j\leq n}\n B_j \n_{L^p(\Omega;X)}\big).
\end{align*}
Because $\inv{q}>\inv{\tau}-\inv{2}$ it follows that
\begin{align*}
& \n \Phi\n_{B^{\inv{\tau}-\inv{2}}_{q,\tau}(0,T;L^p(\Omega;\gamma(H,X_{\theta_G})))}\\
&\qquad  \lesssim \n \Phi\n_{L^q(0,T;\gamma(H,L^p(\Omega;X_{\theta_G})))} + n^{-z\eta_{\max}+\eps}
\big(1+ \sup_{0\leq j\leq n}\n B_j\n_{L^p(\Omega;X)}\big)\\
& \qquad \lesssim n^{-z\eta_{\max}+\eps}\big(1+ \sup_{0\leq j\leq n}\n B_j\n_{L^p(\Omega;X)}\big).
\end{align*}
Inserting the above and \eqref{GinftyEst} into \eqref{GBesov} gives the required
result.
\end{proof}\par

The final lemma is an elementary calculus fact.
\begin{lemma}\label{lem:L2unif}
For all $0\leq \delta,\theta<\inv{2}$ there exists a constant $C$, depending
only on $\delta$ and 
$\theta$, such that for all  $0\leq u \leq t$, all $T>0$ and all $n\in \N$:
\begin{align*}
\int_{\un{u}}^{t} (t-s)^{-2\theta}(\ov{s}-\un{u})^{-2\delta} \,ds \leq C^2
(t+\tfrac{T}{n}-u)^{1-2\delta-2\theta}.
\end{align*}
\end{lemma}
\begin{proof}
If $t-\un{u}\le\tfrac{T}{n}$, then for $s\in [\un{u},t)$ one has
$\ov{s}-\un{u}=\ov{u}-\un{u}=\tfrac{T}{n}$ so
\begin{align*}
\int_{\un{u}}^{t} (t-s)^{-2\theta}(\ov{s}-\un{u})^{-2\delta}\,ds &=
(1-2\theta)^{-1}
(\tfrac{T}{n})^{-2\delta}(t-\un{u})^{1-2\theta}.
\end{align*}
Note that $\tfrac{T}{n}\ge \inv{2}(t+\tfrac{T}{n}-u)$
 and 
$(t-\un{u})^{1-2\theta}\leq (t+\tfrac{T}{n}-u)^{1-2\theta}$. Thus: 
\begin{align*}
\int_{\un{u}}^{t} (t-s)^{-2\theta}(\ov{s}-\un{u})^{-2\delta}\,ds &\leq
2^{2\delta} (1-2\theta)^{-1}
(t+\tfrac{T}{n}-u)^{1-2\delta-2\theta}.
\end{align*}\par

On the other hand if $t-\un{u}>\frac{T}{n}$ then ${t-\un{u}} < {t+T/n-u} <
2({t-\un{u}})$.
 Moreover, 
$\ov{s}-\un{u} \geq s-\un{u},$ and the substitution
$v={(s-\un{u})}/{(t-\un{u})}$ gives:
\begin{align*}
\int_{\un{u}}^{t} (t-s)^{-2\theta}(\ov{s}-\un{u})^{-2\delta}\,ds &\leq
\int_{\un{u}}^{t}
(t-s)^{-2\theta}(s-\un{u})^{-2\delta}\,ds\\
&  \leq (t-\un{u})^{1-2\delta-2\theta} \int_{0}^{1}
(1-v)^{-2\theta}v^{-2\delta}\,dv \\
& \leq
2^{(2\delta+2\theta-1)^+}(t+\tfrac{T}{n}-u)^{1-2\delta-2\theta}\int_{0}^{1}
(1-v)^{-2\theta}v^{-2\delta}\,dv.
\end{align*}
\end{proof}

\section{Estimates for stochastic convolutions}\label{app:1b}
We shall present two estimates for stochastic convolutions.
Throughout this section, $Y$ is a \textsc{umd} Banach space and $\tau\in (1,2]$ denotes its type.
Moreover, $S$ is an analytic semigroup on $Y$.
 
Roughly speaking, Lemma \ref{lem:detConv} is contained in Step 2 of the
proof of \cite[Proposition 6.1]{NVW08}, but there the space $\Vapc([0,T]\times
\Omega;X)$ is considered (see \eqref{eq:Vapc}). For completeness we give the proof below.

\begin{lemma}\label{lem:detConv}
Let $\delta\in (-\frac{3}{2}+\inv{\tau},\infty)$, 
$\alpha\in [0,\inv{2})$, and $p\in [2,\infty)$. For all $\Phi\in
L^{\infty}(0,T;L^p(\Omega;Y_{\delta}))$, the convolution $S*\Phi$ belongs 
to $\Winfp{\alpha}([0,T]\times \Omega;Y)$, and for all
$T_0\in [0,T]$ we have:
\begin{align*}
\n S*\Phi \n_{\Winfp{\alpha}([0,T_0]\times \Omega;Y)} & \lesssim
(T_0^{1+(\delta\minsym 0)}+T_0^{\inv{2}-\alpha})\n \Phi
\n_{L^{\infty}(0,T_0;L^p(\Omega;Y_{\delta}))}.
\end{align*}
\end{lemma}
\begin{proof}
By analyticity of the semigroup (equation \eqref{analyticDiff}) we have, for $t\in [0,T_0]$:
\begin{align*}
\n (S*\Phi)(t) \n_{L^p(\Omega;Y)} & \lesssim \int_{0}^{t} (t-s)^{\delta\minsym 0}\,ds \n \Phi
\n_{L^{\infty}(0,T_0;L^p(\Omega;Y_{\delta}))}\\
& \leq T_0^{1+(\delta\minsym 0)}\n \Phi
\n_{L^{\infty}(0,T_0;L^p(\Omega;Y_{\delta}))}.
\end{align*}
Taking the supremum over $t\in [0,T_0]$ gives the estimate in $L^{\infty}(0,T_0;L^p(\Omega,Y))$.\par
It remains to prove the estimate in the weighted $\gamma$-norm. Fix $t\in [0,T_0]$. As $p\geq 2$, 
it follows that $L^p(\Omega,Y)$ has type $\tau\in [1,2]$ whenever $Y$ has type $\tau$. Moreover, 
if we interpret $A$ as an operator on $L^p(\Omega,Y)$ acting pointwise, 
then $(L^p(\Omega,Y))_{\delta}=L^p(\Omega,Y_{\delta})$. Thus by \cite[Proposition 3.5]{NVW08} 
with $E=L^p(\Omega,Y)$, $\eta=0$, and $\theta=-\delta$ we have, as $\delta>-\frac{3}{2}+\inv{\tau}$;
\begin{align*}
\n s\mapsto (t-s)^{-\alpha}(S*\Phi)(s)\n_{\gamma(0,t;L^p(\Omega,Y))} 
& \lesssim T_0^{\inv{2}-\alpha}\n \Phi\n_{L^\infty(0,T_0;L^p(\Omega;Y_{\delta}))}.
\end{align*}
Taking the supremum over $t\in [0,T_0]$ gives the desired estimate.
\end{proof}\par

We proceed with the second Lemma.
\begin{lemma}\label{lem:stochConv}
Let $\delta\in (-\inv{2},\infty)$ and $\alpha\in [0,\inv{2})$. Suppose 
$\Phi: [0,T]\times \O\to \calL(H,Y_\d)$
is strongly measurable and adapted and satisfies 
\begin{equation}\label{lem_stochConv_cond}
\sup_{0\leq t\leq T}\n s\mapsto (t-s)^{-\alpha} \Phi(s)
\n_{L^p(\Omega;\gamma(0,T;H,Y_{\delta}))}<\infty,
\end{equation}
for some $p\in (1,\infty)$.
\begin{itemize}
 \item [(i)]
If $0\le \beta < \min\{\inv{2}-\alpha,\inv{2}+\delta\}$, then there exists an $\epsilon>0$ such that
for all $T_0\in [0,T]$:
\begin{align*}
& \sup_{0\leq t\leq T_0}\n s\mapsto (t-s)^{-\alpha-\beta} 
\int_0^{s} S(s-u)\Phi(u)\, dW_H(s) \Big\n_{L^p(\Omega;\gamma(0,t;Y))}\\
& \qquad \qquad \lesssim T_0^{\epsilon}\sup_{0\leq t\leq T_0}\n s\mapsto (t-s)^{-\alpha} \Phi(s)
\n_{L^p(\Omega;\gamma(0,t;H,Y_{\delta}))}.
\end{align*}\par
\item[(ii)] If, moreover, $\alpha>-\delta$, then 
there exists an $\epsilon>0$ such that for all $T_0\in [0,T]$:
\begin{align*}
&\Big\n s\mapsto \int_{0}^{s} S(s-u)\Phi(u)\,dW_H(u)
\Big\n_{\Winfp{\alpha+\beta}([0,T_0]\times \Omega;Y)} \\
& \qquad\qquad \lesssim T_0^{\epsilon}\sup_{0\leq t\leq T_0}\n s\mapsto
(t-s)^{-\alpha} \Phi(s)
\n_{L^p(\Omega;\gamma(0,t;H,Y_{\delta}))}.
\end{align*}
\end{itemize}
\end{lemma}
\begin{proof}
Fix $t\in [0,T_0]$. Let $\epsilon>0$ be such that $\epsilon<\inv{2}-\delta^- -\beta$. Here
$\delta^-= (-\d)\vee 0$.
We
apply Lemma \ref{lem:h1} with $X_1=Y_\delta$, $X_2=Y$, $R=S= [0,t]$,
and the functions
$\Phi_1(u)=(t-u)^{-\alpha}\Phi(u)$, $\Phi_2(r) =
\frac{d}{dr}[r^{\delta^{-}+\epsilon}S(r)]$, 
and $f(r,u)(s)=
(t-s)^{-\alpha-\beta}(s-u)^{-\delta^- -\epsilon}(t-u)^{\alpha}1_{0\leq r\leq s-u}$.
By \eqref{analyticDiff} we have $\n \Phi_2(r)\n_{\calL(X_{\delta},X)} \lesssim r^{-1+\eps}$ for
$r\in [0,T]$. From the lemma it follows that:
\begin{align*}
& \Big\n s\mapsto
(t-s)^{-\alpha-\beta}\int_{0}^{s} S(s-u)\Phi(u)\,dW_H(u)\Big\n_{L^p(\Omega;\gamma(0,
t;Y))} \\
&\qquad  \lesssim t^{\frac12-\b-\delta^-}\n s\mapsto
(t-s)^{-\alpha}\Phi(s)\n_{L^p(\Omega;\gamma(0,t;H,Y_\delta))}.
\end{align*}
Taking the supremum over $t\in[0,T_0]$ we obtain (i).\par

For the estimate in $\Winfp{\alpha+\beta}$-norm it remains, by part (i), to prove the estimate in
$L^{\infty}(0,T_0;L^p(\Omega,Y_{\delta}))$. Let $\epsilon <
\min\{\alpha+\delta,\inv{2}-\delta^{-}-\beta\}$.
By Lemma \ref{lem:analyticRbound} (apply part (1) if $\d\in (-\frac12, 0]$ and part (2) 
if $\d\in [0,\infty)$) the operators $r^\a S(r)$, $r\in [0,t]$, are $\g$-bounded from $Y_\d$ to $Y$, 
with $\g$-bound at most $C t^{\a+\delta}$ with $C$ independent of $t\in [0,T]$. 
Hence, by the $\g$-multiplier theorem, for all $t\in [0,T]$,
\begin{align*}
\Big\n \int_{0}^{t} S(t-s)\Phi(s)\,dW_H(s)
\Big\n_{L^p(\Omega;Y)} \lesssim t^{\a+\delta}\n s\mapsto (t-s)^{-\alpha}
\Phi(s)
\n_{L^p(\Omega;\gamma(0,t;H,Y_{\delta}))}.
\end{align*}
The norm estimate in
$L^{\infty}(0,T;L^p(\Omega;Y_{\delta}))$ is obtained by
taking the supremum over $t\in [0,T]$.\par
\end{proof}
\section{Existence and uniqueness}\label{app:2}
The aim of this section is to outline the proof of Theorem \ref{thm:NVW08}.
The setting is always that of Section \ref{sec:prelim}.

\begin{proof}[Proof of Theorem \ref{thm:NVW08}]
Assume first that $\alpha\in [0,\inv{2})$ is so large that 
 $\alpha+\theta_G>\eta$. Let $p\in [2,\infty)$ and $T_0\in [0,T]$ be fixed. 
For $\Phi\in \Winfp{\alpha}([0,T_0]\times \Omega;X_\eta)$ define
\begin{align*}
L(\Phi)(t)&:= S(t)x_0 + \int_{0}^{t}S(t-s)F(s,\Phi(s))\,ds +
\int_{0}^{t}S(t-s)G(s,\Phi(s))\,dW_H(s).
\end{align*}\par
Copying Step 1 of the proof of \cite[Proposition 6.1]{NVW08}
without changes, and substituting Steps 2 and 3 by the Lemmas \ref{lem:detConv}
and \ref{lem:stochConv} above, we find that there exists an
$\eps_0>0$ and a $C>0$ such that $L:\Winfp{\alpha}([0,T_0]\times\Omega;X_\eta) \rightarrow
\Winfp{\alpha}([0,T_0]\times\Omega;X_\eta)$
and
\begin{align*}
\n L(\Phi)\n_{\Winfp{\alpha}([0,T_0]\times \Omega;X_\eta)}
& \leq C\n x_0 \n_{X} + CT^{\eps_0}\n
F(\cdot,\Phi(\cdot))\n_{L^{\infty}(0,T_0;L^p(\Omega;X_{\theta_F}))}
\\
& \quad+ CT^{\eps_0}\sup_{0\leq t\leq T_0}\n s\mapsto
(t-s)^{-\alpha} G(s,\Phi(s))\n_{L^p(\Omega;\gamma(0,t;X_{\theta_G}))}\\
&\leq C\n x_0 \n_{X} \\
& \quad + C (M(F) + M(G)) T_0^{\eps_0}(1+\n\Phi
\n_{\Winfp{\alpha}([0,T_0]\times\Omega;X_\eta)}),
\end{align*}
where in the last line we used \MF{} and \eqref{GLipschitzV1}. Moreover,
\begin{align*}
& \n L(\Phi_1)-L(\Phi_2)\n_{\Winfp{\alpha}([0,T_0]\times \Omega;X_\eta)}\\
& \qquad \leq C (\textrm{Lip}(F) + \textrm{Lip}_{\gamma}(G))
T_0^{\eps_0}\n \Phi_1 - \Phi_2 \n_{\Winfp{\alpha}([0,T_0]\times\Omega;X_\eta)},
\end{align*}
where in the last line we used \eqref{GLipschitzV2}.

Thus by a fixed-point argument, for sufficiently small $T_0$ 
there exists a unique process $\Phi\in \Winfp{\alpha}([0,T_0]\times\Omega;X_\eta)$
satisfying \eqref{SDE_sol} on the interval $[0,T_0]$.
By repeating this construction a finite number of times, each time
taking the final value of the previous step as the initial value of the next,  
we obtain a solution on $[0,T]$.\par

So far, we have proved existence and uniqueness under the additional assumption 
$\alpha+\theta_G>\eta$. Existence in
$\Winfp{\alpha}([0,T]\times \Omega;X_\eta)$ for arbitrary $\alpha\in [0,\inv{2})$ 
follows by from \eqref{Vchange-of-alpha}. It remains to prove uniqueness 
for arbitrary $\alpha\in [0,\inv{2})$.\par 

Let $\a\in [0,\inv{2})$ be arbitrary and let $\Phi\in \Winfp{\a}([0,T]\times \Omega;X_{\eta})$. 
Viewing $F$ as a mapping from $[0,T]\times X_\eta$ to $X_{\theta_F}$ (as $\eta\ge 0$),
we have $F(\cdot,\Phi(\cdot)) \in \Winfp{\a}([0,T]\times \Omega;X_{\theta_F})$.
Then, by Lemma \ref{lem:detConv} with $\delta=\theta_F-\eta$ and $Y = X_\eta$ 
(and $\tilde \a = \a+\b$ ),
we find $S*F(\cdot,\Phi(\cdot))\in \Winfp{\a+\beta}([0,T]\times \Omega;X_{\eta})$
for all $\beta\in [0,\inv{2}-\delta)$.\par
By Lemma \ref{lem:stochConv} (with $Y=X_{\eta}$ 
and $\delta = \theta_G-\eta$) and \eqref{GLipschitzV1} we have, 
for all $\beta\in [0,\inv{2}-\a)$ such that $\beta<\inv{2}+\theta_G-\eta$:
\begin{align*}
&\sup_{0\leq t\leq T} \Big\n s\mapsto (t-s)^{-\a-\beta}
\int_{0}^{s} S(s-u)G(u,\Phi(u))dW_H(u)\Big\n_{L^p(\Omega,\gamma(0,t;X_\eta))} \\
&\qquad \qquad \qquad \qquad \leq 
\sup _{0\leq t\leq T} \Big\n s\mapsto (t-s)^{-\a} 
G(s,\Phi(s))\n_{L^p(\Omega;\gamma(0,t;H,X_{\theta_G}))}\\
&\qquad \qquad \qquad \qquad \leq  \n \Phi\n_{\Winfp{\a}([0,T]\times \Omega;X_\eta)}.
\end{align*}
Since also $S(\cdot)x_0 \in \Winfp{\a+\b}([0,T]\times \Omega;X_\eta)$ for all $\beta\in
[0,\inv2-\alpha)$, we see that 
if $\a\in [0,\frac12)$ and $\Phi\in \Winfp{\a}([0,T]\times \Omega;X_{\eta})$ satisfies
\eqref{SDE_sol}, then 
$\Phi\in \Winfp{\a+\b}([0,T]\times \Omega;X)$ for all $\beta\in [0,\inv{2}-\a)$ 
such that $\beta<\inv{2}+\theta_G-\eta$. Repeating this argument a finite number of steps 
if necessary, 
we obtain 
that $\Phi\in \Winfp{\a+\beta}([0,T]\times \Omega;X)$ for all $\beta\in [0,\inv{2}-\a)$. 
As uniqueness of a process in $\Winfp{\a}([0,T]\times \Omega;X)$ satisfying \eqref{SDE_sol} 
has been established for $\alpha>\eta-\theta_G$ this completes the proof.
\end{proof}
\begin{remark}
Inspection of the proofs of the main theorems reveals that uniqueness is only used 
for large $\a\in [0,\frac12)$.
As a consequence, the last part of the above proof is not needed for our purposes. It has been
included for completeness reasons. 
\end{remark}

\end{document}